\documentclass[10pt,a4paper]{amsart}
\usepackage[a4paper,left=3.3cm,right=3.3cm,top=3.4cm,bottom=3.4cm]{geometry}
\usepackage{newtxtext,newtxmath}
\usepackage{xcolor}
\usepackage{graphicx}
\usepackage{caption}
\usepackage{bbm}
\usepackage{amsfonts}
\usepackage{amsthm}
\usepackage{amsmath}
\usepackage{mathrsfs}
\usepackage{amscd}
\usepackage{mathdots}
\usepackage{bm}
\usepackage{mathtools}
\usepackage[all]{xy}
\usepackage{hyperref}

\selectfont

\newcommand{\SO}{{\mathrm{SO}}}

\newcommand{\RAT}{\mathbb{Q}}

\newcommand{\R}{\mathbb {R}}

\newcommand{\Hom}{\mathrm{Hom}}
\newcommand{\Sh}{\mathrm{Sh}}

\newcommand{\GL}{\mathrm{GL}}
\newcommand{\GSp}{\mathrm{GSp}}

\DeclareMathOperator{\Spec}{\mathrm{Spec}}

\renewcommand{\min}{{\mathrm{min}}}
\renewcommand{\max}{{\mathrm{max}}}
\newcommand{\rF}{{\mathrm{F}}}
\newcommand{\rV}{{\mathrm{V}}}
\newcommand{\rg}{{\mathrm{g}}}

\newcommand{\BF}{{\mathbb{F}}}
\newcommand{\GU}{{\mathrm{GU}}}
\renewcommand{\Bar}{\overline}
\newcommand{\Fp}{{\mathbb{F}_p}}
\newcommand{\Fpbar}{{\overline{\mathbb{F}}_p}}
\newcommand{\End}{{\mathrm{End}}}
\newcommand{\EO}{{\mathrm{EO}}}
\newcommand{\kk}{{\bm{k}}}
\newcommand{\Z}{{\mathbb{Z}}}
\newcommand{\F}{\mathbb{F}}
\newcommand{\Spin}{{\mathrm{Spin}}}
\newcommand{\Q}{{\mathbb{Q}}}
\newcommand{\bR}{{\mathbb{R}}}
\newcommand{\C}{{\mathbb{C}}}
\newcommand{\Af}{{\mathbb{A}_f}}

\newcommand{\Gzips}{\ensuremath{{G\text{-}\mathsf{Zip}^{\mu}}}}

\newcommand{\GSpin}{{\mathrm{GSpin}}}
\newcommand{\diag}{\mathrm{diag}}
\newcommand{\ad}{\mathrm{ad}}

\newcommand{\Zp}{{\mathbb{Z}_p}}
\newcommand{\antidiag}{\mathrm{antidiag}}
\newcommand{\Fpsq}{{\mathbb{F}_{p^2}}}
\newcommand{\new}{{\mathrm{new}}}
\newcommand{\zip}{{\mathrm{zip}}}
\newcommand{\KS}{\mathrm{KS}}
\newcommand{\Qp}{\mathbb{Q}_p}

\providecommand{\Gm}{\mathbb{G}_m}
\newcommand{\Frob}{\mathrm{Frob}}
\newcommand{\Span}{\mathrm{Span}}
\newcommand{\cancong}{\ensuremath{{\ \displaystyle \mathop{\cong}^{\text{\tiny{can}}}}\ }}
\newcommand{\data}{\mathrm{data}}
\newcommand{\fram}{\mathrm{fram}}
\newcommand{\Zpbrev}{\Breve{\mathbb{Z}}_p}
\newcommand{\Qpbrev}{\Breve{\mathbb{Q}}_p}

\normalsize
\newtheorem{innercustomthm}{Theorem}
\newenvironment{customthm}[1]{\renewcommand\theinnercustomthm{#1}\innercustomthm}{\endinnercustomthm}

\newtheorem{theorem}{Theorem}[subsection]
\newtheorem{lemma}[theorem]{Lemma}
\newtheorem{remark}[theorem]{Remark}
\newtheorem{proposition}[theorem]{Proposition}
\newtheorem{corollary}[theorem]{Corollary}
\theoremstyle{definition}
\newtheorem{definition}[theorem]{Definition}

\theoremstyle{remark}

\theoremstyle{plain}



\author{Qijun Yan and Chao Zhang}
\address{Qijun Yan: Hetao Institute of Mathematics and Interdisciplinary Sciences
(HIMIS), Shenzhen, 518\,017, China}
\email{yanqijun@himis-sz.cn}
\address{Chao Zhang: Qingdao University, Qingdao,  266\,071,China}
\email{zhangchao1217@gmail.com}

\subjclass[2010]{14G35, 11G18}
\title[Ekedahl-Oort strata under embeddings]{Ekedahl-Oort strata under natural embeddings of orthogonal and unitary Shimura varieties}
\begin{document}
	  \begin{abstract}
    In this paper, we study the behavior of Ekedahl-Oort strata under natural embeddings between the good reductions modulo \(p\) of GSpin Shimura varieties and Rapoport-Smithling-Zhang unitary Shimura varieties, a prototypical setting for the construction of special cycles in the Kudla program. In each case, we determine the EO stratum containing the image of a given EO stratum under the embedding. We also compute discrete invariants of these Shimura varieties, including their \(p\)-ranks and \(a\)-numbers; in the GSpin case, these are obtained via the Kuga-Satake embedding.
\end{abstract}
	\maketitle
	\section*{Introduction}

\subsection{Background and main goal of the paper}

Let \(\mathcal{A}_{g,N,\mathbb{F}_p}\) be the moduli space over \(\mathbb{F}_p\) of principally polarized abelian varieties of dimension \(g\) with symplectic level \(N\)-structure, where \(p\) is a fixed odd prime and \(N \ge 3\) is an integer prime to \(p\). As a characteristic-\(p\) phenomenon, \(\mathcal{A}_{g,N,\mathbb{F}_p}\) admits several natural stratifications, including the Newton stratification, the Ekedahl--Oort (EO) stratification, the \(a\)-number stratification, and the \(p\)-rank stratification. These stratifications are defined by different invariants of abelian varieties in characteristic \(p\), and some are finer than others. For instance, the EO stratification refines both the \(p\)-rank stratification and the \(a\)-number stratification, while the Newton stratification is finer than the \(p\)-rank stratification. Several of these stratifications extend to the good reductions modulo \(p\) of more general Shimura varieties.

The EO stratification is the main focus of this paper, although some of our results also concern related strata. In the Siegel case, the EO stratum of a point is determined by the isomorphism class of the associated principally polarized truncated BT-1 \cite{EOTexel01}. More precisely, two geometric points of \(\mathcal{A}_{g,N,\mathbb{F}_p}\) lie in the same EO stratum if and only if their \(p\)-torsion group schemes are isomorphic as principally polarized BT-1's. This stratification has been generalized to broader classes of Shimura varieties; see \cite{VW13} for the PEL case, \cite{EOZhang} for the Hodge type case, and \cite{SZ22} for the abelian type case.

Now let
\(
\Sh_{H,\overline{\mathbb{F}}_p} \longrightarrow \Sh_{G,\overline{\mathbb{F}}_p}
\)
be a morphism of geometric mod \(p\) Shimura varieties arising from a morphism of Shimura data with good reduction at \(p\) (see \S 1.1). It is known that the image of an EO stratum of \(\Sh_{H,\overline{\mathbb{F}}_p}\) is contained in a single EO stratum of \(\Sh_{G,\overline{\mathbb{F}}_p}\); we call the latter the \emph{image EO stratum}. Once explicit parametrizations of EO strata on both sides are available, it is natural to ask which EO stratum of \(\Sh_{G,\overline{\mathbb{F}}_p}\) contains the image of a given EO stratum of \(\Sh_{H,\overline{\mathbb{F}}_p}\).

In this paper, we study this question for certain natural embeddings between GSpin Shimura varieties and Rapoport--Smithling--Zhang (RSZ) unitary Shimura varieties. These embeddings arise in a prototypical setting for the construction of special cycles in the Kudla program; see, for example, \cite{Kud97, KR00, KR99}. In the cases considered here, the EO stratifications admit relatively simple parametrizations. For instance, for the embedding
\[
\Sh_{\widetilde{\mathrm{GU}}(n,1),\Fpbar}
\longrightarrow
\Sh_{\widetilde{\GU}(n+1,1),\Fpbar},
\]
the EO strata on the source are parametrized by \(\{0,1,\dots,n\}\), while those on the target are parametrized by \(\{0,1,\dots,n+1\}\). Our goal is to determine, for each EO stratum on the source, its image EO stratum; see Theorem~\ref{Thm:MainThmOrthog} and Theorem~\ref{Thm:MainThmUnit} below.
    
	\subsection{Main results} 
\paragraph{\textbf{Orthogonal case}}
We begin with the orthogonal case. To state the theorem, we introduce some notation; for simplicity, in this introduction we assume that \(n \ge 2\). Let \((V,q)\) be a quadratic space over \(\mathbb{Q}\) of real signature \((n,2)\), and assume that \(V\) contains a self-dual \(\mathbb{Z}_{(p)}\)-lattice \(\Lambda\), which we fix once and for all. Let \(\mathcal{S}h_{\GSpin(\Lambda)}\) and \(\mathcal{S}h_{\SO(\Lambda)}\) denote the corresponding \(p\)-integral models of the GSpin and SO Shimura varieties, with hyperspecial level at \(p\) and sufficiently small prime-to-\(p\) level (see \S2). They are related by a finite morphism
\[
\mathcal{S}h_{\GSpin(\Lambda)} \to \mathcal{S}h_{\SO(\Lambda)}.
\]
 The former is of Hodge type, while the latter is of abelian type. As recalled in Theorem~6.3.1 below, their EO stratifications are defined via zip period maps.

To make the dependence on the signature (hence dimension) more transparent, we write
$\GSpin(n-1,2)$ and $\GSpin(n,2)$ for $\GSpin(\Lambda')$ and $\GSpin(\Lambda)$,
and similarly $\SO(n-1,2)$ and $\SO(n,2)$ for $\SO(\Lambda')$ and $\SO(\Lambda)$.
By the general theory of EO stratifications for Hodge type and abelian type Shimura varieties, the EO stratification on \(\Sh_{\GSpin(n,2),\Fpbar}\) is the pullback of that on \(\Sh_{\SO(n,2),\Fpbar}\). Therefore, for the purpose of describing EO strata, we may work with the corresponding orthogonal zip data. In the present case, the EO strata admit an explicit parametrization by a finite partially ordered set ${}^{\mu}W$, and are in fact almost determined by their dimensions. More precisely, when $n$ is odd, there is a unique EO stratum of dimension $i$ for each $0 \le i \le n$; when $n$ is even, there are two EO strata of dimension $n/2$, and a unique EO stratum in every other dimension. We refer to the main text for the precise definitions. For the purposes of this introduction, it suffices to note that ${}^{\mu}W$ is a finite poset whose order governs the closure relations of the EO strata.
	
    Let $V'\subseteq V$ be a codimension $1$ subspace such that the restriction $q'=q|_{V'}$ has signature $(n-1,2)$. Assume that $\Lambda'=\Lambda\cap V'$ is also self-dual. Then the canonical map $\Sh_{\GSpin(n-1,2),\Qp}\to \Sh_{\GSpin(n,2),\Qp}$ extends to a morphism of $p$-integral models $\mathcal{S}h_{\GSpin(n-1,2)}\hookrightarrow \mathcal{S}h_{\GSpin(n,2)}$ over $\Zp$, where the levels at $p$ are hyperspecial and the away-from-$p$ levels are chosen compatibly so that the map is a closed embedding. Reducing modulo $p$ and passing to geometric fibers, we obtain an embedding,
	\[
	\Sh_{\GSpin(n-1,2),\Fpbar}\hookrightarrow \Sh_{\GSpin(n, 2),\Fpbar}.
	\]

	\begin{customthm}{\textbf{A}}
    \label{Thm:MainThmOrthog} Let $n\geq 1$ be an integer and $0\leq i\leq n-1$. Under the embedding above, an EO stratum of dimension $i$ in $\Sh_{\GSpin(n-1, 2), \Fpbar}$ (equivalently, the corresponding EO stratum on the $\SO$ side) has image EO stratum of dimension $j$, where:
		\begin{itemize} 
			\item If $n$ is even, then $j = \begin{cases} 
				i, & i < n/2, \\
				i+1, & i \geq n/2.
			\end{cases}$
			\item If $n$ is odd, then $j = \begin{cases} 
				i, & i < (n-1)/2, \\
				i+1, & i > (n-1)/2,
			\end{cases}$\\
			and for $i = (n-1)/2$: $j = \begin{cases}
				i+1, & \text{if\, } \SO(\Lambda')_{\Fp} \text{ is split over }\Fp , \\
				i, & \text{if\, } \SO(\Lambda')_{\Fp} \text{ is nonsplit over }\Fp.
			\end{cases}$
		\end{itemize}
	\end{customthm}
    In particular, the theorem says that in case $n$ is odd, the two different EO strata of dimension $(n-1)/2$ of $\Sh_{\GSpin(n-1, 2), \Fpbar}$ have the same image EO stratum inside $\Sh_{\GSpin(n, 2), \Fpbar}$, whose dimension could be $(n-1)/2$ or $(n+1)/2$, depending on the splitting behavior of the  $\Fp$-reductive group $\SO(n-1, 2)_{\Fp}$.  We refer to \S~\ref{S:Examples} for examples of Theorem~\ref{Thm:MainThmOrthog} when $n\leq 3$, and to \S~\ref{S:DiagOrthog} for an illustrative diagram of Theorem~\ref{Thm:MainThmOrthog}. The proof is computational, relying on the functoriality of the zip period maps defining the EO strata of $\Sh_{\GSpin(n-1, 2), \Fpbar}$ and $\Sh_{\GSpin(n, 2), \Fpbar}$.
    %A key technical input is the notion of frames introduced in \cite{PWZ}, which offers a flexible framework for identifying the parametrizing poset ${}^\mu W$ order-preservingly with the underlying topological space of the orthogonal zip stack. This allows us to reduce the determination of image EO strata to explicit calculations with framed zip orbits. We also use the orbit-coincidence criteria of loc.\ cit.\ to determine when two such framed zip orbits agree; cf.\ Theorem~\ref{Thm:Criteria}.
    
    As an application of this theorem, we compute the $a$-numbers of the EO strata of $\Sh_{\GSpin(n, 2),\Fpbar}$ via the Kuga--Satake embedding, as stated in Corollary~\ref{Cor:anumber}. Moreover, in Theorem~\ref{Thm:Prank} we determine the Newton slopes and $p$-ranks of the Newton strata of $\Sh_{\GSpin(n, 2),\Fpbar}$, again via the Kuga--Satake embedding; in particular, this gives the $p$-ranks of the EO strata.

 \paragraph{\textbf{Unitary case}} We next consider the analogous question for the natural embeddings of RSZ unitary Shimura varieties $\Sh_{\widetilde{\GU}(n,1), \Fpbar}$ into $\Sh_{\widetilde{\GU}(n+1,1),\Fpbar}$; see \S 12 for details. We work with the RSZ version because $\Sh_{\mathrm{U}(n,1), \Fpbar}$ is not of PEL type, but only of abelian type, whereas the associated similitude unitary Shimura varieties $\Sh_{\GU(n,1), \Fpbar}$ and the RSZ Shimura varieties are of PEL type. Moreover, although there is a natural embedding $\mathrm{U}(n,1)\to \mathrm{U}(n+1,1)$, it does not extend to an embedding $\GU(n,1)\to \GU(n+1,1)$, and hence does not induce a map between the corresponding Shimura varieties. By contrast, the RSZ Shimura varieties do admit natural embeddings $\Sh_{\widetilde{\GU}(n,1), \Fpbar}\to \Sh_{\widetilde{\GU}(n+1,1), \Fpbar}$.  Moreover, the natural projection \(\Sh_{\widetilde{\GU}(n,1),\Fpbar}\to \Sh_{\mathrm{U}(n,1),\Fpbar}\) induces a bijection between their EO strata, compatible with closure relations. In the unitary case, we obtain results similar to, but slightly simpler than, those in the orthogonal case.
	
	\begin{customthm}{\textbf{B}}\label{Thm:MainThmUnit}
	    Let $n\geq 0$ be an integer and $0\leq i\leq n$. Under the embedding  $\Sh_{\widetilde{\GU}(n, 1), \Fpbar} \to \Sh_{\widetilde{\GU}(n+1, 1), \Fpbar}$, an EO stratum of dimension $i$ in $\Sh_{\widetilde{\GU}(n, 1), \Fpbar}$ has image EO stratum of dimension $j$, where
        \[
j=
\begin{cases}
i+1, & \text{if } p \text{ splits in } \kk,\\
i,   & \text{if } p \text{ is inert in } \kk \text{ and } i \le \frac{n}{2},\\
i+1, & \text{if } p \text{ is inert in } \kk \text{ and } i > \frac{n}{2}.
\end{cases}
\]
	\end{customthm}

    Thanks to the moduli interpretation, the proof of Theorem~\ref{Thm:MainThmUnit} is simpler than that of its orthogonal analogue, Theorem~\ref{Thm:MainThmOrthog}. When $p$ splits in $\kk$, the Shimura varieties $\Sh_{\GU(n,1), \Fpbar}$ and $\Sh_{\widetilde{\GU}(n,1), \Fpbar}$ are of \emph{fake} unitary type, so their EO, Newton, and $p$-rank stratifications coincide; in this case, the image EO strata can be determined from the $p$-rank of each stratum. When $p$ is inert in $\kk$, the EO and Newton stratifications for unitary Shimura varieties of good reduction and signature $(n,1)$ were studied in \cite{BW06}; see also \cite{VW11}. In both cases, the moduli interpretation together with the classification of BT-$1$'s with additional structures in \cite{MoonenWeyl} allows one to describe the relevant EO strata. We present two slightly different proofs of this theorem, both via the moduli interpretation, and we also record the $p$-ranks and $a$-numbers of $\Sh_{\widetilde{\GU}(n,1), \Fpbar}$; see Lemmas~\ref{Lem:PrankUnit} and~\ref{Lem:AnumberUnit}.

\subsection{Structure of the paper}
Sections~\ref{S:preliminary} and \ref{S:ShvGSpinSO} provide general preliminaries on Shimura varieties of Hodge and abelian type, with particular emphasis on the orthogonal and \(\GSpin\) cases, as well as on the \(G\)-zip stacks that serve as period domains for the zip period maps defining the EO strata.

The orthogonal case is developed in \S\S~\ref{S:RedPairNorm}--\ref{S:Examples}. Sections~\ref{S:RedPairNorm}--\ref{S:Frames} supply the zip-theoretic and combinatorial preparation, especially the normalization of the relevant reductive pairs and the choice of frames for the zip group data attached to the \(G\)-zip stacks whose underlying topological spaces parametrize the EO strata of \(\Sh_{\GSpin(n, 2), \Fpbar}\). In \S~\ref{S:EOOrthog}, we establish the order-preserving bijection between \({}^\mu W\) and the underlying topological space of the corresponding \emph{normalized} zip stack, and apply the general criteria of \cite{PWZ} in our special orthogonal setting to determine when two framed zip orbits coincide. We then prove Theorem~\ref{Thm:MainThmOrthog} in \S~\ref{EOunderKudlaEmbed} and compute the \(a\)-numbers of the EO strata of \(\Sh_{\GSpin(n, 2), \Fpbar}\); \S~\ref{S:ZerodimKS} treats the analogous question for the zero-dimensional Kuga--Satake embedding, while Sections~\ref{S:NewtonStra} and \ref{S:PRank} are devoted to the Newton stratification of \(\Sh_{\GSpin(n, 2), \Fpbar}\) and the computation of the corresponding \(p\)-ranks. Finally, \S~\ref{S:Examples} presents examples of Theorem~\ref{Thm:MainThmOrthog} in small ranks, together with an illustrative diagram.

The unitary case is developed in \S\S~\ref{S:RSZVar}--\ref{S:UnitDiagram}. In \S~\ref{S:RSZVar}, we introduce the RSZ unitary Shimura varieties, their moduli interpretations, and the natural embeddings between them, and in \S~\ref{S:EOUnitShv} we formulate and prove Theorem~\ref{Thm:MainThmUnit}. The \(p\)-ranks and \(a\)-numbers of EO strata for unitary Shimura varieties are computed in \S\S~\ref{S:PrankUnitShv} and \ref{S:AnumberUnitShiv}, and \S~\ref{S:UnitDiagram} concludes this part with an illustrative diagram for Theorem~\ref{Thm:MainThmUnit}.

	\subsection{Notations and conventions}
	\begin{enumerate}
		\item In this paper we focus on Shimura varieties with \emph{good} reduction at $p$ and on morphisms between them; thus the level at $p$ is always hyperspecial. The EO stratification on the special fiber is independent of the prime-to-$p$ level: for any Shimura variety attached to a group $G$ and fixed hyperspecial $\mathsf{K}_p$, the natural change-of-level map
		\[
		\Sh_{G,\Fpbar}(\mathsf{K}_p\mathsf{K}^p)\longrightarrow \Sh_{G,\Fpbar}(\mathsf{K}_p\mathsf{K}^{\prime p})
		\]
		pulls back EO strata. Accordingly, when discussing $p$-integral models or their reductions we often suppress the level in the notation. For example, a map $\Sh_{G_1,\Fp}\to \Sh_{G_2,\Fp}$ will be understood from context. 
        \item We use \(\mathcal{S}h\) for \emph{integral models} (global or local) of Shimura varieties, and \(\Sh\), often with a subscript such as \(\Fpbar\), for their \emph{fibers}, to reflect the fact that the former have dimension one greater than the latter. 
		\item  Let \((V,q)\) be a quadratic space over a field and \(\delta_1,\dots,\delta_n\) a fixed basis of $V$. The notation \(q=\sum_{i,j} a_{ij} x_i x_j\) means that \(q(x)=\sum_{i,j} a_{ij} x_i x_j\), for all \(x=\sum_i x_i\delta_i\).
		
		\item For a matrix $A = (a_{i,j})$ with entries in a ring of characteristic $p$, we write
		$\Frob(A)=A^{(p)}$ for the matrix $(a_{i,j}^p)$ obtained by raising each entry to the $p$-th power.
		\item Even when a reductive group \(G\) is given explicitly---for example, \(G=\SO(\Lambda)_{\Fp}\)---we often simply speak of the \(G\)-zip stack rather than the \(\SO(\Lambda)_{\Fp}\)-zip stack.
	\end{enumerate}

    \subsection*{Acknowledgements}
Some initial discussions of this work took place during the authors' visits to Westlake University. We thank Yigeng Zhao for his invitation. During the preparation of this work, Q. Yan was affiliated with the Beijing Institute of Mathematical Sciences and Applications (BIMSA), and C. Zhang was affiliated with the Shing-Tung Yau Center of Southeast University (SEUYC). The authors thank BIMSA and SEUYC for their hospitality and support during visits to each other’s institutes. Q. Yan also thanks Ruiqi Bai and Qingchao Yu for helpful discussions.
	
	\section{Preliminaries}\label{S:preliminary}
	Throughout the paper, we let $p$ denote a fixed odd prime number. 
	
\subsection{Standing conventions on Shimura varieties} \label{S:Shv}
Let $(G, X)$ be a Shimura datum of Hodge type, together with a fixed Siegel embedding $(G, X) \hookrightarrow (\GSp_{2g}, \mathbb{H}^\pm_g)$. In this paper, we always assume that $G(\mathbb{Q}_p)$ admits a hyperspecial subgroup $\mathsf{K}_p$, i.e., $\mathsf{K}_p = \mathcal{G}(\mathbb{Z}_p)$ for a reductive model $\mathcal{G}$ of $G_{\mathbb{Q}_p}$. Whenever we talk about a Shimura variety $\Sh_G = \Sh_{G,\mathsf{K}} = \Sh_{\mathsf{K}}(G, X)$ attached to a Shimura datum $(G, X)$, we always assume that $\mathsf{K}$ is of the form $\mathsf{K} = \mathsf{K}_p \mathsf{K}^p \subseteq G(\mathbb{A}_f)$ with $\mathsf{K}^p$ sufficiently small. Of course, here we have suppressed the level subgroup $\mathsf{K}$ from the notation; this is what we will do from now on, unless the level becomes sensitive.

Moreover, when we talk about a morphism of Shimura varieties $\Sh_{G'} \to \Sh_G$, it is always understood that we have chosen level structures $\mathsf{K}' = \mathsf{K}'_p \mathsf{K}'^p \subseteq G'(\mathbb{A}_f)$ and $\mathsf{K} = \mathsf{K}_p \mathsf{K}^p \subseteq G(\mathbb{A}_f)$, with $\mathsf{K}'^p$ and $\mathsf{K}^p$ sufficiently small as above, and that the underlying homomorphism $G' \to G$ sends $\mathsf{K}'_p$ to $\mathsf{K}_p$ and $\mathsf{K}'^p$ to $\mathsf{K}^p$.

Let $E$ be the reflex field of $(G,X)$, $v\mid p$ a fixed place of $E$, with residue field $\kappa$ and ring of integers
	\(
	O_{E_v} = W(\kappa).
	\)
	Let $[\mu] = [\mu]_{W(\kappa)}$ be a conjugacy class over  $W(\kappa)$ of cocharacters of $\mathcal{G}$ determined by the Shimura datum $(G,X)$. Since $\mathcal{G}$ is unramified over $W(\kappa)$, the conjugacy class $[\mu]$ admits a representative cocharacter $\mu: \mathbb{G}_{m, W(\kappa)} \to \mathcal{G}_{W(\kappa)}$ defined over $W(\kappa)$. 
	Let
	\(
	\mathcal{S}h_{G,W(\kappa)}
	\)
	denote the (integral) canonical model of $\Sh_G$ over $W(\kappa)$, constructed by Kisin-Vasiu. To summarize, we fix the following \emph{mod $p$ data}:
	\begin{equation}\label{Eq:ModpData}
		(G,X),\, \mathcal{G},\, \mathsf{K}_p=\mathcal{G}(\Zp),\, \mathsf{K}=\mathsf{K}_p\mathsf{K}^p\subseteq G(\Af),\,
		E,\, v,\, O_{E_v}=W(\kappa),\, \mu=\mu_{W(\kappa)},\, \Sh_G=\Sh_{G,W(\kappa)}.
	\end{equation}
	 We call $\Sh_{G, \kappa} := \mathcal{S}h_{G, W(\kappa)} \otimes_{W(\kappa)} \kappa$ a mod $p$ Shimura variety over the residue field $\kappa$, and $\Sh_{G, \Fpbar} := \Sh_{G, \kappa} \otimes_{\kappa} \Fpbar$ the corresponding geometric mod $p$ Shimura variety.  We call the triple $(\mathcal{G}, \mathcal{G}_{W(\kappa)}, \mu)$, or $(\mathcal{G}, \mathcal{G}_{W(\kappa)}, [\mu])$ when $\mu$ is chosen, a \emph{reductive triple} over $W(\kappa)$.

\noindent\textbf{Convention:} In what follows, we use $\mathcal{S}h$ for (integral) canonical models and $\Sh$ for their fibers. Whenever we refer to a mod $p$ Shimura variety, the associated mod $p$ datum is understood to be as above. Moreover, any morphism between mod $p$ Shimura varieties is always induced by a morphism of the corresponding mod $p$ data in the following sense: we have two tuples of mod $p$ data as in \eqref{Eq:ModpData}, which we distinguish by subscripts 1 and 2, and the following:
\begin{enumerate}
    \item A morphism of Shimura data $\alpha: (G_1, X_1) \to (G_2, X_2)$ such that $\alpha(\mathsf{K}_{p,1}) \subseteq \mathsf{K}_{p,2}$. Let $E_1, E_2$ denote their reflex fields; then we automatically have $E_1 \supseteq E_2$.
    \item A requirement $v_1 | v_2 | p$, where $v_i$ is a place of $E_i$. Then we have induced inclusions of local fields $E_{2,v_2} \to E_{1,v_1}$ and local rings $W(\kappa_2) = O_{E_2, v_2} \to O_{E_1, v_1} = W(\kappa_1)$.
    \item A morphism of reductive triples $(\mathcal{G}_1, \mathcal{G}_{1, W(\kappa_1)}, \mu_1) \to (\mathcal{G}_2, \mathcal{G}_{2, W(\kappa_2)}, \mu_2)$, meaning a homomorphism of reductive $\mathbb{Z}_p$-group schemes $\mathcal{G}_1 \to \mathcal{G}_2$ and a morphism of conjugacy classes $[\mu_1] \to [\mu_2]$ over $W(\kappa)$. 
    \item A morphism of $p$-integral Shimura varieties $f: \mathcal{S}h_{G_1, W(\kappa_1)} \to \mathcal{S}h_{G_2, W(\kappa_2)} \otimes_{W(\kappa_2)} W(\kappa_1)$, compatible with the data above.
\end{enumerate}
With the setting and conventions above understood, in this paper we will simply write $\Sh_{G_1, \Fpbar} \to \Sh_{G_2, \Fpbar}$ for the base change along the embedding $\kappa_1 \hookrightarrow \Fpbar$ of the reduction mod $p$ of $f$, i.e., $\Sh_{G_1, \kappa_1} \to \Sh_{G_2, \kappa_2} \otimes_{\kappa_2} \kappa_1$.

    \subsection{The $G$-zip stack $\Gzips$}\label{S:Zipstack}
	Let $(G, \mu)$ be a reductive pair over a finite field $\kappa$; by this we mean $G$ is reductive group defined over $\Fp$ and $\mu: \mathbb{G}_{m, \kappa} \to G_{\kappa}$ is a cocharacter. Let \(P_{+}:=P_{\mu}\) and \(P_{-}:=P_{\mu^{-1}}\) be the parabolic subgroups attached to \(\mu\), with Levi decompositions \(P_{\pm}=U_{\pm}\rtimes L\). Their common Levi is \(L=P_{+}\cap P_{-}\), equivalently \(L=\mathrm{Cent}_G(\mu)\). Attached to the reductive pair $(G,\mu)$ is the quotient stack
	\[
	\Gzips := [E_{\mu}\backslash G_{\kappa}],
	\]
	introduced and studied in \cite{PWZ,PWZ2}, building on earlier work of Moonen--Wedhorn on $F$-zips \cite{MWFzip}.
	Here
	\(
	E_{\mu} \subseteq P_+ \times P_-^{\varphi}
	\)
	consists of pairs $(p_+,p_-)$ such that the Levi components $m(p_+)$ of $p_+$ and $m(p_-)$ of $p_-$ satisfy
	\(
	\varphi_L\bigl(m(p_+)\bigr) = m(p_-),
	\)
	where $\varphi_L \colon L \to L^{\varphi}$ is the relative Frobenius of $L$.
	We refer to \cite{PWZ2} for the moduli interpretation of $\Gzips$ in terms of $G$-zips. Associated with the geometric zip stack \((G\text{-}\mathsf{Zip}^\mu)_{\Fpbar}\) is the \(G\)-zip group datum over \(\Fpbar\),
\begin{equation}
	Z_{\mu}^{\data}=(G_{\Fpbar},\, P=(P_{\mu})_{\Fpbar},\, Q=(P_{\mu^{-1}}^{\varphi})_{\Fpbar},\, \varphi: L=L_{\Fpbar}\to L^{\varphi}),
\end{equation}
which is termed a \emph{connected algebraic zip datum} in \cite{PWZ2}. Here and below, we omit the subscript \(\Fpbar\) when it is clear from the context that we are working over \(\Fpbar\).

\subsection*{Frames}
The following definition is the special case relevant for us of the general notion of frames introduced in \cite{PWZ}.
\begin{definition}
	For the zip group datum \(Z_{\mu}^{\mathrm{data}}\) as above, a \emph{frame} is a triple \((B,T,g)\) such that \((B,T)\) is a Borel pair of \(G_{\Fpbar}\) and \(g \in G(\Fpbar)\), satisfying:
	(i) \(B \subseteq P\);
	(ii) \(g^{-1}Bg \subseteq Q\);
	(iii) \(\varphi(B \cap L) = g^{-1}Bg \cap L^\varphi\); and
	(iv) \(\varphi(T) = g^{-1}Tg \subseteq L^\varphi\).
\end{definition}

Let \(W\) be the Weyl group of \(G\) with respect to \((B,T)\), let \(W_\mu \subseteq W\) be the Weyl group of \(L\) with respect to \((B \cap L, T)\), and let \({}^\mu W \subseteq W\) be the set of minimal representatives for the coset space \(W_\mu \backslash W\).
    
	\begin{theorem}[{\cite[Thm.~7.5]
			{PWZ}}]\label{Thm:ZipstackTopWeyl} For the zip group data $Z_{\mu}^{\mathsf{data}}$, and fixed frame $(B,T, g)$ as above, we have
		homeomorphisms
		\[
		{}^\mu W \xrightarrow[\cong]{\mathsf{Frame}} |\Gzips| \cancong E_{\mu}(\Fpbar) \backslash G(\Fpbar), \quad w \mapsto G^w:=E_\mu\cdot (g w),
		\]
         where the topology of ${}^\mu W$ is induced by the partial order recalled in Theorem \ref{Thm:PartOrd} below. 
	\end{theorem}    
	
	Sometimes, for the convenience of computing zip orbits and the index set ${}^\mu W$, one uses another algebraic group that has the same adjoint group as $G$. Let $(G, \mu)$ be a reductive pair as above, with $G$ defined over $\mathbb{F}_p$, and let
	\(
	f: G \to G'
	\)
	be a morphism of reductive groups over $\mathbb{F}_p$. Set
	\(
	\mu' := f \circ \mu.
	\)
	We then also have the zip stack
	\(
	G'\text{-}\mathsf{Zip}^{\mu'}
	\)
	for the reductive pair $(G', \mu')$ over $\kappa$. The map $f$ induces a morphism of $\kappa$-algebraic stacks,
	\[
	[f]: \Gzips \to G'\text{-}\mathsf{Zip}^{\mu'}.
	\]
	
	\begin{lemma}\label{Lem:ZipIdentity}
		If $f$ induces an isomorphism
		\(
		G^{\ad} \to G'^{\ad},
		\)
		then $[f]$ induces a homeomorphism
		\[
		|\Gzips| \cong |G'\text{-}\mathsf{Zip}^{\mu'}|.
		\]
		
	\end{lemma}
	\begin{proof}
		This result is well known. See \cite[Lem.~4.22]{WZEoCycle} for a proof.
	\end{proof}

	\subsection{Functoriality of zip period maps}\label{S:ZipPerMap}

    Let $\Sh_{G, \kappa}$ be a mod $p$ Shimura variety of Hodge type. Recall from \cite{EOZhang} that the EO stratification of $\Sh_{G, \Fpbar}=\Sh_{G, \kappa}\otimes_{\kappa}\Fpbar$ is defined via taking the geometric fibers of the \emph{zip period map} (cf. Theorem \ref{Thm:EOAbel} below),
    \[\zeta: \Sh_{G, \kappa}\to \Gzips.\] It is a standard fact that the zip period map is functorial in the sense that if we have a morphism between geometric mod $p$ Shimura varieties $\Sh_{G, \Fpbar}\to \Sh_{G', \Fpbar}$, then we have a commutative diagram, 
		\begin{equation}
			\xymatrix{\Sh_{G, \Fpbar}\ar[r]\ar[d]& \Sh_{G', \Fpbar}\ar[d]\\ \Gzips\ar[r] & G'\text{-}\mathsf{Zip}^{\mu'}.
			}
		\end{equation}

	\section{Orthogonal and spin Shimura varieties}\label{S:ShvGSpinSO}
	\subsection{Orthogonal and Spin Groups}
    We recall basic facts about $\GSpin$ and $\SO$-Shimura varieties, mainly following \cite{KS} and \cite{CIMK}, and record the setup needed for the later sections.
	Let $(V,\Psi)$ be a quadratic space over $\Q$ with symmetric bilinear form
	$\Psi\colon V\times V\to\Q$, and let $q$ be the associated quadratic form
	defined by $q(v)=\Psi(v,v)$. Assume that the real base change $(V,\Psi)_\bR$
	has signature $(n,2)$. For simplicity, we refer to the quadratic space simply
	as $V$, with $\Psi$ and $q$ understood. We also assume the existence of a
	self-dual $\Z_{(p)}$-lattice $\Lambda\subseteq V$, which will be fixed later.
	Let $\SO(V)$ be the special orthogonal group determined by $(V,\Psi)$ over $\Q$.
	Recall that the Clifford algebra $C(V)$ of $(V, q)$ is
	\[
	C(V)=T(V)/\langle\,v^2-q(v):\,v\in V\,\rangle,
	\]
	where $T(V)$ is the tensor algebra of $V$. This algebra is a
	$\Z/2\Z$-graded associative $\Q$-algebra,
	\(
	C(V)=C^+(V)\oplus C^-(V), 
	\)
	with $C^+(V)$ a $\Q$-subalgebra. As a $\Q$-vector space, $C(V)$ has dimension
	$2^{n+2}$ and hence $\dim_\Q C^+(V)=2^{n+1}$. As a $C^+(V)$-algebra, $C(V)$ is
	generated by the image of $V$ (which lies in $C^-(V)$). Moreover, $C(V)$ is
	equipped with a canonical involution $x\mapsto x^*$, characterized by
	$(v_1\cdots v_m)^*=v_m\cdots v_1$ for all $v_1,\ldots,v_m\in V$.
	The formation of $C(V)$ and $C^+(V)$ is functorial in $\Q$-algebras; that is,
	for every $\Q$-algebra $R$ one has $C(V)_R=C(V_R)$ and $C^+(V)_R=C^+(V_R)$.
	
	Recall also that the spinor similitude group $\GSpin(V)$ is the group functor that
	assigns to each $\Q$-algebra $R$ the group
	\[
	\GSpin(V)(R)
	=\{\,g\in C^+(V_R)^\times : gV_Rg^{-1}=V_R,\ \eta(g)\in R^\times\,\},
	\]
	where $\eta\colon \GSpin(V)\to\mathbb{G}_{m,\Q}$ is the similitude character
	given by $\eta(g)=g^*g$, and whose kernel is the derived subgroup
	$\Spin(V)$. By definition, $\GSpin(V)$ acts on $V$ (via conjugation in the Clifford
	algebra), inducing a homomorphism
	\(
	\GSpin(V)\to \GL(V)
	\)
	that factors through $\SO(V)$. The induced orthogonal representation
	\(
	\pi: \GSpin(V)\to\SO(V)
	\)
	is surjective with kernel $\mathbb{G}_m$, where $\mathbb{G}_m$ embeds into
	$\GSpin(V)$ by mapping, for each $\Q$-algebra $R$, the units $R^\times$ into
	$C^+(V_R)^\times$ via the structure map $R\to C^+(V_R)=C^+(V)_R$. In particular, the map $\pi$ is a central isogeny.
	
	\subsection{Orthogonal and spin Shimura varieties}
	
	Define
	\[
	\mathbb{D}:=\{v\in V_\C \mid \Psi(v,v)=0,\;\Psi(v,\bar{v})<0\}/\C^\times
	\subseteq \mathbb{P}(V_\C)
	\]
	as the space of oriented negative $2$-planes in \(V_\bR\).
	Given a class \(z=a+ib\in V_\C\) with \(a,b\in V_\bR\) representing an
	element of \(\mathbb{D}\), the corresponding negative $2$-plane is
	\(\operatorname{Span}_\bR\{a,b\}\) with oriented basis \(\{a,b\}\).
	
	The group \(\SO(V)(\bR)\) acts on \(\mathbb{D}\) via its natural action on
	\(V_\bR\) (and hence on \(V_\C\)). The pair \((\SO(V),\mathbb{D})\) forms a
	Shimura datum in the sense of Deligne \cite{DeligneShimura}, with reflex field
	\(\Q\) for $n\geq 1$, and an imaginary quadratic field for $n=0$; see \S~\ref{S:0DimGSpin} below for a separate discussion on $0$-dimensional GSpin Shimura varieties. Denote by $E$ the reflex field of $(\SO(V),\mathbb{D})$.
	For each compact open subgroup \(\mathsf{K}\subseteq \SO(V)(\Af)\), denote by
	\(\Sh_{\mathsf{K}}(\SO(V),\mathbb{D})\) the associated orthogonal Shimura variety.
	This is a Deligne–Mumford stack over \(\Q\) whose $\C$-valued points, as a
	groupoid, are identified with the orbifold quotient
	\[
	\SO(V)(\Q)\backslash\bigl(\mathbb{D}\times \SO(V)(\Af)/\mathsf{K}\bigr).
	\]
	When \(\mathsf{K}\) is sufficiently small, \(\Sh_{\mathsf{K}}(\SO(V),\mathbb{D})\) is a
	quasi-projective \(E\)-scheme. Similarly, the action of \(\GSpin(V)(\bR)\) on \(V_\bR\) via the projection $\GSpin(V)(\R)\to \SO(\R)$ induces an action on
	\(\mathbb{D}\), yielding the spin Shimura datum \((\GSpin(V),\mathbb{D})\) and
	the corresponding Shimura variety \(\Sh_{\mathsf{K}'}(\GSpin(V),\mathbb{D})\) for
	each compact open \(\mathsf{K}'\subseteq \GSpin(V)(\Af)\).
	The orthogonal representation \(\pi\colon \GSpin(V)\to \SO(V)\) induces a
	morphism of Shimura data
	\(
	(\GSpin(V),\mathbb{D})\to (\SO(V),\mathbb{D})
	\)
	and, whenever \(\pi(\mathsf{K}')\subseteq \mathsf{K}\), a morphism of Shimura varieties
	\[
	\pi:\Sh_{\GSpin(V)}:= \Sh_{\mathsf{K}'}(\GSpin(V),\mathbb{D})
	\to
	\Sh_{\SO(V)}:= \Sh_{\mathsf{K}}(\SO(V),\mathbb{D}).
	\]
	These varieties are quasi-projective smooth schemes of dimension $n$ over $E$. 
	Via the Kuga–Satake construction, the Shimura variety
	\(\Sh_{\mathsf{K}}(\SO(V),\mathbb{D})\) embeds into a Siegel modular variety, making
	\(\Sh_{\mathsf{K}'}(\GSpin(V),\mathbb{D})\) a Hodge type Shimura variety and
	\(\Sh_{\mathsf{K}}(\SO(V),\mathbb{D})\) one of abelian type.
	More explicitly, choose an element \(\delta\in C^+(V)^\times\) such that
	\(\delta^*=-\delta\). Then the form \( \Psi_\delta(x,y)=\operatorname{Tr}(x\,\delta\,y^*) \) defines a perfect alternating pairing on the \(\Q\)-vector space \(C^+(V)\). Denote by \(\GSp(C^+(V),\Psi_\delta)\) the corresponding symplectic
	group over \(\Q\). Consider the left multiplication action of $\GSpin(V)$ on $C^+(V)$, which gives rise to an embedding of reductive $\Q$-groups, 
	\[
	\rho_{\KS}: \GSpin(V)\to \GSp(C^+(V), \Psi_{\delta})\cong \GSp_{2\rg}, \quad  \rg= 2^n.
	\]
	Indeed, one checks that $\Psi_{\delta}(gx, gy)=\eta(g) \Psi_{\delta}(x, y)$ for all $g\in \GSpin (V)$. 
	Thus we obtain the Siegel Shimura datum
	\((\GSp(C^+(V),\Psi_\delta),\mathbb{H}^\pm_{\rg})\). This embedding induces the following morphism of 
	Shimura data by Kuga-Satake construction, 
	\begin{equation}\label{Eq:SeigelEmbedd}
		(\GSpin(V),\mathbb{D})
		\hookrightarrow
		\bigl(\GSp(C^+(V),\Psi_\delta), \mathbb{H}_{\rg}^{\pm}\bigr)\cong \bigl(\GSp_{2\rg},\mathbb{H}^\pm_{\rg}\bigr),
	\end{equation}
	hence also morphism of Shimura varieties over $E$,
	\(
	\iota_{\KS}: \Sh_{\GSpin(V)}\to \Sh_{\GSp(C^+(V))}\cong \Sh_{\GSp_{2\rg}}.
	\)
	
	\subsection{Hodge cocharacters}
	Let the setting be as in the previous subsection.  Recall that for an oriented negative definite $2$-plane
	$W\subseteq V_\R$, the orientation endows $W$ with a canonical complex structure, hence a morphism
	\[
	h_{W}:\mathbb{S}=\mathrm{Res}_{\C/\R}\mathbb{G}_m \longrightarrow \SO(V)_\R
	\]
	characterized by the property that $h_W$ acts on $W\simeq \C$ by complex
	scalar multiplication and acts trivially on $W^\perp$.  Base-changing to $\C$
	and using $\mathbb S_\C\simeq \mathbb{G}_m\times \mathbb{G}_m$, we obtain the associated (inverse)
	Hodge cocharacter
	\[
	\mu_{W_0}:\mathbb{G}_{m,\C}\longrightarrow \SO(V)_\C,\qquad \mu_{W_0}(z)=h_{W_0,\C}(z,1),
	\]
	whose $\SO(V)(\C)$-conjugacy class depends only on the Shimura datum (hence is
	defined over the reflex field).  Fix once and for all a base point
	$W_0\in\mathbb D$. Suppose now that $n\geq 1$, then $E=\Q$ and one can choose a Witt decomposition over $\Q$,
	\[
	V = V_{hyp} \oplus V_{hyp}^{\perp}, \quad V_{hyp}=\Q\langle e, f\rangle,\qquad \Psi(e,e)=\Psi(f,f)=0,\ \Psi(e,f)=\tfrac12,
	\]
	such that under the induced decomposition over $\C$ the Hodge filtration attached
	to $W_0$ satisfies $V^{1,-1}=\C e$ and $V^{-1,1}=\C f$.  Choose a $\Q$-basis \( \{\delta_2,\dots,\delta_{n+1}\}\) of $V_{hyp}^{\perp}$. With respect to the
	ordered $\Q$-basis,
	\[
	\mathfrak B=\{\,\delta_1=e,\,\delta_2,\dots,\delta_{n+1},\
	\delta_{n+2}=f\,\},
	\]
	a standard representative of the conjugacy class $[\mu]_\C$ (for the inverse
	Hodge cocharacter) on the standard representation of $\SO(V)_\C$ is
	\[
	\mu_{\SO}:\mathbb{G}_{m,\C}\to \SO(V)_\C,\qquad t\longmapsto \diag(t,1,\dots,1,t^{-1}).
	\]
	
	Put $u:=ef\in C^+(V_{hyp})\subseteq C^+(V)$; then $u^2=u$ and $u^*=fe=1-u$.  Define
	\[
	\mu:\mathbb{G}_{m,\C}\to \GSpin(V)_\C\subseteq C^+(V)_\C{}^\times,\qquad
	\mu(t):=t\,u+(1-u).
	\]
	One checks that $\eta(\mu(t))=\mu(t)^*\mu(t)=t$, so $\mu$ indeed lands in
	$\GSpin(V)_\C$ with similitude character $t$.  Moreover, under the natural
	projection $\pi:\GSpin(V)\to \SO(V)$, the cocharacter $\mu$
	projects to $\mu_{\SO}$; in particular, $\mu$ is the standard lift of the inverse
	Hodge cocharacter determined by the chosen base point $W_0\in\mathbb D$. Clearly $\mu_{\SO}$ and $\mu$ are defined over $\Q$. From now on, we fix the choice
\[
u=\delta_1\delta_{n+2}\in C^+(V), \qquad \delta:=1-2u\in C^+(V).
\]
Then \(\delta^*=-\delta\), and we use this fixed \(\delta\) to define the perfect alternating form \(\Psi_{\delta}\) on \(C^+(V)\).  Over $\C$ we have the $u$-splitting
	$C^+(V_{\C})=u\,C^+(V_{\C})\oplus u^*\,C^+(V_{\C})$ with $u^*=1-u$, and one checks that
	$\mu_{\mathrm{Siegel}}:=\rho_{\KS}\circ\mu$ acts by scalar multiplication by $t$ on
	$u\,C^+(V_{\C})$ and trivially on $u^*\,C^+(V_{\C})$.  Thus, choosing a $\C$-basis of
	$u\,C^+(V_{\C})$ and the $\Psi_\delta$-dual basis on $u^*\,C^+(V_{\C})$, we obtain a
	symplectic basis for which $\Psi_\delta$ has matrix $\antidiag(I_{\rg},-I_{\rg})$ and
	$\mu_{\mathrm{Siegel}}$ is identified with the standard Siegel cocharacter
	$t\mapsto \diag(tI_{\rg},I_{\rg})$.

    If $n=0$, $(V, q)$ is negative definite and cannot be a hyperbolic plane, one needs to pass to the base change $V_\kk$ to split the quadratic form ($q$ becomes $x_1^2-x_2^2$) to obtain the analogues of the idempotent $u$ as above. In this case, we still have the cocharacter representative $\mu_{\SO}$ and $\mu$ described as above. But in this case, they are only defined over $\kk$ (rather than $\Q$).
	
    We assume the existence of a self-dual $\mathbb{Z}_{(p)}$-lattice $\Lambda \subseteq V$. When $n \ge 1$, after replacing $\Lambda$ by an isometric lattice, we may assume that $\Lambda := \mathbb{Z}_{(p)}\langle \delta_1, \dots, \delta_{n+2} \rangle$ is a self-dual lattice of $V$. Moreover, the cocharacter representatives $\mu_{\SO}$ and $\mu$ are defined over $O_{E,(p)}$. For our purpose, we choose a place $v$ of $E$ above $p$. Then $v=p$ in case $n\geq 1$ and $v$ is split or inert in $E=\kk$ when $n=0$; cf. \S~\ref{S:0DimGSpin}. Consider the base change to $O_{E_v}=W(\kappa)$ (see \S~\ref{S:Shv} for notation) of these cocharacters.  We then have a diagram of morphisms of reductive pairs over $W(\kappa)$:  
\begin{equation}\label{Eq:ReductivePairDiag}
\xymatrix{
(\GSpin(\Lambda), \mu) \ar[r] \ar[d] & (\GSp(C^+(\Lambda), \Psi_{\delta}), \mu_{\mathrm{Siegel}}) \\
(\SO(\Lambda), \mu_{\SO}). &
}
\end{equation}
We write \(\mathcal{S}h_{\SO(\Lambda)}, \mathcal{S}h_{\GSpin(\Lambda)}\), and \(\mathcal{S}h_{\GSp(C^+(\Lambda))}\) for the canonical integral models over \(W(\kappa)\) of \(\Sh_{\SO(V)}, \Sh_{\GSpin(V)}\), and \(\Sh_{\GSp(C^+(V))}\), respectively. Let
\[
\pi_{W(\kappa)} : \mathcal{S}h_{\GSpin(\Lambda)} \to \mathcal{S}h_{\SO(\Lambda)}, \quad
\iota_{\KS,W(\kappa)} : \mathcal{S}h_{\GSpin(\Lambda)} \to \mathcal{S}h_{\GSp(C^+(\Lambda))},
\]
be morphisms of \(W(\kappa)\)-schemes with generic fibers \(\pi_{E_v}\) and \(\iota_{\KS,E_v}\), respectively. Following our notational convention in \S~\ref{S:Shv}, we write \(\Sh_{\GSpin(\Lambda),\Fpbar}\), \(\Sh_{\SO(\Lambda),\Fpbar}\), and \(\Sh_{\GSp(C^+(\Lambda)),\Fpbar}\) for the corresponding geometric mod \(p\) fibers.

      \subsection{Interlude on zero dimensional GSpin Shimura varieties}\label{S:0DimGSpin}
   In the case \(n=0\), the quadratic space \(V\) has signature \((0,2)\), hence \(\dim_{\Q}V=2\) and \(V_{\R}\) is negative definite. Choosing an orthogonal basis \(\delta_1,\delta_2\) of \(V\), the even Clifford algebra is
\(
C^+(V)=\Q\oplus \Q\cdot(\delta_1\delta_2),
\)
and, since \(\delta_1\delta_2=-\delta_2\delta_1\), one has
\[
(\delta_1\delta_2)^2=-\,\delta_1^2\delta_2^2=-\,q(\delta_1)q(\delta_2).
\]
Therefore
\(
C^+(V)\cong \Q[X]/(X^2+q(\delta_1)q(\delta_2)).
\)
Since \(q(\delta_1)\) and \(q(\delta_2)\) are both negative over \(\R\), the algebra \(C^+(V)\) is an imaginary quadratic field; denote it by \(\kk\). Moreover, since \(\dim_{\Q}V=2\), the odd part of the Clifford algebra is exactly \(C^-(V)=V\). As multiplication preserves parity, every \(g\in C^+(V)^\times\) preserves \(V\) under conjugation. Since \(C^+(V)=\kk\) is commutative, the canonical involution on \(C^+(V)\) is the nontrivial automorphism of \(\kk/\Q\), so for every \(g\in C^+(V)^\times\) one has \(g^*g=N_{\kk/\Q}(g)\in \Q^\times\). Hence \(C^+(V)^\times=\GSpin(V)\). Therefore there is an isomorphism of \(\Q\)-algebraic groups
\[
\GSpin(V)\cong C^+(V)^\times\cong \mathrm{Res}_{\kk/\Q}\mathbb{G}_m,
\]
under which the similitude character \(\eta:\GSpin(V)\to \mathbb{G}_m\) is identified with the norm morphism \(N_{\kk/\Q}:\mathrm{Res}_{\kk/\Q}\mathbb{G}_m\to \mathbb{G}_m\). In particular, \(\Sh_{\GSpin(V)}\) is a zero-dimensional Shimura variety, with reflex field \(\kk\).

Let $\Lambda \subset V$ be a self-dual $\mathbb{Z}_{(p)}$-lattice, whose existence we have assumed. This lattice provides a reductive group scheme model $\mathcal{G} := \GSpin(\Lambda)$ over $\mathbb{Z}_{(p)}$ for the group $\GSpin(V)$. The self-duality of $\Lambda$ implies that $p$ is \emph{unramified} in the imaginary quadratic field $\kk$. Under the identification $\GSpin(V) \cong \mathrm{Res}_{\kk/\mathbb{Q}} \mathbb{G}_m$, the integral group scheme is an unramified torus given by $\mathcal{G} \cong \mathrm{Res}_{O_{\kk, (p)}/\mathbb{Z}_{(p)}} \mathbb{G}_m$. Consequently, the zero-dimensional Shimura variety $\Sh_{\GSpin(V)}$ admits an integral canonical model $\mathcal{S}h_{\GSpin(V)}$ over $O_{\kk, (p)}$ that is finite and \'etale. This model is isomorphic to a disjoint union of $\mathrm{Spec}O_{\kk, (p)}$. In particular, for each place $v$ of $\kk$ above $p$, we have a local integral model $\mathcal{S}h_{\GSpin(\Lambda), O_{\kk,v}}$ over $O_{\kk, v}$. Following \S~\ref{S:Shv} we write $\kappa$ for the residue field of $\kk_v$; then $\kappa = \mathbb{F}_p$ if $p$ splits in $\kk$, and $\kappa = \mathbb{F}_{p^2}$ if $p$ is inert in $\kk$. 

	\subsection{Natural embeddings of spin and orthogonal Shimura varieties}\label{S:ShiEmb} 
	Let $V$ be as above, and let $V'\subseteq V$ be a $\Q$-subspace such that
$(V',q'=q|_{V'})_{\R}$ is of signature $(n-1,2)$. Let $(\SO(V'), \mathbb{D}')$
be the Shimura datum attached to $(V', q')$. Let
$\mathfrak{B}=\{\delta_1, \ldots, \delta_{n+2}\}$ be the basis of $V$ fixed in
\S~2.3, and assume that
$V'=\Span_{\Q}\{\delta_1,\ldots,\widehat{\delta_{m+1}},\ldots,\delta_{n+2}\}$,
where $\widehat{\delta_{m+1}}$ means that $\delta_{m+1}$ is deleted. Here we have
set $m=\frac n2+1$ if $n$ is even and $m=\frac{n+1}{2}$ if $n$ is odd. Let
$u$ and $\delta$ be the elements fixed in \S~2.3;
we use this $\delta$ to define $\Psi_{\delta}$ on both $C^+(V')$ and $C^+(V)$.
Then we obtain morphisms of Shimura data,
	\[
	\begin{aligned}
		(\SO(V'), \mathbb{D}') \to (\SO(V), \mathbb{D}),\quad
		(\GSpin(V'), \mathbb{D}') \to (\GSpin(V), \mathbb{D}),\\
		\bigl(\GSp(C^+(V'), \Psi_{\delta}), \mathbb{H}^\pm_{\rg'}\bigr)
		\to
		\bigl(\GSp(C^+(V), \Psi_{\delta}), \mathbb{H}^\pm_{\rg}\bigr).
	\end{aligned}
	\]
	For the purposes of reduction modulo $p$, we also fix the following self-dual $\Z_{(p)}$-lattice for $V'$,
	\[
	\Lambda':=\Span_{\Z_{(p)}}\{ \delta_1,\ldots,\widehat{\delta_{m+1}}, \ldots,\delta_{n+2}\}= \Lambda\cap V'.
	\]
   
    Let \(E\) be the reflex field of \((\SO(V'),\mathbb{D}')\), and choose a place \(v\) of
\(E\) above \(p\). As usual, we write \(\kappa\) for the residue field of \(E_v\). Let \(\mu'_{\SO}\), \(\mu'\), and \(\mu'_{\mathrm{Siegel}}\) be cocharacters of
\(\SO(\Lambda')\), \(\GSpin(\Lambda')\), and \(\GSp(C^+(\Lambda'),\Psi_{\delta})\),
respectively, defined over \(O_{E_v}=W(\kappa)\), as in \S~2.3. Since
\((\SO(V),\mathbb{D})\) has reflex field \(\Q\) for all \(n\ge 1\), the unprimed
cocharacters are defined over \(\Z_p\) and will be viewed over \(W(\kappa)\) by base
change. With these data, we have the following commutative
	diagram of morphisms between reductive pairs over $W(\kappa)$, and between morphisms of Shimura
	varieties over $W(\kappa)$ respectively (cf. convention in \S~\ref{S:Shv}):
	\begin{equation}\label{Eq:SpinEOProjInt}
		\begin{minipage}{0.48\textwidth}
			\centering
			\xymatrix{
				(\GSpin(\Lambda'), \mu') \ar[d]\ar[r]& (\GSpin(\Lambda), \mu)\ar[d]\\
				(\SO(\Lambda'), \mu'_{\SO})\ar[r] & (\SO(\Lambda), \mu_{\SO}),
			}
		\end{minipage}\hfill
		\begin{minipage}{0.48\textwidth}
			\centering
			\xymatrix{\mathcal{S}h_{\GSpin(\Lambda')} \ar[d]\ar[r]& \mathcal{S}h_{\GSpin(\Lambda)}\ar[d]\\
				\mathcal{S}h_{\SO(\Lambda')}\ar[r] & \mathcal{S}h_{\SO(\Lambda)}.}
		\end{minipage}
	\end{equation}
	Moreover, after choosing the compatible symplectic bases discussed below, we obtain the following commutative diagram; its left square is a diagram of morphisms of reductive pairs, and the right horizontal arrows are the chosen identifications with the standard Siegel data:
	\begin{equation}
		\xymatrix{
			(\GSpin(\Lambda'), \mu') \ar[d]\ar[r]& (\GSp(C^+(\Lambda'), \Psi_{\delta}), \mu'_{\mathrm{Siegel}})\ar[d]\ar[r]^{\quad \quad\cong }& (\GSp_{2\rg'}, \mathbb{H}_{\rg'}^\pm) \ar[d]\\
			(\GSpin(\Lambda), \mu) \ar[r] & (\GSp(C^+(\Lambda), \Psi_{\delta}), \mu_{\mathrm{Siegel}})\ar[r]^{\quad \quad\cong }& (\GSp_{2\rg}, \mathbb{H}_{\rg}^\pm),
		} 
	\end{equation}
	and between morphisms of Shimura varieties over $W(\kappa)$,
	\begin{equation}\label{Eq:EmdKSInt}
		\xymatrix{\mathcal{S}h_{\GSpin(\Lambda')} \ar[d]\ar[r]& \mathcal{S}h_{\GSp(C^+(\Lambda'))}\ar[d]\ar[r]^{\cong}&\mathcal{S}h_{\GSp_{2\rg'}}\ar[d]\\
			\mathcal{S}h_{\GSpin(\Lambda)} \ar[r]& \mathcal{S}h_{\GSp(C^+(\Lambda))}\ar[r]^{\cong}&\mathcal{S}h_{\GSp_{2\rg}}.}
	\end{equation}
    Here we need to clarify the right commutative diagrams. Using the decomposition $C^+(\Lambda)= C^+(\Lambda')\oplus C^-(\Lambda')\delta_{m+1} $, one may choose a compatible symplectic basis for $C^+(\Lambda')$ and $C^+(\Lambda)$ such that under the resulting identifications
\[
\GSp(C^+(\Lambda'),\Psi_\delta)\simeq \GSp_{2\rg'}
\qquad\text{and}\qquad
\GSp(C^+(\Lambda),\Psi_\delta)\simeq \GSp_{2\rg},
\]
with \(\rg=2\rg'\), the natural map
\(
\GSp(C^+(\Lambda'),\Psi_\delta)\longrightarrow \GSp(C^+(\Lambda),\Psi_\delta)
\)
is given by
\[
\begin{pmatrix}
A&B\\ C&D
\end{pmatrix}
\longmapsto
\begin{pmatrix}
A&0&B&0\\
0&A&0&B\\
C&0&D&0\\
0&C&0&D
\end{pmatrix}
\]
This means that if $\mu'_{\mathrm{Siegel}}$ is given by $t\mapsto \diag (t I_{\rg'}, I_{\rg'})$, $\mu_{\mathrm{Siegel}}$ is given by $t\mapsto \diag (t I_{\rg}, I_{\rg})$. Recall $\rg'=2^{n-1}$ and $\rg=2\rg'$. Moreover, the map $\mathcal{S}h_{\GSp_{2\rg'}}\to \mathcal{S}h_{\GSp_{2\rg}}$ of Siegel modular varieties is
	given by $(A, \lambda) \mapsto (A\times A, \lambda \times \lambda)$, where $(A, \lambda)$ is a principally
	polarized abelian variety of dimension $\rg'$. 
	
	We assume also that every morphism of Shimura varieties in \eqref{Eq:SpinEOProjInt} and
\eqref{Eq:EmdKSInt} is an embedding by carefully choosing level subgroups away from $p$.

	\section{Normalized reductive pair $(\SO(n,2)_{\Fpbar}, \mu)$ and the associated zip datum}\label{S:RedPairNorm}
	
	\subsection{Reductive triples and associated $G$-zip data}
	
	By a \emph{reductive triple} over a finite field $k \supseteq \Fp$ we mean
data $(G, G_k, [\mu]_k)$, where $G$ is a reductive group over $\Fp$,
$G_k := G \otimes_{\Fp} k$, and $[\mu]_k$ is a $G_k(\bar{k})$-conjugacy class
of cocharacters of $G_{\bar{k}}$ defined over $k$. Since $k$ is finite, such a
class admits a representative $\mu:\mathbb{G}_{m,k} \to G_k$. In this paper, we
usually fix such a representative and still refer to $(G, G_k, \mu)$ as a
reductive triple; cf. \S~\ref{S:Shv}. 

We now consider the reductive triple
\(
(\SO(\Lambda)_{\Fp}, \SO(\Lambda)_{\kappa}, [\mu]_{\kappa})
\)
attached to the geometric mod $p$ Shimura variety
$\Sh_{\SO(\Lambda),\Fpbar}$ introduced in the previous section, and fix a
representative
\(
\mu \colon \mathbb{G}_{m,\kappa} \to \SO(\Lambda)_{\kappa}
\)
of $[\mu]_{\kappa}$. The associated $G$-zip stack over $\kappa$ is,
\[
\SO(\Lambda)_{\kappa}\text{-}\mathsf{Zip}^{\mu}
= [E_{\mu} \backslash \SO(\Lambda)_{\kappa}].
\]

	\subsection{Normalization isomorphisms}\label{S:NormPair}
    Once an $\Fpbar$-basis $\mathfrak{B}=\{\delta_1,\delta_2,\ldots,\delta_{n+2}\}$ of $\Lambda_{\Fpbar}$ is fixed, we write $\SO(n,2)_{\Fpbar}$ for the image of
\[
\SO(\Lambda)_{\Fpbar}\subseteq \GL(\Lambda)_{\Fpbar}\overset{\epsilon_{\mathfrak B}}{\cong}\GL_{n+2,\Fpbar}.
\]
If $\mathfrak B$ is defined over a subfield $k\subseteq \Fpbar$, we similarly write $\SO(n,2)_k$ for the corresponding image of $\SO(\Lambda)_k$ in $\GL_{n+2,k}$ whenever it is necessary to emphasize the field of definition. We call the resulting isomorphism $\SO(\Lambda)_{k}\cong \SO(n,2)_k$ the \emph{normalization isomorphism} (w.r.t. $\mathfrak{B}$). Transporting the chosen reductive triple along $\epsilon_{\mathfrak B}$ yields a normalized geometric reductive pair $(\SO(n,2)_{\Fpbar},\mu)$ and a normalized geometric $G$-zip group datum
\begin{equation}
(Z_{\mu}^{\data})_{\Fpbar}=(\SO(n,2)_{\Fpbar},\,P,\,Q=P_-^{\varphi},\,\varphi\colon L\to L^{\varphi}),
\end{equation}
and hence the associated \emph{normalized} geometric $G$-zip stack
\(
\SO(n,2)_{\Fpbar}\text{-}\mathsf{Zip}^{\mu}=[E_{\mu}\backslash \SO(n,2)_{\Fpbar}]
\); see \S~\ref{S:Zipstack} for the notation. We include the subscript $\Fpbar$ in $(Z_{\mu}^{\data})_{\Fpbar}$ to emphasize that we work geometrically. One shall see that the normalized zip group datum depends on the choice of normalization isomorphism, especially its \emph{field of definition} $k$. 

From now on, whenever we write $\SO(n,2)_{\Fpbar}$ or say that a reductive pair is ``normalized'', a basis $\mathfrak B$, clear from the context, is understood to have been fixed. We choose $\mathfrak B$ so that the resulting cocharacter
\(
\mu:\mathbb{G}_{m,\Fpbar}\to \SO(n,2)_{\Fpbar}
\)
is given by
\(
t\mapsto \diag(t,1,\ldots,1,t^{-1}).
\)
We also let \(m\) denote the reductive rank of \(\SO(n,2)_{\Fpbar}\); thus
\(m=\frac{n+1}{2}\) when \(n\) is odd, and \(m=\frac{n}{2}+1\) when \(n\) is even. This agrees with the convention in \S~\ref{S:ShiEmb}. The root system of $\SO(n,2)_{\Fpbar}$ is of Dynkin type $B_m$ (resp.\ $D_m$) if $n$ is odd (resp.\ even). Finally, we recall that a nondegenerate quadratic form on $\Lambda_{\Fp}$ is determined up to isomorphism by its discriminant in $\mathbb{F}_p^\times/\mathbb{F}_p^{\times 2}$ (see \cite[p.~35]{course}); we will use this implicitly when writing explicit matrix expressions for $q$.

	\subsection{The case $n $ is odd}\label{S:ZipdataNormOdd}
	In this case, $\kappa=\Fp$. By the classification of quadratic forms over $\Fp$, under suitable $\Fp$-basis $\mathfrak{B}$ of $\Lambda_{\Fp}$, $q$ takes the form,
	\begin{equation}\label{Eq:StdFormOdd}
		q = x_1x_{n+2} + x_2x_{n+1} + \cdots + x_mx_{m+2} + c\,x_{m+1}^2,   \quad \quad \text { (standard form) }
	\end{equation}
	for some $c \in \mathbb{F}_p^\times$. We use the basis $\mathfrak{B}$ to obtain the normalization isomorphism $\SO(\Lambda)_{\Fp}\cong \SO(n, 2)_{\Fp}$. Clearly, we have $\SO(n,2)_{\Fp}= \SO(J_{n+2})_{\Fp}$, consisting of matrices $X$ in  $\GL_{n+2, \Fp}$ such that $X^tJ_{n+2} X=J_{n+2}$ and $\det(J_{n+2})=1$, where $J_{n+2}$ is the Gram matrix,
	\[J_{n+2}=\antidiag(1/2, \ldots, 1/2, c, 1/2, \ldots, 1/2).\] 
	Then the transported conjugacy class $[\mu]_{\Fpbar}$ on $\SO(n, 2)_{\Fpbar}$ admits a representative 
	\begin{equation}\label{Eq:MuNorm}
		\mu: \mathbb{G}_{m, \Fpbar}\to \SO(n, 2)_{\Fpbar}, \quad t\mapsto \diag(t, 1,\ldots, 1, t^{-1}), 
	\end{equation}
	which is already defined over $\Fp$. The normalized geometric $G$-zip datum $(Z_{\mu}^{\data})_{\Fpbar}$ admits concrete description as follows. The parabolic $P$ (resp. $Q$) consists of block upper triangular matrices 
	\[
	\begin{pmatrix}
		\ast & \ast & \ast\\
		0 & A & \ast\\
		0 & 0 & \ast
	\end{pmatrix} \, \text{ resp. }  \begin{pmatrix}
		\ast & 0 & 0\\
		\ast & A & 0\\
		\ast & \ast & \ast
	\end{pmatrix},  \quad  A\in \SO(J_n),
	\] 
	where $J_n\in \GL_{n}$ is of the same form as $J_{n+2}$. It is easy to see that  $L\subseteq \SO(n, 2)$ consists of matrices of the form $\diag(t, A, t^{-1})$, with $t\in \mathbb{G}_m$, and $A\in \SO(J_{n})$. In particular, we have $L\cong \mathbb{G}_m\times \SO(J_n)$. The Frobenius map $\varphi: L\to L$ involved in $(Z_{\mu}^{\mathrm{data}})_{\Fpbar}$ is simply the usual Frobenius $X\mapsto \Frob(X)=X^{(p)}$, hence inducing the identity map on the Weyl group  of $L$. 
	
	\subsection{The case $n$ is even}\label{S:ZipdataNormEven}
	By the classification of quadratic forms over $\Fp$, there exists an $\Fp$-basis 
	$\mathfrak{B} = \{\delta_1,\delta_2,\ldots,\delta_{n+2}\}$ of $\Lambda_{\Fp}$ with respect to which $q$ has one of the following forms:
	\begin{equation}\label{Eq:StdFormEven}
		q =
		\begin{cases}
			x_1x_{n+2} + x_2x_{n+1} + \cdots + x_m x_{m+1}, & 
			\text{(standard split form)},\\[2pt]
			x_1x_{n+2} + \cdots + x_{m-1}x_{m+2} + x_m^2 - c x_{m+1}^2, &
			c \in \mathbb{F}_p^{\times} \setminus \mathbb{F}_{p}^{\times 2}.
		\end{cases}
	\end{equation}
	
	Clearly $\SO(\Lambda)_{\Fp}$ is split over $\Fp$ if and only if we are in the first case\footnote{Equivalently, $\SO(\Lambda_{\Fp})$ is split over $\Fp$ if and only if the Legendre symbols satisfy $\bigl(\tfrac{-1}{p}\bigr)^m = \bigl(\tfrac{\det(q)}{p}\bigr)$.}. In the split case we use $\mathfrak{B}$ to obtain the normalization isomorphism
	$\SO(\Lambda)_{\Fp} \cong \SO(n,2)_{\Fp}$.
	The normalized conjugacy class $[\mu]_{\Fpbar}$ on $\SO(n,2)_{\Fpbar}$ and the normalized $G$-zip datum $(Z_{\mu}^{\data})_{\Fpbar}$ admit the same description as in the case $n$ odd, after redefining $J_{n+2}$ there to be $\antidiag(1/2,\ldots,1/2)$. In particular, here we have chosen $\mu$ as in \eqref{Eq:MuNorm}, which is again defined over $\Fp$ already. 
	
	In the nonsplit case we pass to $\Fpsq$ to split $q$.
	Consider $W = \langle \delta_m,\delta_{m+1} \rangle$ with $q_m := x_m^2 - c\,x_{m+1}^2$ on $W$.
	Choose $\sqrt{c} \in \mathbb{F}_{p^2}^\times$ and set
	\begin{equation}\label{Eq:ChangBasSplit}
		\delta_m' := \tfrac12\!\Bigl(\delta_m + \frac{\delta_{m+1}}{\sqrt{c}}\Bigr),
		\qquad
		\delta_{m+1}' := \tfrac12\!\Bigl(\delta_m - \frac{\delta_{m+1}}{\sqrt{c}}\Bigr).
	\end{equation}
	In the coordinates $(x_m,x_{m+1})$ relative to the basis $(\delta_m',\delta_{m+1}')$ we have $q_m = x_m x_{m+1}$.
	Consider the $\Fpsq$-basis
	$\mathfrak{B}^{\new} := \{\delta_1,\ldots,\delta_{m-1},\delta_m',\delta_{m+1}',\ldots,\delta_{n+2}\}$
	of $\Lambda_{\Fpsq}$, and use it to obtain the normalization
	$\SO(\Lambda)_{\Fpsq} \cong \SO(n,2)_{\Fpsq}$. 
	We then still have $\SO(n,2)_{\Fpsq} = \SO(J_{n+2})_{\Fpsq}$. Moreover, the transported conjugacy class $[\mu]_{\Fpbar}$ and the parabolic $P$ in the normalized $G$-zip datum $(Z_{\mu}^{\data})_{\Fpbar}$ admit a description analogous to the case $n$ odd. As in the split case, here we again have chosen $\mu$ as in \eqref{Eq:MuNorm}, which is defined over $\Fp$ already. 
	
	However, since the change of basis from $\mathfrak{B}$ to $\mathfrak{B}^{\new}$ is defined over $\Fpsq$, the transported Frobenius map $\varphi: \SO(n,2)_{\Fpbar} \to \SO(n,2)_{\Fpbar}$ becomes the twisted one
	\begin{equation}\label{Eq:TwistFrobL}
		X \longmapsto h\,X^{(p)} h^{-1},
		\qquad
		h = \diag(1,\ldots,1,\antidiag(1,1),1,\ldots,1) \in \GL_{n+2}(\Fp),
	\end{equation}
	where $\antidiag(1,1)$ occupies the central $m$-th and $(m+1)$-st rows and columns.
	The isogeny $\varphi: L \to L$ is induced by this map and is given by the same formula. Here note that although $h$ does not belong to $\SO(n,2)(\Fpbar)$ for determinant reasons, it still acts by conjugation on $\SO(n,2)_{\Fpbar}$ and on $L$. Since $h$ normalizes $\mu$, we still have $Q = P_{\mu^{-1}}^{\varphi} = P_{\mu^{-1}}$ and $L^{\varphi} = L$ as subgroups of $\SO(n,2)_{\Fpbar}$.
	
	To summarize, in both the split and nonsplit cases, we obtain the same normalized geometric reductive pair
\[
(\SO(n,2)_{\Fpbar},\mu),
\]
where $\mu$ is given by $t \mapsto \diag(t,1,\ldots,1,t^{-1})$, and where $\SO(n,2)_{\Fpbar}=\SO(J_{n+2})_{\Fpbar}$. Thus, we obtain a uniform description of the triple $(\SO(n,2)_{\Fpbar}, P, Q)$ in the normalized zip group datum
\[
(Z_{\mu}^{\data})_{\Fpbar}=(\SO(n,2)_{\Fpbar}, P, Q,\,\varphi\colon L\to L).
\]
Indeed, they are determined completely by $(\SO(n,2)_{\Fpbar},\mu)$. In particular, the Weyl group $W$ of $\SO(n,2)_{\Fpbar}$, the subgroup $W_{\mu}\subseteq W$, and the corresponding sets of simple reflections for $L\subseteq \SO(n,2)_{\Fpbar}$ do not depend on the splitting behavior of $q$, when computed with respect to $(B_{\SO},T_{\SO})$, where $B_{\SO}=B\cap \SO(n,2)_{\Fpbar}$ and $T_{\SO}=T\cap \SO(n,2)_{\Fpbar}$. Here $B$ is the standard upper triangular Borel subgroup and $T$ is the diagonal maximal torus of $\GL_{n+2,\Fpbar}$. However, there is a clear \emph{distinction} for the ``Frobenius'' isogeny $\varphi\colon L\to L$. We shall see in \S~\ref{S:Frames} that the different isogenies $\varphi\colon L\to L$ lead to different choices of frames for $(Z_{\mu}^{\data})_{\Fpbar}$.
	
	\section[]{Weyl groups and minimal length representatives (orthogonal)}\label{S:WeylEOIndex}
	
	Recall that we have obtained normalized reductive pairs $(\SO(n,2)_{\Fpbar}, \mu)$ in \S\ref{S:ZipdataNormOdd} and \S\ref{S:ZipdataNormEven}. In this section we collect the standard permutation descriptions of the Weyl group $W$ of $\SO(n,2)_{\Fpbar}$, the Weyl subgroup $W_{\mu}$ of $L = \mathrm{Cent}(\mu) \subseteq \SO(n,2)_{\Fpbar}$, and the set ${}^{\mu}W$ of the minimal length representatives for the coset space $W_{\mu}\backslash W$, following \cite{bruh}.

	\subsection{The Weyl groups of $\SO(n, 2)_{\Fpbar}$}\label{S:WeylGpSO}

	\subsubsection{The case $n$ is odd}\label{S:WeylOdd}
	In this case, $m=(n+1)/2$ and the maximal torus $T_{\SO}$ consists of diagonal matrices of the form $\diag(t_1, \ldots, t_{m}, 1, t_{m}^{-1}, \ldots, t_1^{-1})$. With respect to the Borel pair $(B_{\SO}, T_{\SO})$, the Weyl group $W$ of $\SO(n, 2)_{\Fpbar}$ is identified with the following group of permutations,
	$$W=\{\sigma\in S_{n+2}\mid\sigma(i)+\sigma(n+3-i)=n+3,\ \forall\ 1\leq i\leq n+2\}.$$
	The simple reflections are then given by $\{s_i\}$, $i=1,2,\cdots, m$, where
	\begin{equation*}
		s_i=\left\{ \begin{aligned}
			(i,i+1)(n+2-i,n+3-i), \ \ \ \ \ i< m; \\
			(m,m+2),\ \ \ \ \ \ \ \ i=m.
		\end{aligned} \right.
	\end{equation*} 
	For the longest element $w_0\in W$, we have $w_0=(1, n+2)(2, n+1)\ldots (m, m+2)$.
	
	\noindent \textbf{Matrix representative}\label{S:MatRep}
	We choose $P_{s_i}\in\SO(n, 2)(\Fpbar)$ as follows. For $1\le i<m$: $P_{s_i}$ is the identity except it swaps the $i$-th and $(i{+}1)$-st
	coordinates and simultaneously swaps the $(n{+}2{-}i)$-th and $(n{+}3{-}i)$-th coordinates.
	For $i=m$: $P_{s_m}$ is the identity except on indices $(m,m{+}1,m{+}2)$, where it is given by
	\[
	\begin{pmatrix}0&0&1\\[2pt]0&-1&0\\[2pt]1&0&0\end{pmatrix}
	\quad\text{(i.e. swap $m\leftrightarrow m{+}2$ and send $m{+}1$ to $-\, (m{+}1)$).}
	\]
	The minus sign here is to guarantee $\det(P_{s_m})=1$. For any reduced word
	$w=s_{i_1}\cdots s_{i_r}$ in $W$, set $P_w:=P_{s_{i_1}}\cdots P_{s_{i_r}}$; this is well-defined
	and $P_{w_1w_2}=P_{w_1}P_{w_2}$, for all $w_1, w_2\in W$.
	
	\subsubsection{The case $n$ is even}
	In this case, $m=n/2+1$, and the maximal torus $T_{\SO}$ consists of diagonal matrices of the form $\diag(t_1, \ldots, t_{m}, t_{m}^{-1}, \ldots, t_1^{-1})$. With respect to $(B_{\SO}, T_{\SO})$, the Weyl group $W$ of $\SO(n, 2)_{\Fpbar}$ is identified with:
	\begin{equation*}
		W=\left\{ \begin{aligned}
			\sigma\in S_{n+2}\mid\ &\sigma(i)+\sigma(n+3-i)=n+3, \forall\ 1\leq i\leq n+2;\\
			&\ \ \ \text{and }\#\{1\leq i\leq m\mid \sigma(i)>m\}\text{ is even.}
		\end{aligned} \right\}
	\end{equation*}
	Its simple reflections are given by $\{s_i\}$, $i=1,2,\cdots, m$, where
	\begin{equation*}
		s_i=\left\{ \begin{aligned}
			(i,i+1)(n+2-i,n+3-i), \ \ \ \ \ i< m; \\
			(m-1,m+1)(m,m+2),\ \ \ \ \ \ i=m.
		\end{aligned} \right.
	\end{equation*}
	For the longest element $w_0\in W$, we have
	\[w_0=\begin{cases}
		(1, n+2)\ldots (m, m+2)& \text{ if } m\equiv 0 \pmod{4}\\
		(1, n+2)\ldots (m-1, m+3)& \text { if } m\equiv  2 \pmod{4}. 
	\end{cases}\] 
	
	\noindent \textbf{Matrix representatives} We also fix the matrix representatives $P_{s_i}\in \SO(n, 2)_{\Fpbar}$ as well as $P_{w_i}$ in $W$ as in the odd case: for $i=m$, we take $P_{s_m}$ to be 
	\[
	\diag(1,  \ldots, 1, A, 1, \ldots, 1), \text{ with } A =\begin{pmatrix}
		0&0&1&0\\
		0&0&0&1\\
		1&0&0&0\\
		0&1&0&0
	\end{pmatrix}.
	\]

	\subsection{Minimal coset representatives for $W_{\mu}\backslash W$}
	The setup is as in \S~\ref{S:WeylGpSO}. Recall that $W_{\mu}\subseteq W$ is the Weyl group of $L$ with respect to the Borel pair $(B_{\SO}\cap L, T_{\SO})$. It follows from the matrix description of $L$,
	\[
	L = \bigl\{\diag(t,A,t^{-1}) \mid t \in \mathbb{G}_m,\ A \in \SO(n-2,2)\bigr\},
	\]
	together with the description of $W$ in \S~\ref{S:WeylGpSO}, that
	\(
	W_{\mu} = \langle s_2,\ldots,s_m \rangle \subseteq W,
	\)
	since the $\mathbb{G}_m$ factor in $L$ does not contribute to the Weyl group. The description of ${}^\mu W$ below is standard.

	\subsubsection{The case $n$ is odd:} \label{S:EOindexOdd} For $0\leq i\leq n$, let
	\begin{equation*}
		w_i=\left\{ \begin{aligned}
			s_1s_2\cdots s_i, \ \ \ \ \ i\leq m; \\
			s_1\cdots s_ms_{m-1}\cdots s_{n-i+1},\ \ \ \ \ \ i>m.
		\end{aligned} \right.
	\end{equation*}
	Then ${}^\mu W=\{w_i\}_{0\leq i\leq n}$. In particular, for each integer $i$ with $0\leq i\leq n$, there is a unique element in ${}^\mu W$, namely $w_i$, of length $\ell(w_i)=i$. 
	%Under the order-preserving bijection ${}^\mu W \to {}^\mu W$ in Lemma \ref{Lem:Swp}, each $w_i $ is exchanged with $w_{n+1-i}$. 
	
	\subsubsection{The case $n\geq 2$ is even} \label{S:EOindexEven}
	For $0\leq i\leq n$, let
	\begin{equation*}
		w_i=\left\{ \begin{aligned}
			s_1s_2\cdots s_i, \ \ \ \ \ i\leq m; \\
			s_1\cdots s_ms_{m-2}\cdots s_{n-i+1},\ \ \ \ \ \ i>m.
		\end{aligned} \right.
	\end{equation*}
	Let $w^{\dagger}_{m-1}=s_1\cdots s_{m-2}s_m$, then ${}^\mu W=\{w_i\}_{0\leq i\leq n}\cup \{w^{\dagger}_{m-1}\}$. Here the convention is that when $m=2$, $w_1^{\dagger}=s_2$. For each $i\neq m-1$, $w_i\in {}^\mu W$ is the unique element with $\ell(w_i)=i$, while both $w_{m-1}$ and $w_{m-1}^{\dagger}$ have length $m-1$.
	%Under the order-preserving bijection ${}^\mu W \to {}^\mu W$ in Lemma~\ref{Lem:Swp}, each $w_i$ is exchanged with $w_{n+1-i}$ for $i\neq m-1$, and $w_{m-1}$ swaps with $w_{m-1}'$.
	
	\section{Choosing frames $Z_{\mu}^{\mathrm{data}}$} \label{S:Frames}
	
	Assume that $n \geq 1$ throughout this section. In the next two lemmas we choose frames for the normalized geometric zip datum 
	\(
	(Z_{\mu}^{\mathrm{data}})_{\Fpbar}
	= \bigl(\SO(n,2)_{\Fpbar}, P, Q, \varphi:L \to L\bigr),
	\)
	obtained in \S~\ref{S:ZipdataNormOdd} and $\S~\ref{S:ZipdataNormEven}$.
	
	\subsection{The case $n$ is odd}
    
\begin{lemma} \label{Lem:TwistorOdd}
    With notations as in \S~\ref{S:WeylGpSO} ($n$ odd), the triple $(B_{\SO}, T_{\SO}, g:=P_{w_{0, \mu}w_0})$ is a frame for the normalized $G$-zip data $(Z_{\mu}^{\data})_{\Fpbar}$ in \S~\ref{S:ZipdataNormOdd}. 
\end{lemma}

\begin{proof}
    To show that $(B_{\SO}, T_{\SO}, g)$ is a frame, we must verify that $\varphi(B_{\SO}\cap L) = g^{-1}B_{\SO}g\cap L$. 
    Recall that for the orthogonal group in the odd case, the isogeny $\varphi: L \to L$ is the standard Frobenius, which acts trivially on the Weyl group representatives. 
    Since $g$ normalizes $L$, it suffices to show:
    \begin{equation}\label{Eq:Framecheck}
        B_{\SO}\cap L = g^{-1}(B_{\SO}\cap L)g.
    \end{equation}
    Let $B^-\subseteq \GL_{n+2, \Fpbar}$ be the lower triangular Borel and set $B_{\SO}^-=B^-\cap \SO(n, 2)$. The pair $(B_{\SO}^-\cap L, T_{\SO})$ is the Borel pair of $L$ opposite to $(B_{\SO}\cap L, T_{\SO})$. 
    By the properties of the longest elements $w_0 \in W$ and $w_{0,\mu} \in W_\mu$, we have:
    \[
    g^{-1}(B_{\SO}\cap L)g = P_{w_0} \left( P_{w_{0,\mu}} (B_{\SO}\cap L) P_{w_{0,\mu}} \right) P_{w_0} = P_{w_0} (B_{\SO}^-\cap L) P_{w_0} = B_{\SO}\cap L.
    \]
    The first equality follows because $P_{w_{0,\mu}}$ maps $B_{\SO} \cap L$ to its opposite $B_{\SO}^- \cap L$, and $P_{w_0}$ maps that back to $B_{\SO} \cap L$ within the Levi $L$.
\end{proof}

	\subsection{The case {$n$} is even}
	\begin{lemma}\label{Lem:TwistorEven}
		With notations as in \S~\ref{S:WeylGpSO} (where $n$ is even),
		the triple $(B_{\SO}, T_{\SO}, g := P_{w_{0,\mu}w_0})$ is a frame for the normalized $G$-zip datum $Z_{\mu}^{\data}$ in \S~\ref{S:ZipdataNormEven}. 
	\end{lemma}
	\begin{proof}
		In the split case, i.e.\ when $\SO(\Lambda)_{\Fp}$ is split over $\Fp$, the same argument as in the odd-dimensional case applies. Suppose now that $\SO(\Lambda)_{\Fp}$ is nonsplit over $\Fp$. By the proof of Lemma~\ref{Lem:TwistorOdd}, the equality in \eqref{Eq:Framecheck} still holds. Hence it remains only to show that $\varphi(B_{\SO}\cap L)=B_{\SO}\cap L$, where $\varphi:L\to L$ is given by $X\mapsto hX^{(p)}h^{-1}$ as in \eqref{Eq:TwistFrobL}, with $h$ defined there. It suffices to check this on $\Fpbar$-points. Thus we are reduced to proving that
\(
h(B_{\SO}\cap L)h^{-1}=B_{\SO}\cap L,
\)
that is, that conjugation by $h$ preserves $B_{\SO}\cap L$, and hence induces an automorphism of it. Since $h$ normalizes $L$, it is enough to show that for every $\Fpbar$-point $X$ of $B_{\SO}\cap L$, the matrix $hXh^{-1}$ is upper triangular.

		Since $hX h^{-1}$ is obtained from $X$ by switching the $m$-th and $(m+1)$-st rows and the $m$-th and $(m+1)$-st columns simultaneously, we are reduced to investigating the central $2\times 2$ block in rows and columns $m,m+1$,
		\[
		\begin{pmatrix}
			a_{m,m}   & a_{m,m+1}\\
			0         & a_{m+1,m+1}
		\end{pmatrix}.
		\]
		But this matrix must satisfy $a_{m,m+1} = 0$ and $a_{m,m}\, a_{m+1,m+1} = 1$, as follows from the relation $X^T J_{n+2}X = J_{n+2}$. 
		We can now conclude, because the $h$-conjugation turns $\diag(a_{m,m}, a_{m+1,m+1})$ into $\diag(a_{m+1,m+1}, a_{m,m})$, which is still upper triangular.
	\end{proof}
	
	\begin{remark}\label{Rmk:WeylGpAuto}
		With the above notation, one has:
		\begin{enumerate}
			\item When $n$ is odd, one computes that $w_{0,\mu}w_0 = (1,n+2)$, whereas when $n$ is even, $w_{0,\mu}w_0 = (1,n+2)(m,m+2)$, with $m = n/2 + 1$.
			\item When $n$ is odd or $n$ is even and $\SO(\Lambda)_{\Fp}$ is split, the isogeny $\varphi: L \to L$ is the usual Frobenius and induces the identity map on $W_{\mu}$. When $n$ is even and $\SO(\Lambda)_{\Fp}$ is nonsplit, the twisted Frobenius $\varphi$ on $\SO(n,2)_{\Fpbar}$ and on $L$ induces an automorphism $\psi \colon W \to W$, which restricts to an automorphism $\psi \colon W_{\mu} \to W_{\mu}$. This automorphism interchanges $s_{m-1}$ and $s_m$ and fixes all other simple reflections. Consequently, its restriction to ${}^\mu W$ induces a bijection ${}^\mu W \to {}^\mu W$ that exchanges $w_{m-1}$ and $w^{\dagger}_{m-1}$ and fixes all other elements.
		\end{enumerate}
	\end{remark}

\section{EO stratification for $\Sh_{\SO(n, 2), \Fpbar}$}\label{S:EOOrthog}
\subsection{Partial order on {${}^\mu W$}}\label{S:ParOrd}
	We work in the setting of \S~\ref{S:RedPairNorm}: Consider a normalized geometric reductive pair $(\SO(n,2)_{\Fpbar},\mu)$ there and its associated\footnote{We note again that $(Z_{\mu}^{\data})_{\Fpbar}$ is not entirely determined by  $(\SO(n,2)_{\Fpbar},\mu)$.} normalized geometric $G$-zip datum $(Z_{\mu}^{\data})_{\Fpbar}$, equipped with the chosen frame of \S~\ref{S:Frames}, in one of the three cases (either $n$ is odd, or $n$ is even with $\SO(\Lambda)_{\Fp}$ split or nonsplit). In each case, the isogeny $\varphi \colon L \to L$ in $(Z_{\mu}^{\data})_\Fpbar$ induces an automorphism $\psi \colon W_{\mu} \to W_{\mu}$, described explicitly in Remark~\ref{Rmk:WeylGpAuto}. We consider the tuple
	\[
	(W, W_{\mu}, {}^\mu W, \psi \colon W_{\mu} \to W_{\mu}),
	\]
	 which is called an
	\emph{abstract zip datum of Coxeter type} in \cite[Def.~9.10]{PWZ}.
	\begin{theorem}[{\cite[Def.~6.1, Cor.~6.3]{PWZ}}]\label{Thm:PartOrd}
		Define a relation on ${}^\mu W$ as follows:
		for $w,w' \in {}^\mu W$, 
		\[
		w \preceq w' \iff \exists\, v \in W_{\mu} \text{ such that } v^{-1} w \psi(v) \le w',
		\]
		where $\le$ denotes the Bruhat order on $W$. Then $\preceq$ is a partial order. 
	\end{theorem}
This order was introduced in \cite{WedhornSpecFzip} in the general framework of \cite{PWZ} (cf.\ \cite{HeGstable}). In the special case \((\SO(n,2)_{\Fpbar},\mu)\), it is given explicitly by
\begin{equation}
\begin{aligned}
1 &\preceq w_1 \preceq \cdots \preceq w_n,
&& \text{if } n \text{ is odd},\\
1 &\preceq w_1 \preceq \cdots \preceq w_{m-2} \preceq w_{m-1},\, w_{m-1}^{\dagger}
   \preceq w_m \preceq \cdots \preceq w_n,
&& \text{if } n \text{ is even},
\end{aligned}
\end{equation}
and there is no partial order relation between \(w_{m-1}\) and \(w_{m-1}^{\dagger}\).

	\subsection{Topological space of $\SO(n,2)_{\Fpbar}\text{-}\mathsf{Zip}^{\mu}$}
	With notation as in \S~\ref{S:ParOrd} and following \cite{PWZ2}, for each $w \in {}^\mu W$ we define the \emph{framed zip orbit}
	\begin{equation}\label{Eq:ZmuOrbFram}
		O_{\fram}^{w} := E_{\mu}(gw) \subseteq \SO(n,2)(\Fpbar).
	\end{equation}
	The following theorem specializes Theorem~\ref{Thm:ZipstackTopWeyl} to the case of \(\SO(n,2)_{\Fpbar}\text{-}\mathsf{Zip}^\mu\).
	\begin{theorem}\label{Thm:ZipStackTop}
		The assignment $w \mapsto O_{\fram}^{w}$ induces a homeomorphism
		\[
		\alpha \colon ({}^\mu W,\, \preceq) \xrightarrow{\;\cong\;} 
		E_{\mu}(\Fpbar)\backslash \SO(n,2)(\Fpbar)
		\;\cancong\;
		\bigl|\SO(n,2)_{\Fpbar}\text{-}\mathsf{Zip}^{\mu}\bigr|.
		\]
		Equivalently, $\alpha$ yields a bijection between ${}^\mu W$ and
		$E_{\mu}(\Fpbar)\backslash \SO(n,2)(\Fpbar)$ such that, for $w,w' \in {}^\mu W$,
		one has $O_{\fram}^{w'} \subseteq \overline{O_{\fram}^{w}}$ if and only if $w' \preceq w$. Here $ \overline{O_{\fram}^{w}}$ denotes the closure of  $O_{\fram}^{w}$. 
	\end{theorem}
	
	\medskip
	\noindent\textbf{Criterion for $O_{\fram}^{w} = O_{\fram}^{w'}$.}
	Clearly one can define the framed zip orbit $O_{\fram}^w:=E_{\mu}(gw)$ for all $w \in W$.
	Given $w,w' \in W$, there is a criterion in \cite[\S9]{PWZ} for when the framed zip orbits
	$O_{\fram}^w$ and $O_{\fram}^{w'}$ coincide; for our purposes we recall it in our setting.
	\begin{theorem}[\cite{PWZ}, Thm.~9.17 and Cor.~9.18]\label{Thm:Criteria}
		The equality $O_{\fram}^w = O_{\fram}^{w'}$ holds if and only if there exist $v \in W_{\mu}$ and
		$v' \in E_w$ such that $w' = v^{-1} w v' \psi(v)$, where
		$E_w \subseteq wW_{\mu} w^{-1}$ is the largest subgroup of $wW_{\mu} w^{-1}$ satisfying
		$\psi(w E_w w^{-1}) = E_w$.
		In particular, if $s \in W_{\mu}$ is a reflection (i.e.\ $s^2 = 1$) that commutes with $w$ and is
		fixed by $\psi$, then $\langle s \rangle \subseteq E_w$, and hence
		$O_{\fram}^{w} = O_{\fram}^{v^{-1} w s \psi(v)}$ for all $v \in W_{\mu}$.
	\end{theorem}

	\subsection{EO stratifications}
We now return to our geometric mod $p$ Shimura varieties $\Sh_{\SO(\Lambda), \Fpbar}$ and $\Sh_{\GSpin(\Lambda), \Fpbar}$ from \S~\ref{S:ShvGSpinSO}. Associated with them is the reductive triple $(\SO(\Lambda)_{\Fp}, \SO(\Lambda)_{\kappa}, [\mu]_{\kappa})$ over $\kappa$, where $\kappa = \Fp$ except in the case $n=0$, where $\kappa$ may be $\Fpsq$. By means of the normalization isomorphism $\SO(n, 2)_{\Fpbar} \cong \SO(\Lambda)_{\Fpbar}$ fixed in \S~\ref{S:NormPair}, we obtain the geometric normalized $G$-zip datum $(Z_{\mu}^{\data})_{\Fpbar}$. Recall that frames for this datum were chosen in \S~\ref{S:Frames}. With the normalization isomorphism, we abuse notation and write $\Sh_{\SO(n, 2), \Fpbar}$ for $\Sh_{\SO(\Lambda), \Fpbar}$. Similarly, we also write $\Sh_{\GSpin(n, 2), \Fpbar}$ for $\Sh_{\GSpin(\Lambda), \Fpbar}$. Thanks to the work \cite{SZ22}, we have the following:
	
	\begin{theorem}\label{Thm:EOAbel}
		With notation as above, the following hold.
		\begin{enumerate}
			\item There is a smooth map of algebraic stacks over $\Fpbar$ (the zip period map)
			\[
			\zeta \colon \Sh_{\SO(n, 2),\Fpbar}
			\longrightarrow \SO(n,2)_{\Fpbar}\text{-}\mathsf{Zip}^{\mu}
			\;\cancong\; [E_{\mu}\backslash \SO(n,2)_{\Fpbar}].
			\]
			\item Let $\alpha \colon {}^\mu W \xrightarrow{\sim} E_{\mu}(\Fpbar)\backslash \SO(n,2)(\Fpbar)$ be as in Theorem~\ref{Thm:ZipStackTop}. For each $w \in {}^\mu W$, let $\Sh_{\SO(n,2),\Fpbar}^{w}$ be the preimage of the single-point substack
			\(
			[E_{\mu}\backslash E_{\mu}\alpha(w)] \subseteq [E_{\mu}\backslash \SO(n,2)_{\Fpbar}].
			\)
			Then $\Sh_{\SO(n,2),\Fpbar}^{w}$ is a smooth locally closed subscheme of $\Sh_{\SO(n,2),\Fpbar}$, of dimension $\ell(w)$. Moreover, for $w,w' \in {}^\mu W$,
			$\Sh_{\SO(n,2),\Fpbar}^{w'}$ is contained in the closure of $\Sh_{\SO(n,2),\Fpbar}^{w}$
			if and only if $w' \preceq w$.
		\end{enumerate}
	\end{theorem}
	
	The resulting decomposition,
	\(
	\Sh_{\SO(n,2),\Fpbar} = \bigsqcup_{w \in {}^\mu W} \Sh_{\SO(n,2),\Fpbar}^{w},
	\)
	is called the \emph{EO stratification} of $\Sh_{\SO(n,2),\Fpbar}$. By pullback along the canonical projection map \(\Sh_{\GSpin(n, 2), \Fpbar}\to \Sh_{\SO(n,2),\Fpbar}\), we obtain the EO stratification of \(\Sh_{\GSpin(n, 2), \Fpbar}\), which is of Hodge type and was studied in \cite{EOZhang}.

	\section{The proof of Theorem \ref{Thm:MainThmOrthog}}\label{EOunderKudlaEmbed}
	\subsection{Strategy}
    Recall that Theorem \ref{Thm:MainThmOrthog} determines the image EO stratum of a given EO stratum of $\Sh_{\GSpin(n-1,2),\overline{\mathbb{F}}_{p}}$ under the natural embedding into $\Sh_{\GSpin(n,2),\overline{\mathbb{F}}_{p}}$. To prove this, our approach is to use the functoriality of zip period maps (\S~\ref{S:ZipPerMap}). Specifically, the map of reductive pairs
	\(
	\iota \colon (\SO(\Lambda')_{\Fpbar},\mu') \longrightarrow (\SO(\Lambda)_{\Fpbar},\mu)
	\)
	induces a morphism of zip stacks
	\begin{equation}\label{Eq:ZipMapSO}
		\iota_{\zip} \colon \SO(\Lambda')_{\Fpbar}\text{-}\mathsf{Zip}^{\mu'} 
		\longrightarrow \SO(\Lambda)_{\Fpbar}\text{-}\mathsf{Zip}^{\mu}.
	\end{equation}
	Assume we have chosen normalizations
	$\SO(\Lambda')_{\Fpbar} \cong \SO(n-1,2)_{\Fpbar}$ and
	$\SO(\Lambda)_{\Fpbar} \cong \SO(n,2)_{\Fpbar}$ as in
	\S~\ref{S:ZipdataNormOdd} and \S~\ref{S:ZipdataNormEven}. The embedding $\iota$ then transports to a normalized embedding
	\[
	\iota \colon (\SO(n-1,2)_{\Fpbar},\mu') \longrightarrow (\SO(n,2)_{\Fpbar},\mu).
	\]
	
	Choosing frames for the associated $G$-zip data as in \S~\ref{S:Frames}, the underlying
	topological spaces on the two sides of \eqref{Eq:ZipMapSO} are identified with the partially
	ordered sets ${}^{\mu'}W$ and ${}^{\mu}W$ (Theorem~\ref{Thm:ZipStackTop}), which admit explicit
	descriptions in terms of permutations. Thus our problem is equivalent to the following:
	for each Weyl group element $v \in {}^{\mu'}W'$, corresponding to a framed $E_{\mu'}$-orbit
	$O_{\fram}^{v} \subseteq \SO(n-1,2)(\Fpbar)$, determine the element $w \in {}^{\mu}W$ such that
	the corresponding framed $E_{\mu}$-orbit $O_{\fram}^{w} \subseteq \SO(n,2)(\Fpbar)$ satisfies
	\(
	\iota\bigl(O_{\fram}^{v}\bigr) \subseteq O_{\fram}^{w}.
	\)
	Here and below, a prime indicates passage from $\SO(\Lambda)$ to $\SO(\Lambda')$;
	in particular, $W'$ denotes the Weyl group of $\SO(\Lambda')$.
	
	\subsection{When $n$ is even}
	In this case $m'=m-1=n/2$.
	We may assume that, for an $\Fp$-basis $\mathfrak{B} = \{\delta_1,\ldots,\delta_{2m}\}$ of $\Lambda_{\Fp}$, with $\mathfrak{B}' = \mathfrak{B} \setminus \{\delta_{m+1}\}$, the pair of quadratic forms $(q',q)$ is given by
	\[
	\begin{aligned}
		q' &= x_1x_{n+2}+\cdots+x_{m-1}x_{m+2}+c\,x_m^2,\quad 
		q  = q' + d\,x_{m+1}^2, 
	\end{aligned}
	\]
	where $c,d \in \mathbb{F}_p^{\times}$. Further, we may assume that $(c,d)\in\{(1,1),(u,u),(1,u),(u,1)\}$, where $u\in\mathbb{F}_p^{\times}$ is a fixed nonsquare, since up to isometry there are only four possibilities for the pair of quadratic forms $(q',q)$.  
	
	After base change to $\Fpsq$, the form $q$ becomes split: choose $r,s \in \mathbb{F}_{p^2}^\times$ with
	$r^2=c$ and $s^2=d$, and set
	\[
	\delta_m' := \tfrac{1}{2rs}\bigl(s\delta_m + r\delta_{m+1}\bigr),
	\qquad
	\delta_{m+1}' := \tfrac{1}{2rs}\bigl(-s\delta_m + r\delta_{m+1}\bigr).
	\]
	Then the quadratic form $q$ on $\Lambda_{\Fpsq}$ becomes the standard split form in \eqref{Eq:StdFormEven},
	with respect to the basis
	\[
	\mathfrak{B}^{\new}
	:= \{\delta_1,\ldots,\delta_{m-1},\delta_m',\delta_{m+1}',\delta_{m+2},\ldots,\delta_{2m}\}
	\]
	of $\Lambda_{\Fpsq}$. Let 
	\[
	(\SO(\Lambda')_{\Fpbar},\mu')
	\cong (\SO(n-1,2)_{\Fpbar},\mu'),\quad (\SO(\Lambda)_{\Fpbar},\mu)
	\cong (\SO(n,2)_{\Fpbar},\mu),\]
    be normalization isomorphisms with respect to $\mathfrak{B}'$ and $\mathfrak{B}$ respectively, where $\mu'$ and $\mu$ on the right hand sides are obtained by transport of structures and are both in the standard form
	of \eqref{Eq:MuNorm}. Via transport of structure, the embedding
	\(
	\iota \colon (\SO(\Lambda')_{\Fpbar},\mu') \to (\SO(\Lambda)_{\Fpbar},\mu)
	\)
	induces an embedding
	\[
	\iota\colon (\SO(n-1,2)_{\Fpbar},\mu')
	\longrightarrow
	(\SO(n,2)_{\Fpbar},\mu).
	\]
	To explicitly illustrate this embedding for the reflection elements discussed below, an element in $\SO(n-1,2)_{\Fpbar}$ acting as $-1$ on $\delta_m$ maps under $\iota$ to an element swapping the split basis vectors $\delta_m'$ and $\delta_{m+1}'$. In block matrix form, this is given by:
\begin{equation}\label{Eq:MatrEmbOdd2Even}
	\begin{pmatrix}
		A & 0 & B \\[2pt]
		0 & -1 & 0 \\[2pt]
		C & 0 & D
	\end{pmatrix}
	\longmapsto
	\begin{pmatrix}
		A & 0 & 0 & B \\[2pt]
		0 & 0 & 1 & 0 \\[2pt]
		0 & 1 & 0 & 0 \\[2pt]
		C & 0 & 0 & D
	\end{pmatrix},\qquad A,B,C,D\in \mathrm{M}_{m'}.
\end{equation}

	Let $(B_{\SO},T_{\SO},g)$ be the frame for the normalized $G$-zip datum on $(\SO(n,2)_{\Fpbar},\mu)$ determined by
	the $G$-zip attached to $(\SO(\Lambda)_{\Fp},\mu)$ as in \S~\ref{S:ZipdataNormEven}, and
	let $(B'_{\SO},T'_{\SO},g')$ be the corresponding frame for the normalized $G$-zip data on $(\SO(n-1,2)_{\Fpbar},\mu')$ determined
	by the $G$-zip attached to $(\SO(\Lambda')_{\Fp},\mu')$ as in \S~\ref{S:ZipdataNormOdd}. Then,
	\[
	\begin{aligned}
		g' &= P_{(1,n+2)}
		= \antidiag\bigl(1,\diag(1,\ldots,1,-1,1,\ldots,1),1\bigr)
		\in \SO(n-1,2)(\Fpbar),\\
		g  &= P_{(1,n+2)(m,m+1)}
		= \antidiag\bigl(1,\diag(1,\ldots,1,\antidiag(1,1),1,\ldots,1),1\bigr)
		\in \SO(n,2)(\Fpbar).
	\end{aligned}
	\]
	
	\textbf{Claim:} We have $\iota(g') = g$. Indeed, $g'$ corresponds to the automorphism of
	$\Lambda'_{\Fpbar}$ which switches $\delta_1$ with $\delta_{2m}$ and sends $\delta_m$ to $-\delta_m$,
	fixing the other $\delta_i$. But $g$ corresponds to the automorphism of $\Lambda_{\Fpbar}$ which
	switches $\delta_1$ with $\delta_{2m}$ and $\delta_m'$ with $\delta_{m+1}'$, again fixing the other
	$\delta_i$; this coincides with $\iota(g')$.
	
	For each $i \leq m'-1$ we have $\iota(P_{s'_i}) = P_{s_i}$. For $i = m'$ we have $\iota(P_{s'_{m'}}) = P_{(m-1,m+2)(m,m+1)} = P_{s_{m'}s_m}$. Indeed, $P_{s'_{m'}}$ corresponds to the automorphism of $\Lambda_{\Fpbar}$ which sends $\delta_m$ to $-\delta_m$ and swaps $\delta_{m-1}$ with $\delta_{m+2}$ (this coincides with our choice of matrix representative in \S~\ref{S:WeylOdd} as $\mathfrak{B}'$ is obtained from $\mathfrak{B}$ by deleting $\delta_{m+1}$ without relabeling).
	Using these expressions for $\iota(s'_i)$ and the relation $\iota(g') = g$, we obtain
	\[
	\iota(O_{\fram}^{w'_i}) \subseteq O_{\fram}^{w_i} \quad \text{if } i < m',
	\qquad
	\iota(O_{\fram}^{w'_i}) \subseteq O_{\fram}^{w_{i+1}} \quad \text{if } i \geq  m'.
	\]
	This concludes the proof of Theorem~\ref{Thm:MainThmOrthog} in the case where $n$ is even.
	
	\subsection{When $n$ is odd}\label{pf th 1 even-1}
	In this case $m'=m=\frac{n+1}{2}$. 
	There are $n+1$ strata on both sides.
	\subsubsection{The case $\SO(\Lambda')_{\Fp}$ is split over $\Fp$}
	We may assume that, for an $\Fp$-basis $\mathfrak{B} = \{\delta_1,\ldots,\delta_{2m}\}$ of $\Lambda_{\Fp}$, with $\mathfrak{B}' = \mathfrak{B} \setminus \{\delta_{m+1}\}$, the pair of quadratic forms $(q',q)$ is given by
	\[
	\begin{aligned}
		q' &= x_1x_{n+2}+\cdots+ x_m x_{m+2},\quad 
		q  = q'  + c\,x_{m+1}^2,
	\end{aligned}
	\]
	where $c \in \mathbb{F}_p^{\times}$. In this situation, both $\SO(\Lambda')_{\Fp}$ and $\SO(\Lambda)_{\Fp}$ are \emph{split} over $\Fp$. The proof is similar to the even case, so we record only the necessary modifications.
	
	We use $\mathfrak{B}'$ and $\mathfrak{B}$ for the normalization isomorphisms of 
	$\SO(\Lambda')_{\Fpbar}$ and $\SO(\Lambda)_{\Fpbar}$ respectively (thus no change of basis is needed). The normalized embedding
	\(
	\iota:\SO(n-1,2)_{\Fpbar} \to  \SO(n,2)_{\Fpbar} \) 
	is given by
	\begin{equation}\label{Eq:MatrEmbEven2Odd}
		\begin{pmatrix}A&B\\ C&D\end{pmatrix}\longmapsto
		\begin{pmatrix}
			A&0&B\\[2pt]
			0&1&0\\[2pt]
			C&0&D
		\end{pmatrix},\qquad A,B,C,D\in \mathrm{M}_{m}.
	\end{equation}
	It follows immediately that $\iota(P_{s'_i}) = P_{s_i}$ for each $i \leq m-1$, and
	$\iota(P_{s'_m}) = P_{s_m s_{m-1} s_m}$.

	Let $(B'_{\SO},T'_{\SO},g')$ be the frame for the associated normalized geometric $G$-zip datum $(Z_{\mu}^{\data'})_{\Fpbar}$ on $\SO(n-1, 2)_{\Fpbar}$ as in Lemma \ref{Lem:TwistorEven} (split case). Then we have 
	$\iota(g') = P_{(1,n+2)(m,m+2)}\in \SO(n, 2)(\Fpbar)$; indeed, the basis $\mathfrak{B}'$ for $\Lambda'$ is obtained from
	$\mathfrak{B}$ by deleting $\delta_{m+1}$. On the other hand, let $(B_{\SO},T_{\SO},g)$ be the frame for the associated normalized geometric $G$-zip datum on $\SO(n, 2)_{\Fpbar}$ as in Lemma \ref{Lem:TwistorOdd}. Then we have $g = P_{(1,n+2)} =\iota(g')P_{s_m}$. Using this equality and the expressions for $\iota(P_{s'_i})$, we obtain,
	\begin{align*}
		\iota(O_{\fram}^{w'_i})
		&\subseteq 
		O_{\fram}^{s_m w_i}\mathrel{\overset{\text{claim}}{=}} \begin{cases}
			O_{\fram}^{w_m}& \text{ if } i= m-1,\\
			O_{\fram}^{w_i} & \text{ if } i< m-1,
		\end{cases}
		\qquad
		\iota(O_{\fram}^{w'^{\dagger}_{m-1}})
		\subseteq O_{\fram}^{s_m w_{m-2} s_m s_{m-1} s_m}
		= O_{\fram}^{w_m},\\
		\iota(O_{\fram}^{w'_{m+j}})
		&\subseteq O_{\fram}^{w_{m-2}s_m (s_{m-1}s_ms_{m-1})s_m s_{m-2}\ldots s_{m-(j+1)}}
		= O_{\fram}^{w_{m+j+1}}, \qquad j\geq 0,
	\end{align*}
	where we use the commutativity of $P_{w_{m-2}}$ with $P_{s_m}$ and $s_{m-1}s_ms_{m-1}$ and $s_m$, and where the first equality (the “claim”) is given below.
    
    \noindent \textbf{Proof of the claim.} Recall from Remark \ref{Rmk:WeylGpAuto} that the induced map $\psi: W_{\mu} \to W_{\mu}$ in this case is the identity map. Applying the criterion in Theorem~\ref{Thm:Criteria} with $v = v' = s_m$, we obtain
$O_{\fram}^{s_m w_i} = O_{\fram}^{w_i s_m} = O_{\fram}^{w_i}$ for all $i < m-1$, since $s_m$ clearly commutes with these $w_i$. Applying the same criterion for $i=m$ with $(v,v') = (s_m,1)$, we obtain that $O_{\fram}^{s_m w_{m-1}}=O_{\fram}^{w_{m}}$. This completes the proof of Theorem~\ref{Thm:MainThmOrthog} in the case where $n$ is odd and $\SO(\Lambda')_{\Fp}$ is split.
	
	\subsubsection{The case $\SO(\Lambda')_{\Fp}$ is nonsplit over $\Fp$}\label{PfMaiThmEven-2}
	We may assume that, for an $\Fp$-basis $\mathfrak{B}$ as in the split case, with $\mathfrak{B}'$ obtained from $\mathfrak{B}$ by deleting $\delta_{m+1}$, the pair of quadratic forms $(q',q)$ is given by
	\begin{equation*}
		\begin{aligned}
			\text{Case (I)}\quad
			&q' = x_1x_{n+2}+\cdots+x_{m-1}x_{m+3}+x_m^2 - u x_{m+2}^2,\quad 
			q = q' - x_{m+1}^2;\\[4pt]
			\text{Case (II)}\quad
			&q' = x_1x_{n+2}+\cdots+x_{m-1}x_{m+3}+x_m^2 - u x_{m+2}^2,\quad 
			q = q' + u\,x_{m+1}^2,
		\end{aligned}
	\end{equation*}
	where $u \in \mathbb{F}_p^{\times} \setminus (\mathbb{F}_p^{\times})^2$.
	Under the assumption that $\SO(\Lambda')_{\Fp}$ is nonsplit, these two cases exhaust, up to isometry, all possibilities for the pair $(\Lambda'_{\Fp},\Lambda_{\Fp})$ by discriminant considerations.
	In what follows we treat Case~(I) in detail; Case~(II) is analogous.
	
	After base change to $\Fpsq$, we can arrange that the forms $q'$ and $q$ become split.
	Indeed, choose $r \in \mathbb{F}_{p^2}^\times$ with $r^2 = u$ and set
	\[
	\delta_m^{\new} := \tfrac12\bigl(\delta_m + r^{-1}\delta_{m+2}\bigr),\qquad
	\delta_{m+2}^{\new} := \tfrac12\bigl(\delta_m - r^{-1}\delta_{m+2}\bigr).
	\]
	With respect to the new bases $\mathfrak{B}^{\new}$ and $\mathfrak{B}'^{\new}$, obtained from $\mathfrak{B}$ and $\mathfrak{B}'$ by replacing $\delta_m,\delta_{m+2}$ with $\delta_m^{\new},\delta_{m+2}^{\new}$, the forms $q$ on $\Lambda_{\Fpsq}$ and $q'$ on $\Lambda'_{\Fpsq}$ become the standard split forms of \eqref{Eq:StdFormEven} and \eqref{Eq:StdFormOdd}, respectively.
	We use $\mathfrak{B}'^{\new}$ and $\mathfrak{B}^{\new}$ for the normalization isomorphisms of $\SO(\Lambda')_{\Fpbar}$ and $\SO(\Lambda)_{\Fpbar}$, respectively.
	
	Let $(B'_{\SO}, T'_{\SO}, g')$ be the frame for the attached normalized $G$-zip datum $Z_{\mu'}^{\data}$ on $\SO(n-1,2)_{\Fpbar}$ as in Lemma~\ref{Lem:TwistorEven}.
	Then we again have $\iota(g') = P_{(1,n+2)(m,m+2)} \in \SO(n,2)(\Fpbar)$, as in the split case.
	On the other hand, as in \S~\ref{S:ZipdataNormEven} (nonsplit case), the transported Frobenius
	$\varphi: \SO(n,2)_{\Fpbar} \to \SO(n,2)_{\Fpbar}$ is the twisted one
	\(
	A \mapsto P_{s_m} A^{(p)} P_{s_m}.
	\)
	
	Note that the triple $(B_{\SO}, T_{\SO}, g = P_{(1,n+2)(m,m+2)})$ is a frame for the normalized $G$-zip datum $Z_{\mu}^{\data}$ on $\SO(n,2)_{\Fpbar}$.
	Indeed, it suffices to check that $\varphi(L \cap B_{\SO}) = g^{-1}(L \cap B_{\SO})g$, which follows from the same argument as in the proof of Lemma~\ref{Lem:TwistorOdd}, together with the equality $P_{(1,n+2)(m,m+2)} = P_{w_0 w_{0,\mu} s_m}$.
	
	Since we use the new bases for the normalizations of $\SO(\Lambda')_{\Fpbar}$ and $\SO(\Lambda)_{\Fpbar}$, the normalized embedding $\iota$ from the former to the latter is again given by \eqref{Eq:MatrEmbEven2Odd}.
	Consequently we still have $\iota(P_{s'_i}) = P_{s_i}$ for each $i \leq m-1$, and
	$\iota(P_{s'_m}) = P_{s_m s_{m-1} s_m}$, as in the split case.
	Moreover, we have $\iota(g') = g$ by construction.
	Using this equality and the formulas for $\iota(s'_i)$, we obtain

		\begin{align*}
\iota(O_{\fram}^{w_i})
		&\subseteq 
		O_{\fram}^{w_i}
		\quad\text{for } i \le m-1,\qquad
		\iota(O_{\fram}^{w'^{\dagger}_{m-1}})
		\subseteq O_{\fram}^{w_{m-2} s_m s_{m-1} s_m}
		= O_{\fram}^{s_m w_{m-2} s_{m-1} s_m}
		\overset{\text{claim}}{=} O_{\fram}^{w_{m-1}},\\
\iota(O_{\fram}^{w_{m+j}})
&\subseteq O_{\fram}^{w_{m-2} (s_{m-1}s_m s_{m-1} s_{m-2}\ldots s_{m-(j+1)}) s_m}
= O_{\fram}^{w_{m+j+1} s_m}
\overset{\text{claim}}{=} O_{\fram}^{w_{m+j+1}}, \qquad j \geq 0,
\end{align*} 
where the two claims are proved below. 

	\noindent \textbf{Proof of claims.}
	In this nonsplit case the twisted Frobenius map $\varphi: L \to L$
	induces an automorphism $\psi = \mathrm{Int}(s_m): W_{\mu} \to W_{\mu}$, so in particular
	$\psi(s_m) = s_m$, as in the case where $\SO(\Lambda')_{\Fp}$ is split.
	Applying Theorem~\ref{Thm:Criteria} with $(v,v') = (s_m,1)$ yields the first claim.
	For the second, we use the same theorem with $(v,v') = (1,s_m)$; this is allowed because
	$s_m$ commutes with $s_i$ for $i \leq m-2$ and with $s_{m-1}s_m s_{m-1}$ (by direct verification).
	This completes the proof of Theorem~\ref{Thm:MainThmOrthog} when $n$ is odd and
	$\SO(\Lambda')_{\Fp}$ is nonsplit.
	Hence Theorem~\ref{Thm:MainThmOrthog} is proved in all cases.
	
	\begin{remark}\label{Rmk:n=1Embed}
	In the case $n=1$, the Shimura variety $\Sh_{\GSpin(0,2),\Fpbar}$ is $0$-dimensional. Accordingly, $\SO(0,2)_{\Fpbar}$ is a torus, so its associated $G$-zip stack is a point, whereas the $G$-zip stack attached to $(\SO(1,2)_{\Fpbar},\mu)$ has two orbits, corresponding to the $\mu$-ordinary and superspecial loci of $\Sh_{\GSpin(1,2),\Fpbar}$. The theorem thus implies that if $\kappa=\Fp$, then $\Sh_{\GSpin(0,2),\Fpbar}$ maps to the $\mu$-ordinary locus of $\Sh_{\GSpin(1,2),\Fpbar}$; if $\kappa=\Fpsq$, it maps to the superspecial locus. Compare also \S~\ref{S:ZerodimKS}.
\end{remark}

     \subsection{$a$-number of EO strata under the Kuga-Satake embedding}

For an abelian variety $A$ over $\Fpbar$, its $a$-number is defined by
\[
a(A)=\dim_{\Fpbar}\Hom_{\Fpbar}(\alpha_p,A).
\]
This invariant gives rise to a stratification of the Siegel moduli space
$\mathcal{A}_{g,N,\Fpbar}$, usually called the \emph{$a$-number stratification}; see, for example, \cite{EvdG09,Harashita04}. This stratification was generalized by Wedhorn \cite{bruh} to a broader class of Shimura varieties using the theory of zip period maps and group-theoretic methods; the resulting stratification is called the \emph{Bruhat stratification}. In general, the Bruhat stratification is coarser than the EO stratification. Hence, with respect to the Kuga-Satake embedding, the $a$-number of an EO stratum of $\Sh_{\GSpin(n,2), \Fpbar}$ is well-defined.
	\begin{corollary}\label{Cor:anumber}
		For each $n \geq 1$, the superspecial EO stratum $\Sh_{\GSpin(n,2), \Fpbar}^{\mathrm{ssp}}$ consists of superspecial abelian varieties, i.e.
		\[
		\Sh_{\GSpin(n, 2), \Fpbar}^{\mathrm{ssp}}
		= \iota^{-1}\bigl(\Sh_{\GSp_{2\cdot 2^n}, \Fpbar}^{\mathrm{ssp}}\bigr).
		\]
		With respect to the Kuga-Satake embedding, the ordinary EO stratum
		$\Sh_{\GSpin(n,2), \Fpbar}^{\mu\text{-}\mathrm{ord}}$ has $a$-number $0$, the superspecial locus
		$\Sh_{\GSpin(n, 2), \Fpbar}^{\mathrm{ssp}}$ has $a$-number $2^n$, while all other EO strata have $a$-number $2^{n-1}$ (when $n\geq 2$). 
	\end{corollary}
	\begin{proof}
		We first consider the case \(n=1\). In this case \(C^+(V)\cong B\) is a quaternion algebra over \(\Q\), and the Shimura variety \(\Sh_{\GSpin(V)}\) is of PEL type: it parametrizes principally polarized abelian surfaces with \(O_B\)-action, namely triples \((A,\iota:O_B\to \End(A),\lambda:A\xrightarrow{\sim} A^{\vee})\). The Kuga--Satake embedding is given by forgetting \(\iota\), so the EO stratification is determined by the isomorphism class of \((A,\iota,\lambda)[p]\). The hyperspecial assumption at \(p\) implies \(B\otimes \Q_p\cong M_2(\Q_p)\), hence \(O_B/p\cong M_2(\Fp)\). By Morita equivalence, there are only two isomorphism classes of such objects, represented by \(E^{\mathrm{ord}}\times E^{\mathrm{ord}}\) and \(E^{\mathrm{ss}}\times E^{\mathrm{ss}}\), equipped with their tautological \(\iota\) and \(\lambda\), where \(E^{\mathrm{ord}}\) and \(E^{\mathrm{ss}}\) denote an ordinary and a supersingular elliptic curve, respectively. This verifies the assertion on \(a\)-numbers for \(n=1\).

Now assume \(n\ge 2\). By \cite[p.~553]{bruh}, the EO strata that are neither ordinary nor superspecial form a single Bruhat stratum of the \(\GSpin\) Shimura variety, and hence all have the same \(a\)-number. Choose a filtration of quadratic subspaces \(V_1\subseteq \cdots \subseteq V_n=V\) together with a sequence of self-dual \(\Z_{(p)}\)-lattices \(\Lambda_1\subseteq \cdots \subseteq \Lambda_n=\Lambda\). This gives a sequence of morphisms of reductive pairs \((\SO(\Lambda_1),\mu_1)\to \cdots \to (\SO(\Lambda_n),\mu_n)\), and hence a sequence of morphisms between geometric mod \(p\) Shimura varieties
\[
\Sh_{\SO(\Lambda_1),\Fpbar}\to \cdots \to \Sh_{\SO(\Lambda_n),\Fpbar}.
\]
By Theorem~\ref{Thm:MainThmOrthog}, best seen from the illustrating diagram in \S~\ref{S:DiagOrthog}, together with the description of the embedding \(\Sh_{\GSp_{2\rg'},\Fpbar}\hookrightarrow \Sh_{\GSp_{2\rg},\Fpbar}\) in \S~\ref{S:ShiEmb}, we are reduced to the case \(n=2\). Thus it remains to show that in this case the two \(1\)-dimensional EO strata have \(a\)-number \(2\).

In the case \(n=2\), the Shimura variety \(\Sh_{\GSpin(\Lambda)}\) is again of PEL type. More precisely, one has \(C^+(V)\cong B\otimes_\Q F\), where \(B\) is a quaternion algebra over \(\Q\) and \(F=Z(C^+(V))\) is either \(\Q\oplus\Q\) or a real quadratic field. For our purposes, we may assume we are in the latter case, the former being a degenerate one. The integral model \(\mathcal{S}h_{\GSpin(\Lambda),\Z_p}\) over \(\Z_p\) parametrizes triples \(\underline{A}=(A,\iota:O_D\otimes \Z_p\to \End(A),\lambda:A\xrightarrow{\sim} A^{\vee})\), where \(A\) is an abelian variety of relative dimension \(4\), \(O_D=O_B\otimes O_F\), and \(O_B\) is a maximal order of \(B\). As before, the Kuga--Satake embedding is given by forgetting the \(O_D\)-action, so the EO stratification is governed by the isomorphism class of \(\underline{A}[p]\). By the hyperspecial assumption at \(p\), we have
\(
O_B\otimes O_F\otimes \Z_p \cong M_2(O_F\otimes \Z_p)
\)
and that $p$ is unramified in $F$. 
 By Morita equivalence, such a \(4\)-dimensional abelian variety is Morita-equivalent to an abelian surface with \(O_F\otimes \Z_p\)-action. Thus we are further reduced to the split case \(B=M_2(\Q)\), in which \(\mathcal{S}h_{\GSpin(\Lambda),\Z_p}\) is a Hilbert modular surface and, under Morita equivalence, the Kuga--Satake embedding is given by \(A\mapsto A\times A\). We then conclude from the description of the \((f,a)\)-invariants of EO strata for Hilbert modular surfaces in \S~\ref{S:ExampHilbert} below. 
		
	\end{proof}
	
		\subsection{The Illustrative Diagram (orthogonal)}\label{S:DiagOrthog}
        The diagram below illustrates Theorem~\ref{Thm:MainThmOrthog} and related phenomena.
\begin{enumerate}
    \item Each \(\boxed{?}\) represents an EO stratum, and the number \(?\) inside the box denotes its dimension. The rightward (dashed) arrows indicate the closure relations among the strata. For example, when \(n=2\), the two arrows \(\boxed{2}\dashrightarrow  \boxed{1}\) show that the dimension-\(2\) stratum specializes to the two distinct dimension-\(1\) strata. The downward arrows describe the image EO strata under embeddings of GSpin Shimura varieties considered above. For instance, the right downward arrow \(\boxed{1}\xrightarrow{\mathrm{ns}}\boxed{1}\) means that the two dimension-\(1\) EO strata of \(\Sh_{\GSpin(2,2),\Fpbar}=\Sh_{\GSpin(\Lambda'),\Fpbar}\) map to the dimension-\(1\) EO stratum of \(\Sh_{\GSpin(3,2),\Fpbar}\) when \(\SO(\Lambda')_{\Fp}\) is nonsplit over \(\Fp\) (equivalently, when \(\SO(V')_{\Qp}\) is nonsplit over \(\Qp\)).
    \item For \(n\) odd or \(n\) even split (meaning $\SO(\Lambda)_\Fp$ is split over $\Fp$), EO strata on/left of the \(p\)-rank-zero line are nonbasic with \(p\)-rank \(>0\), while those right of it are basic with \(p\)-rank \(0\). For \(n\) even nonsplit, EO strata left of the line are nonbasic with \(p\)-rank \(>0\), while those on/right of it are basic with \(p\)-rank \(0\).
    \item  For fixed \(n\), all EO strata inside the \emph{cone of middle \(a\)-numbers} have the same \(a\)-number, namely \(2^{n-1}\). The stratum outside the right blue line forms the superspecial loci and has \(a\)-number \(2^n\), whereas the stratum outside the left blue line is the maximal ordinary EO stratum, and has \(a\)-number \(0\).
\end{enumerate}

\begin{figure}[!htbp]
    \centering
    \includegraphics[width=\textwidth]{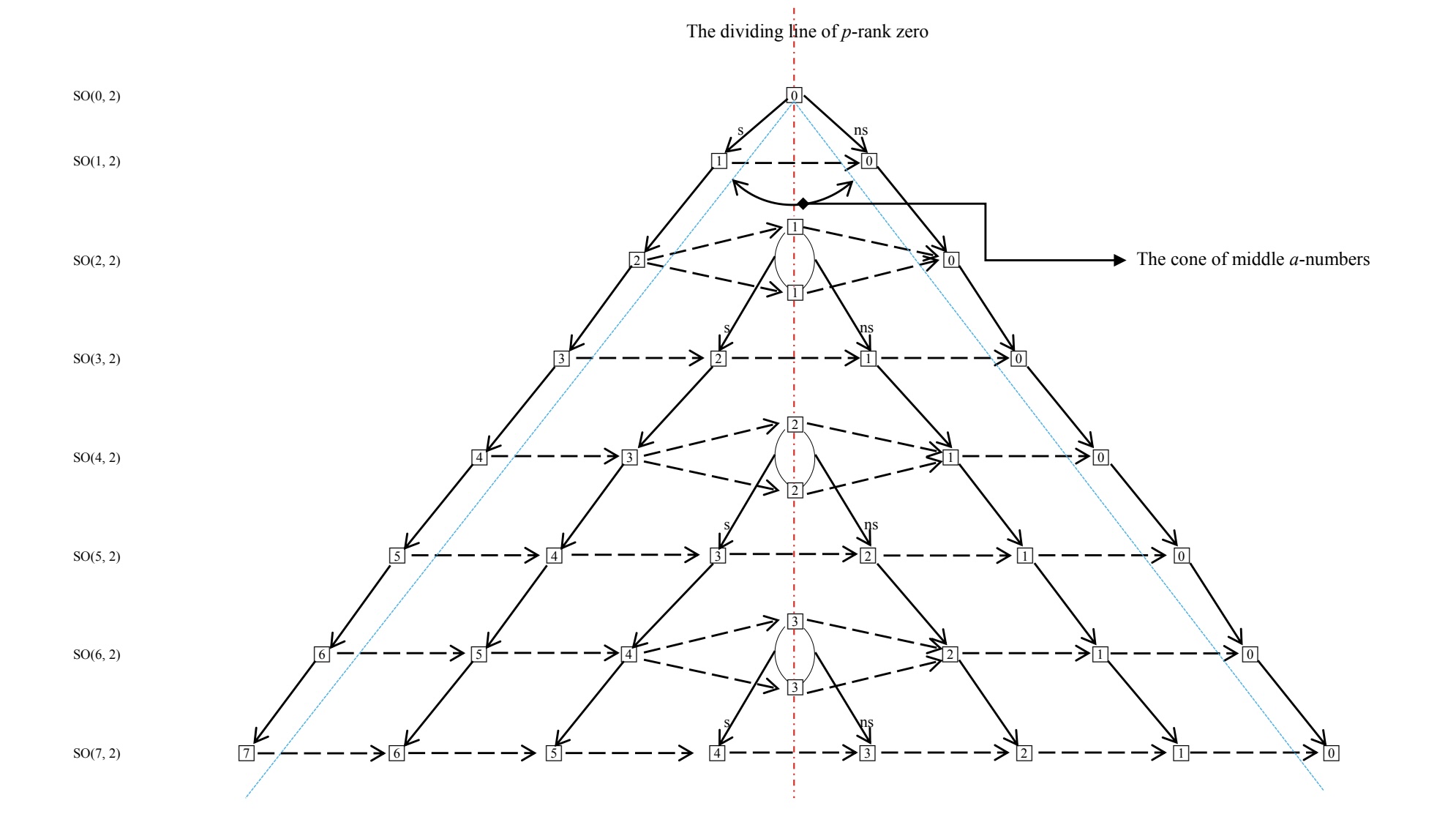}
   \caption{}
   \label{FunEOSpin}
   \[ \mathrm{s}= \mathrm{split}, \quad \mathrm{ns}= \mathrm{nonsplit}.\]
\end{figure}

\section{Zero-dimensional Kuga-Satake embedding}\label{S:ZerodimKS}
	 
In this section we consider the Kuga--Satake embedding
\(
\iota_{\KS}:\Sh_{\GSpin(0,2),\Fpbar} \longrightarrow \Sh_{\GL_2,\Fpbar},
\)
for the geometric mod \(p\) Shimura variety \(\Sh_{\GSpin(0,2),\Fpbar}\), whose target is the classical modular curve. The result and method of proof are similar to those in the \(n=1\) case of Theorem~\ref{Thm:MainThmOrthog}; see \S~\ref{Rmk:n=1Embed}. We treat this case separately because it is interesting in its own right and also helps clarify the proof of that case.

Let \(n=0\). Fix an \(\mathbb{F}_p\)-basis \(\delta_1,\delta_2\) of \(\Lambda_{\Fp}\) such that the quadratic form \(q\) on \(\Lambda_{\Fp}\) is given by either case~(I), \(q=x_1^2-x_2^2\), or case~(II), \(q=x_1^2-u x_2^2\), where \(u\in \mathbb{F}_p^\times\setminus (\mathbb{F}_p^\times)^2\). In this special case, we have the following elementary facts about \((\Lambda_{\Fp}, q)\): \textbf{(1)} As an \(\Fp\)-vector space, \(C^+(\Lambda_{\Fp})\) has dimension \(2\), with standard basis \(\mathfrak{B} = \{1, \delta_1 \delta_2\}\). \textbf{(2)} There is an \(\Fp\)-algebra isomorphism \(C^+(\Lambda_{\Fp}) \cong \Fp[X]/(X^2 - u)\). Hence, in case~(I) we have \(C^+(\Lambda_{\Fp}) \cong \Fp \times \Fp\); in case~(II) we have \(C^+(\Lambda_{\Fp}) \cong \Fp[\sqrt{u}] \cong \Fpsq\), where \(\sqrt{u} \in \mathbb{F}_{p^2}\) is a chosen square root of \(u\). Below we view case~(I) as the degenerate case of case~(II) by taking \(u=1\). \textbf{(3)} The basis \(\mathfrak{B}\) induces an isomorphism of algebraic groups over \(\Fp\),
\[
\epsilon_{\mathfrak{B}} : \GL_{\Fp}(C^+(\Lambda_{\Fp})) \xrightarrow[\cong]{\mathfrak{B}} \GL_{2,\Fp}.
\]
Using the basis \(\mathfrak{B}\) of \(C^+(\Lambda_{\Fp})\), left multiplication by \(x + y \delta_1 \delta_2 \in C^+(\Lambda_{\Fp})^\times\) has matrix \(\begin{psmallmatrix} x & y u \\ y & x \end{psmallmatrix}\). Hence, via \(\epsilon_{\mathfrak{B}}\), the embedding \(\iota : \GSpin(\Lambda_{\Fp}) \hookrightarrow \GL_{\Fp}(C^+(\Lambda_{\Fp}))\), induced by left multiplication of \(C^+(\Lambda_{\Fp})^\times\) on \(C^+(\Lambda_{\Fp})\), becomes the inclusion,
\(
x+y\delta_1\delta_2\mapsto \big(\begin{smallmatrix} x & yu \\ y & x \end{smallmatrix}\big),
\)
with image
\(
\mathrm{GSpin}(\Lambda_{\Fp})_{\mathfrak{B}}
=
\left\{ \big(\begin{smallmatrix} x & yu \\ y & x \end{smallmatrix}\big) \right\}
\subseteq \GL_{2,\Fp}.
\)

The Shimura variety \(\Sh_{\GSpin(0,2), \Fpbar}\) is defined over \(\Fp\) in case~(I) and over \(\Fpsq\) in case~(II). The corresponding conjugacy classes of cocharacters \( [\mu]_{\Fpbar}\) admit representatives over \(\Fp\) and \(\Fpsq\), respectively. To describe a standard representative cocharacter, we change basis for the base change \(C^+(\Lambda_{\Fpsq})\) over \(\Fpsq\),
\[
\widetilde{\mathfrak{B}}
:= \left\{ \delta'_1 := \frac{1}{2}(1 + \frac{\delta_1 \delta_2}{\sqrt{u}}),\,
\delta'_2 := \frac{1}{2}(1 - \frac{\delta_1 \delta_2}{\sqrt{u}}) \right\}.
\]
Let \(\epsilon_{\widetilde{\mathfrak{B}}} : \GL_{\Fpsq}(C^+(\Lambda_{\Fpsq})) \xrightarrow[\cong]{\widetilde{\mathfrak{B}}} \GL_{2,\Fpsq}\) be the induced isomorphism. Using this new basis, the algebra \(C^+(\Lambda_{\Fpsq})\) is identified with \(\Fpsq\times \Fpsq\), and an element \(a\delta'_1+b\delta'_2\in C^+(\Lambda_{\Fpsq})^\times\) acts on \(C^+(\Lambda_{\Fpsq})\) by left multiplication as \(\diag(a,b)\). Thus, under \(\epsilon_{\widetilde{\mathfrak{B}}}\), the base-changed embedding \(\iota_{\Fpsq}\) has image
$\GSpin(\Lambda_{\Fpsq})_{\widetilde{\mathfrak{B}}} =
T_{\Fpsq}\subseteq \GL_{2,\Fpsq}$,
where \(T\subseteq \GL_2\) denotes the standard diagonal torus. Identifying \(\GSpin(\Lambda_{\Fpsq})\) with \(\GSpin(\Lambda_{\Fpsq})_{\widetilde{\mathfrak{B}}}\), the conjugacy class \( [\mu]_{\Fpbar}\) for \(\Sh_{\GSpin(\Lambda), \Fpbar}\) admits a representative over \(\Fpsq\),
\[
\mu : \mathbb{G}_{m,\Fpsq} \to T_{\Fpsq}, \quad t \mapsto \diag(t,1).
\]
In the degenerate case~(I), \(\epsilon_{\widetilde{\mathfrak{B}}}\) is already defined over \(\Fp\), and accordingly \(\mu\) is also defined over \(\Fp\).

Now we encounter a distinction between case~(I) and case~(II): In case~(I), \(\epsilon_{\widetilde{\mathfrak{B}}}\) is compatible with the Frobenius endomorphism, i.e.\ the Frobenius map on the left-hand side corresponds to the termwise \(p\)-power Frobenius \(\Frob : \GL_{2,\Fpsq} \to \GL_{2,\Fpsq}\) on the right-hand side. In case~(II), \(\epsilon_{\widetilde{\mathfrak{B}}}\) is not compatible with the Frobenius endomorphism. Using \((\sqrt{u})^p = -\sqrt{u}\), one computes that the Frobenius on the left-hand side transports to the twisted one on \(\GL_{2, \Fpsq}\),
\[
A \longmapsto \tilde{w}_0 \Frob(A) \tilde{w}_0,\, \tilde{w}_0 := \antidiag(1,1).
\]
The zip group datum 
\(Z_{\mu}^{\mathsf{data}}
= (P := P_{\mu},\, Q := (P_{\mu^{-1}})^{\varphi},\, \varphi : L \to L^{\varphi})
\)  admits simple description in this special case: we have \(P=B_{\Fpbar}\) and \(L=T_{\Fpbar}\) in both case~(I) and case~(II), while \(Q=B_{-, \Fpbar}\) in case~(I) and \(Q=B_{\Fpbar}\) in case~(II). Here \(B\) and \(B_-\) stand for the usual upper and lower triangular Borel of \(\GL_2\) respectively. Moreover, \(\varphi : L \to L^\varphi\) becomes the usual Frobenius \(\Frob : T_{\Fpbar} \to T_{\Fpbar}\) in case~(I), sending \(\diag(t_1,t_2)\) to \(\diag(t_1^p,t_2^p)\), and the twisted Frobenius \(\tilde{w}_0 \Frob \tilde{w}_0\) in case~(II).

\begin{lemma}[Choice of frames]\label{Lem:ChoiFrameRk1}
	The triple \((B, T, \tilde{w}_0)\) (resp.\ \((B, T, 1)\)) is a frame for the zip group data \(Z_{\mu}^{\mathsf{data}}\) in case~(I) (resp.\ case~(II)).
\end{lemma}

\begin{proof}
	This follows from a direct verification using the above explicit description of the zip group datum \(Z_{\mu}^{\mathsf{data}}\) and the identity \(B_-=\tilde{w}_0 B \tilde{w}_0\).
\end{proof}
	
\begin{proposition}\label{Prop:ZeroDimKuS}
Under the Kuga--Satake embedding
\(
\Sh_{\GSpin(0,2),\Fpbar}\to \Sh_{\GL_2,\Fpbar},
\)
the torus Shimura variety \(\Sh_{\GSpin(0,2),\Fpbar}\) is mapped into the ordinary locus in case~(I), and into the superspecial locus in case~(II).
\end{proposition}

\begin{proof}
We use the functoriality of the zip period map defining the EO stratifications on \(\Sh_{\GSpin(0,2),\Fpbar}\) and \(\Sh_{\GL_2,\Fpbar}\), that is, the commutative diagram
\[
\xymatrix{
\Sh_{\GSpin(\Lambda),\Fpbar}\ar[d]\ar[r]
& \GSpin(\Lambda_{\Fpbar})\text{-}\mathsf{Zip}^{\mu}\ar[d]\ar[r]^{\tilde{\mathfrak{B}}}_{\cong}
& T_{\Fpbar}\text{-}\mathsf{Zip}^{\mu}\ar[d]\\
\Sh_{\GSp(C^+(\Lambda))_{\Fpbar}}\ar[r]
& \GSp(C^+(\Lambda)_{\Fpbar})\text{-}\mathsf{Zip}^{\mu}\ar[r]^{\tilde{\mathfrak{B}}}_{\cong}
& \GL_{2,\Fpbar}\text{-}\mathsf{Zip}^{\mu}.
}
\]
Since \(\Sh_{\GSpin(0,2),\Fpbar}=\Sh_{\GSpin(\Lambda),\Fpbar}\) is zero-dimensional, its EO stratification is trivial. Hence the image of the inclusion \(T_{\Fpbar}\to \GL_{2,\Fpbar}\) lies in the \(E_\mu\)-orbit of \(1\in \GL_2(\Fpbar)\). By our choice of frames in \eqref{Lem:ChoiFrameRk1} and by Theorem~\ref{Thm:ZipstackTopWeyl}, we have
\[
1\in \GL_2^{w_0}=E_\mu\cdot (w_0w_0)
\quad\text{in case~(I),}
\qquad
1\in \GL_2^{1}=E_\mu\cdot (1\cdot 1)
\quad\text{in case~(II).}
\]
These orbits correspond to the ordinary locus in case~(I) and to the superspecial locus in case~(II).
\end{proof}

	\begin{remark}
		Consider the case where the map \(C^+(\Lambda_{\Fp})\hookrightarrow \GL_{2,\Fp}\) arises as the reduction modulo \(p\) of an embedding \(\kk\hookrightarrow M_2\), where \(\kk\) is an imaginary quadratic field. This proposition is consistent with the classical result due to Deuring that an elliptic curve over \(\Q\) with CM by \(O_{\kk}\) and good reduction at \(p\) has \emph{ordinary} reduction if and only if \(p\) \emph{splits} in \(\kk\).
	\end{remark}

	\section{Newton stratifications and Newton cocharacters}\label{S:NewtonStra} 
	
		\subsection{Newton stratification for general Shimura varieties} In this subsection we consider a mod $p$ Shimura variety $\Sh_{G, \kappa}$ of Hodge type over a finite field $\kappa\subseteq \Fpbar$, with attached good reduction mod $p$ data as in \eqref{Eq:ModpData}.  
		
		To discuss the Newton stratification on the mod $p$ Shimura variety $\Sh_{G,\Fpbar}$, we first recall some group-theoretic notions. 
		Set $L(\kappa) = W(\kappa)[1/p]$ and take the reductive pair
		\(
		(\mathcal{G},\mu_{W(\kappa)}) \otimes_{W(\kappa)} L(\kappa),
		\) over the local field $L(\kappa)$, together with the Galois group
		\(
		\Gamma = \mathrm{Gal}(\overline{\Q}_p / L(\kappa))
		\) as the input. 
		Attached to this datum is a finite partially ordered set $(B(G,\mu),\leq)$, the Kottwitz set associated with $(G_{L(\kappa)},\mu_{L(\kappa)})$, which we now recall following \cite{Kot85,Kot97}. 
		
		Let $B(G)=B(G_{L(\kappa)})$ be the set of $G(\Qpbrev)$-$\sigma$ conjugacy classes inside $G(\Qpbrev)$. It parametrizes the isomorphism classes of isocrystals with $G$ structure, and is classified by Kottwitz in \cite{Kot85, Kot97}. Let $C(G, \mu)=C(\mathcal{G}, \mu)$ be the set of $\mathcal{G}(\Zpbrev)$-$\sigma$ conjugacy classes inside the coset $\mathcal{G}(\Zpbrev))\mu(p)\mathcal{G}(\Zpbrev)$. Clearly there is a natural map $d\colon C(G, \mu)\to B(G)$. Define $B(G, \mu)= B(G_{L(\kappa)}, \mu_{L(\kappa)})$ to be the image of $d$. To describe the partial order $\leq$, it is convenient to fix a Borel pair $(B, T)$ over $L(\kappa)$; this is possible in our good reduction Shimura variety setting, where $G_{L(\kappa)}$ is unramified (in particular, quasi-split). Attached to each class $[b]\in B(G, \mu)$ is a rational dominant Galois-invariant cocharacter $\nu_{b}\in X_*(T)_{\Q}^{\Gamma}$. For $[b], [b']\in B(G, \mu)$, $[b]\leq [b']$ if and only if $\nu_{b}-\nu_{b'}$ is a $\Q_{\geq 0}$-linear combination of Galois invariant positive coroots of $T$.
		
		The Shimura variety $\Sh_{G,\Fpbar}$ carries a universal isocrystal with $G_{L(\kappa)}$-structure.
		The Newton stratification on $\Sh_{G,\Fpbar}$ is defined via the map
		\(
		\mathsf{Nt} : \Sh_G(\Fpbar) \longrightarrow B(G,\mu),
		\)
		which sends each $\Fpbar$-point to the isomorphism class of the specialization of this universal isocrystal at that point.
		By \cite{RR96}, for each $[b] \in B(G,\mu)$ the subset
		\[
		\Sh_{G,\Fpbar}^{[b]} := \mathsf{Nt}^{-1}([b]) \subseteq \Sh_{G,\Fpbar}(\Fpbar)
		\]
		is the set of $\Fpbar$-points of a locally closed subscheme of $\Sh_{G,\Fpbar}$; it is called the Newton stratum of $\Sh_{G,\Fpbar}$ attached to $[b]$ and we still denote it by $\Sh_{G,\Fpbar}^{[b]}$.
		There is a rich literature on this stratification, including \cite{VW13,Ham15, SZ22, Man20, Vie20}, to which we refer for basic properties of Newton strata.

		There is a slightly more structural description of $B(G,\mu)$.
		By \cite{Kot85,Kot97}, there is an injective map
		\[
		B(G) \xrightarrow{(\nu,\,\kappa_G)}
		\bigl(X_*(T)\otimes \Q\bigr)_{\mathrm{dom}}^{\Gamma} \times \pi_1(G_{L(\kappa)})_{\Gamma},
		\]
		where $\nu$ is the Newton map and $\kappa_G$ is the Kottwitz map (in the form reformulated in \cite{RR96}).
		Then
		\[
		B(G,\mu)
		= \bigl\{ [b]\in B(G)\,\bigm|\,
		\nu([b]) \leq \nu([\mu(p)]),\;
		\kappa_G([b]) = \mu^{\flat} := \kappa_G([\mu(p)])
		\bigr\}.
		\]
		We refer to \cite{VW13} for a more detailed summary of $B(G,\mu)$ in the present  setting of good reduction mod $p$ Shimura data.
		It allows us to identify $B(G,\mu)$ with its image under the Newton map inside $(X_*(T)\otimes \Q)^{\Gamma}$, i.e.\ with the associated set of Newton cocharacters, which we denote by $N(G,\mu)$. For Siegel Shimura varieties $\Sh_{\GSp_{2\rg}, \Fpbar}$, the Newton cocharacters $N(G, \mu)$ correspond to classical Newton Polygons for polarized abelian varieties. 
		
		The partially ordered set $N(G,\mu)$ has a unique minimal element $\nu_{\mathrm{basic}}$. The corresponding Newton stratum, denoted $\Sh_{G,\Fpbar}^{\mathrm{basic}}$, is closed and contained in the closure of every other Newton stratum; it is called the \emph{basic locus} of $\Sh_{G,\Fpbar}$.
		Moreover, $N(G,\mu)$ has a unique maximal element, i.e., $\nu([\mu(p)])$.
		The corresponding stratum, denoted $\Sh_{G,\Fpbar}^{\mu\text{-ord}}$, is called the \emph{$\mu$-ordinary locus} of $\Sh_{G,\Fpbar}$; it coincides with the ordinary EO stratum and (hence) is open and dense, see \cite{Wortmann}.

		\begin{lemma}\label{Lem:adiso_BHmu}
			Let $f:H_1\to H_2$ be a morphism of quasi-split connected reductive groups over $L(\kappa)$.
			Assume that $f(Z(H_1))\subseteq Z(H_2)$ and that the induced morphism $f^{\ad}:H_1^{\ad}\rightarrow H_2^{\ad}$ is an isomorphism.
			Fix Borel pairs $(B_1,T_1)$ for $H_1$ and $(B_2,T_2)$ for $H_2$ such that $f(B_1)\subseteq B_2$ and $f(T_1)\subseteq T_2$.
			Let $\mu_1\in X_*(T_1)$ be $B_1$-dominant and put $\mu_2 := f\circ\mu_1 \in X_*(T_2)$.
			Then the induced map $f_*:B(H_1)\to B(H_2)$ restricts to an order-preserving bijection
			\[
			f_*:B(H_1;\mu_1)\xrightarrow{\sim}B(H_2;\mu_2).
			\]
		\end{lemma}
		
		\begin{proof}
			For $i=1,2$, the natural map $H_i \to H_i^{\ad}$ induces an isomorphism 
			\(
			B(H_i;\mu_i) \xrightarrow{\sim} B(H_i^{\ad};\mu_i^{\ad})
			\)
			of posets by \cite[(6.5.1)]{Kot97}.
			Since $f_*:B(H_1;\mu_1)\to B(H_2;\mu_2)$ is compatible with these identifications and $f^{\ad}$ is an isomorphism,
			it follows that $f_*:B(H_1;\mu_1)\to B(H_2;\mu_2)$ is an isomorphism of posets. 
		\end{proof}
		
		\subsection{Newton cocharacters for $\Sh_{\GSpin(\Lambda), \Fpbar}$}
		We describe the parametrization set $N(\GSpin(\Lambda)_{\Qp}, \mu)$ of Newton strata of $\Sh_{\GSpin(\Lambda), \Fpbar}$. 
		By Lemma \ref{Lem:adiso_BHmu}, the natural projection map of reductive pairs $(\GSpin(\Lambda)_{W}, \mu)\to (\SO(\Lambda)_{W}, \mu)$ induces an isomorphism of posets,
		\[ 
		N(\GSpin(\Lambda)_{\Qp}, \mu) \cong N(\SO(\Lambda)_{\Qp}, \mu).
		\]
		Hence it suffices to describe $N(\SO(\Lambda)_{\Qp}, \mu)$. We do this, following \cite[\S7.3]{SZ22}. The quadratic space $(\Lambda,q)_{\Qp}$ is determined up to isometry by its discriminant (viewed in $\RAT_p^\times/\RAT_p^{\times2}$) and its Hasse invariant. Under the good reduction assumption ($\Lambda$ is self-dual), the Hasse invariant is trivial, so $(\Lambda,q)_{\Qp}$ is determined by its discriminant alone. Below, after identifying \(\SO(\Lambda)_{\Qp}\) via a choice of basis of \(\Lambda_{\Qp}\), we take the standard upper triangular Borel subgroup \(B=B_{\SO}\) as the Borel subgroup in the Borel pair \((B,T)\). 
		
		\subsubsection{If $n$ is odd} There is a $\mathbb{Z}_p$-basis $\delta_1,\delta_2,\cdots, \delta_{n+2}$ of $\Lambda_{\Qp}$ such that
		$$q=x_1x_{n+2}+x_2x_{n+1}+\cdots+x_{m}x_{m+2}+cx^2_{m+1}, \text{ for some }c\in \mathbb{Z}_p^\times.$$
		Let \(T=\Big\{\mathrm{diag}\bigl(t_1,t_2,\cdots,t_m,1,t_m^{-1},\cdots,t_2^{-1},t_1^{-1}\bigr)\Big\}\), and $\alpha_i^\vee\in X_*(T)$, $1\leq i\leq m$, be the cocharacter $$t\mapsto \mathrm{diag}(1,\cdots, 1, t, 1, \cdots, 1, t^{-1}, 1, \cdots, 1)$$
		where the $t$ and $t^{-1}$ are at the $i$-th and $2m+2-i$-th place respectively. The set $N(\SO(\Lambda)_{\Qp}, \mu)$ of Newton cocharacters consists of,
		$$b_1=\alpha_1^\vee,\ b_2=\frac{1}{2}(\alpha_1^\vee+\alpha_2^\vee),\ \cdots,\ b_m=\frac{1}{m}(\sum_{i=1}^m\alpha_i^\vee), \ 1.$$

		\subsubsection{If $n$ is even and $\SO(\Lambda)_{\Qp}$ is split} Then there is a $\mathbb{Z}_p$-basis $\delta_1,\delta_2,\cdots, \delta_{n+2}$ of $\Lambda_{\Qp}$ such that
		$$q=x_1x_{n+2}+x_2x_{n+1}+\cdots+x_{m}x_{m+1}.$$
		With notations as above, the set $N(\SO(\Lambda)_{\Qp}, \mu)$ of Newton cocharacters consists of
		$$b_1=\alpha_1^\vee,\ b_2=\frac{1}{2}(\alpha_1^\vee+\alpha_2^\vee),\ \cdots,\, b_m=\frac{1}{m}(\sum_{i=1}^m\alpha_i^\vee),\ b_m'=\frac{1}{m}(\sum_{i=1}^{m-1}\alpha_i^\vee-\alpha_m^\vee),\ 1.$$
		
		\subsubsection{If $n$ is even and $\SO(\Lambda)_{\Qp}$ is nonsplit} Then there is a $\mathbb{Z}_p$-basis $\delta_1,\delta_2,\cdots, \delta_{n+2}$ of $\Lambda_{\Qp}$ such that
		$$q=x_1x_{n+2}+\cdots+x_{m-1}x_{m+2}+x^2_{m}-cx^2_{m+1},\ \ c \text{ non-square in } \mathbb{Z}_p^\times.$$
		Let $T_0=\Big\{\mathrm{diag}(t_1,t_2,\cdots,t_{m-1},1,1,t_{m-1}^{-1},\cdots,t_2^{-1},t_1^{-1})\Big\}$ be the maximal split torus over $\Zp$ and $T$ be the centralizer of $T_0$. Then we have identification \(
         X_*(T)_{\Q}^{\Gamma} \cong X_*(T_0)_{\Q}. \)
        The set $N(\SO(\Lambda)_{\Qp}, \mu)$ of Newton cocharacters consists of
		$$b_1=\alpha_1^\vee,\ b_2=\frac{1}{2}(\alpha_1^\vee+\alpha_2^\vee),\, \cdots,\ b_{m-1}=\frac{1}{m-1}(\sum_{i=1}^{m-1}\alpha_i^\vee), \ 1.$$

        \begin{remark}\label{Rmk:ParOrderNewton}
            The partial order of $N(\SO(\Lambda)_{\Qp}, \mu)$ is given by 
            \[
            b_1 \geq b_2\geq \cdots \geq b_{m-1} \geq b_m, b_m'\geq 1,
            \]
        \end{remark}
        where there is no closure relation between $b_{m}$ and $b_{m}'$ when they both appear ($n$ even split case), and where the trivial cocharacter $1$ corresponds to the basic locus of $\Sh_{\GSpin(\Lambda), \Fpbar}$. By purity of Newton stratification, for each $1\leq j \leq m$, the Newton stratum $\Sh_{\GSpin(\Lambda), \Fpbar}^{[b_j]}$ has dimension $n+1-j$; in $n$ split even case, \[\dim\Sh_{\GSpin(\Lambda), \Fpbar}^{[b_m]} =\dim \Sh_{\GSpin(\Lambda), \Fpbar}^{[b_m']}= m-1.\]  
		
		\section[]{{$p$}-ranks of $\Sh_{\GSpin(\Lambda), \Fpbar}$ under the Kuga-Satake embedding}\label{S:PRank} 
	
        In this section, we determine the \(p\)-ranks of Newton strata of \(\Sh_{\GSpin(\Lambda),\Fpbar}\) via the Kuga--Satake embedding. More precisely, for a Newton stratum of \(\Sh_{\GSpin(\Lambda),\Fpbar}\), its \(p\)-rank is understood as the \(p\)-rank of the Newton stratum of the Siegel Shimura variety \(\Sh_{\GSp(C^+(\Lambda)),\Fpbar}\) containing its image under the Kuga--Satake embedding. We recall that, in a Siegel Shimura variety, all \(\Fpbar\)-points in a given Newton stratum have the same \(p\)-rank, equal to the multiplicity of the weight \(0\) among the Newton slopes of the corresponding Newton cocharacter. For \(\nu\in N(\GSp(C^+(\Lambda))_{\Qp},\mu)\), these Newton slopes are precisely the \(\Q\)-weights of \(\nu\) on the standard representation of \(\GSp(C^+(\Lambda))_{\Qp}\).

By the functoriality of the map \(\mathsf{Nt}\) defining the Newton stratification, we have a commutative diagram
\begin{equation}
\xymatrix{
\Sh_{\GSpin(\Lambda), \Fpbar}\ar[r]^{\mathrm{KS}}\ar[d]&\Sh_{\GSp(C^+(\Lambda)), \Fpbar}\ar[d] \\
N(\GSpin(\Lambda)_{\Qp}, \mu)\ar[r]^{N_{\mathrm{KS}}}& N(\GSp(C^+(\Lambda))_{\Qp}, \mu).}
\end{equation}
To ease notation, until the end of this section we sometimes omit \(\Qp\) in notations such as \(C^+(\Lambda)_{\Qp}\), \(\SO(\Lambda)_{\Qp}\), etc. Below we describe representatives on the \(\GSpin(\Lambda)\)-side for the Newton cocharacters in \(N(\GSpin(\Lambda)_{\Q_p},\mu)\), first for the nonbasic ones and then for the basic one.

We first lift the \(\alpha_i^\vee\)'s to cocharacters of \(\GSpin(\Lambda)\). Note that \(\alpha_i^\vee\) is a cocharacter of \(\SO(\Lambda')\subseteq \SO(\Lambda)\), attached to the quadratic subspace \(\Lambda'\) spanned by \(\delta_i\) and \(\delta_{n+3-i}\). By the commutative diagram,
\[
\xymatrix{\GSpin(\Lambda')\ar[r]\ar[d] &\GSpin(\Lambda)\ar[d]\\
\SO(\Lambda')\ar[r]&\SO(\Lambda),
}\]
for the purpose of lifting \(\alpha_i^\vee\), we reduce to the \(2\)-dimensional case, where \(q=x_1^2-x_2^2\), and \(\Lambda= \Zp \delta_1\oplus \Zp \delta_2\). Now \(C^+(\Lambda)\) is spanned by \(1,\delta_1\delta_2\), and \(\GSpin(\Lambda)=C^+(\Lambda)^\times\cong \Gm\times \Gm\). Note that \(\alpha^\vee\) has weight \(1\) (resp.\ \(-1\)) on \(\delta_1-\delta_2\) (resp.\ \(\delta_1+\delta_2\)). We consider the cocharacter
\[
\chi_{\alpha^\vee}:\Gm\to C^+(\Lambda)^\times,\qquad
t\mapsto t\frac{1+\delta_1\delta_2}{2}+\frac{1-\delta_1\delta_2}{2}.
\]
The induced conjugation action on \(\Lambda\) is then
\[
\delta_1\mapsto \frac{t+t^{-1}}{2}\delta_1+\frac{-t+t^{-1}}{2}\delta_2,\qquad
\delta_2\mapsto \frac{-t+t^{-1}}{2}\delta_1+\frac{t+t^{-1}}{2}\delta_2.
\]
Clearly, this means that it has weight \(1\) (resp.\ \(-1\)) on \(\delta_1-\delta_2\) (resp.\ \(\delta_1+\delta_2\)), and \(\chi_{\alpha^\vee}\) is a lift of \(\alpha^{\vee}\). With respect to the standard representation of \(\GSp(C^+(\Lambda))\) on \(C^+(\Lambda)\), it has weight \(1\) on \(1+\delta_1\delta_2\), and weight \(0\) on \(1-\delta_1\delta_2\).

It follows from the discussion above that for \(1\le j\le m\), the Newton cocharacter \(b_j\) of \(\SO(\Lambda)\) admits the lift to \(\GSpin(\Lambda)\) given by
\[
\nu_j:t\mapsto \prod_{i=1}^j\left(\sqrt[j]{t}\frac{1+\delta_i\delta_{n+3-i}}{2}+\frac{1-\delta_i\delta_{n+3-i}}{2}\right).
\]
Here of course, when \(n\) is even and \(\SO(\Lambda)\) is nonsplit, this formula is only for \(1\le j\le m-1\). When \(n\) is even and \(\SO(\Lambda)\) is split, we choose a lift of
\(
b_m'
\)
as
\[
\nu'_m:t\mapsto \left(\frac{1+\delta_m\delta_{m+1}}{2}+\sqrt[m]{t}\frac{1-\delta_m\delta_{m+1}}{2}\right)
\prod_{i=1}^{m-1}\left(\sqrt[m]{t}\frac{1+\delta_i\delta_{n+3-i}}{2}+\frac{1-\delta_i\delta_{n+3-i}}{2}\right).
\]
On the other hand, the basic element \(1\in N(\SO(\Lambda),\mu)\) corresponds on the \(\GSpin(\Lambda)\)-side to the rational central cocharacter
\[
\nu_{\mathrm{basic}}=\frac12\,\eta^\vee \in X_*(Z(\GSpin(\Lambda)))_{\Q},
\]
which, in Clifford-algebra scalar notation, is represented by \(t\mapsto t^{1/2}\).

\begin{theorem}\label{Thm:Prank}\begin{enumerate}
    \item The basic locus $\Sh_{\GSpin(\Lambda), \Fpbar}^{\mathrm{basic}}$ has slope $1/2$, with multiplicity $2^{n+1}$.
    \item For each \(1\le j\le m\), the Newton slopes attached to \(N_{\KS}(\nu_j)\) are
\[
i/j, \quad  \text{ with multiplicity } 2^{n+1-j}\binom{j}{i}, \quad 0\le i\le j.
\]
Hence the Newton stratum $\Sh_{\GSpin(\Lambda), \Fpbar}^{[\nu_j]}$ has $p$-rank \(2^{n+1-j}\). 
\item In case \(n\) is even and \(\SO(\Lambda)_{\Qp}\) is split, \(N_{\KS}(\nu_{m})\) and \(N_{\KS}(\nu'_{m})\) have the same Newton slopes. In particular, the Newton strata $\Sh_{\GSpin(\Lambda), \Fpbar}^{[\nu_m]}$ and $\Sh_{\GSpin(\Lambda), \Fpbar}^{[\nu'_m]}$ have the same \(p\)-rank, \(2^{m-1}\).
\end{enumerate}
\end{theorem}

\begin{proof} The assertion on the basic locus is clear. 
Fix \(1\le j\le m\), and consider the \(j\) elements in \(C^+(\Lambda)\),
\[
\theta_j^{\pm}:=\frac{1}{2}(1\pm \delta_j\delta_{n+3-j}),\ 
\theta_{j-1}^{\pm}:=\frac{1}{2}(1\pm \delta_{j-1}\delta_{n+4-j}),\ 
\cdots,\ 
\theta_1^{\pm}:= 
\frac{1}{2}(1\pm \delta_1\delta_{n+2}).
\]
One checks directly that these elements commute with each other, and for all \(1\le i \leq j\),
\[
\theta_i^+\theta_i^-=0=\theta_i^-\theta_i^+,\qquad
(\theta_i^{\sigma_i})^2=\theta_i^{\sigma_i}, \quad \sigma_i\in \{\pm\}.
\]
For \((\sigma_1,\sigma_2,\dots,\sigma_j)\in\{\pm\}^j\), set
\(
\theta^{\underline{\sigma}}:=\theta_1^{\sigma_1}\theta_2^{\sigma_2}\cdots \theta_j^{\sigma_j}.
\)
Left multiplication by \(\theta^{\underline{\sigma}}\) on \(C^+(\Lambda)\) induces an endomorphism of \(C^+(\Lambda)\), whose image is denoted by \(C_{\underline{\sigma}}\). As \(\underline{\sigma}\in\{\pm\}^j\) varies, we obtain \(2^j\) subspaces of \(C^+(\Lambda)\).

\noindent\textbf{Claim}: These \(C_{\underline{\sigma}}\) induce a direct sum decomposition of \(C^+(\Lambda)\),
\[
C^+(\Lambda)= \oplus_{\underline{\sigma}\in \{\pm\}^j} C_{\underline{\sigma}}, \text{ with } \dim C_{\underline{\sigma}}= 2^{n+1-j}.
\]
Assuming the claim, note that
\(
\nu_j=\prod_{i=1}^j(\sqrt[j]{t}\theta_i^+ + \theta_i^-).
\)
Its left multiplication action on \(C_{\underline{\sigma}}\) is via scalar multiplication by \(t^{r/j}\), where
\(
r=\#\{\,1\le i\le j\mid \sigma_i=+\,\}.
\)
It follows from the claim that the slope \(i/j\) has multiplicity \(2^{n+1-j}\binom{j}{i}\). Notice also that we have
\(
\nu'_m=(\theta_m^+ + \sqrt[m]{t}\theta_m^-)\prod_{i=1}^{m-1}(\sqrt[m]{t}\theta_i^+ + \theta_i^-).
\)
Thus its left multiplication action on \(C_{\underline{\sigma}}\) also has only one weight, namely \(r/m\), where
\[
r=\#\{\,1\le i\le m-1\mid \sigma_i=+\,\}+\#\{\,\sigma_m=-\,\}.
\]
It follows that for each \(0\leq i \leq m\), the multiplicity of the slope \(i/m\) is equal to \(\binom{m-1}{i}+\binom{m-1}{i-1}=\binom{m}{i}\), and therefore \(\nu_m\) and \(\nu'_m\) have the same Newton slopes.

\noindent\textbf{Proof of Claim}. For \(c\in C^+(\Lambda)\), we have
\begin{equation*}
\begin{split}
c
&=1\cdot c
=(\theta_j^+ + \theta_j^-)c
=\theta_j^+c + \theta_j^-c \\
&=(\theta_{j-1}^+ + \theta_{j-1}^-)\theta_j^+c + (\theta_{j-1}^+ + \theta_{j-1}^-)\theta_j^-c
=\theta_{j-1}^+\theta_j^+c+\theta_{j-1}^-\theta_j^+c+\theta_{j-1}^+\theta_j^-c+\theta_{j-1}^-\theta_j^-c
=\cdots .
\end{split}
\end{equation*}
and hence inductively
\(
c=\sum_{\underline{\sigma}} c_{\underline{\sigma}}, \underline{\sigma}\in \{\pm\}^j.
\)
If \(\sum_{\underline{\sigma}}c_{\underline{\sigma}}=0\) for some \(c_{\underline{\sigma}}\in C_{\underline{\sigma}}\), then acting by
\(
\theta^{\underline{\sigma}'}:=\theta_1^{\sigma'_1}\theta_2^{\sigma'_2}\cdots \theta_j^{\sigma'_j}
\)
on both sides, we get \(0=c_{\underline{\sigma}'}\) for all \(\underline{\sigma}'\in \{\pm\}^j\). Here we use the identities
\[
(\theta^{\underline{\sigma}})^2=\theta^{\underline{\sigma}},\qquad
\theta^{\underline{\sigma}}\theta^{\underline{\sigma}'}=0
\quad\text{whenever}\quad
\underline{\sigma}\ne \underline{\sigma}'.
\]
This shows that the \(C_{\underline{\sigma}}\) induce a direct sum decomposition of \(C^+(\Lambda)\).

For the assertion on $\dim C_{\underline{\sigma}}$, let us write \(C_+\) for the \(C_{\underline{\sigma}}\) where \(\sigma_i=+\) for all \(1\le i\le j\). Fix \(\underline{\sigma} = (\sigma_1, \dots, \sigma_j) \in \{\pm\}^j\). We define an automorphism \(\alpha: \Lambda \to \Lambda\) by specifying its action on the basis \(\{\delta_1, \dots, \delta_{n+2}\}\). Let \(\alpha(\delta_i) = \sigma^*_i \delta_i\), where the signs \(\sigma^*_i\) are defined as follows:
\[
\sigma^*_i = \sigma^*_{n+3-i} = \sigma_i \text{ for } 1 \le i \le j, \text{ and } \sigma^*_i = 1 \text{ otherwise.} 
\]
This \(\alpha\) preserves the quadratic form \(q\) because it preserves each hyperbolic pair product \(x_i x_{n+3-i}\) (since \(\sigma_i \cdot \sigma_i = 1\)) and leaves the central term fixed in the odd case. Thus, \(\alpha\) induces an isometry of quadratic spaces \((\Lambda, q) \xrightarrow{\cong} (\Lambda, q)\) and an automorphism \(
C(\Lambda)\to C(\Lambda),
\) of the associated Clifford algebra \(C(\Lambda)\), which we still denote by \(\alpha\). Clearly, \(\alpha\) preserves \(C^+(\Lambda)\) and we have
\(
\alpha(\theta_1^+\cdots\theta_j^+)=\theta^{\underline{\sigma}}.
\)
Hence \(C_+\) and \(C_{\underline{\sigma}}\) have the same dimension, namely
\(
\frac{2^{n+1}}{2^j}=2^{n+1-j}.
\)
\end{proof}

As a corollary of Theorem \ref{Thm:Prank}, we have the following (cf. the diagram in \S~\ref{S:DiagOrthog}):
	\begin{corollary}\label{Cor:EONewton}
For each \(n \ge 1\), let \(\Sh_{\GSp(C^+(\Lambda)), \Fpbar}^{f=0}\) and
\(\Sh_{\GSp(C^+(\Lambda)), \Fpbar}^{\mathrm{ss}}\) denote the \(p\)-rank-zero locus
and the supersingular locus of \(\Sh_{\GSp(C^+(\Lambda)), \Fpbar}\), respectively.
Then
\begin{equation}\label{Eq:BasicLoc}
  \Sh_{\GSpin(\Lambda), \Fpbar}^{\mathrm{basic}}
=
\iota_{\KS}^{-1}\bigl(\Sh_{\GSp(C^+(\Lambda)), \Fpbar}^{f=0}\bigr)
=
\iota_{\KS}^{-1}\bigl(\Sh_{\GSp(C^+(\Lambda)), \Fpbar}^{\mathrm{ss}}\bigr).  
\end{equation}
Each nonbasic Newton stratum of \(\Sh_{\GSpin(\Lambda),\Fpbar}\) is a single EO
stratum, whereas the basic locus is a union of EO strata. In particular, all
nonbasic EO strata have pairwise distinct positive \(p\)-ranks, with one
exception: if \(n \ge 2\) is even and \(\SO(\Lambda)_{\Qp}\) is split, then the
two EO strata of dimension \(m-1\) have the same \(p\)-rank, namely \(2^{m-1}\).
\end{corollary}

\begin{proof}
By Theorem~\ref{Thm:Prank}, the basic locus is mapped into
\(\Sh_{\GSp(C^+(\Lambda)), \Fpbar}^{\mathrm{ss}}\), whereas every nonbasic Newton
stratum has positive \(p\)-rank. Hence \eqref{Eq:BasicLoc} holds.
For the second assertion, we use that \(\Sh_{\SO(\Lambda), \Fpbar}\) is fully
Hodge--Newton decomposable~\cite{GHNHodgeNewton}; see \cite[Cor.~7.3.4]{SZ22}. Since the EO stratification on \(\Sh_{\GSpin(\Lambda), \Fpbar}\) is the pullback of that on
\(\Sh_{\SO(\Lambda), \Fpbar}\), it follows that the Newton stratification on
\(\Sh_{\GSpin(\Lambda), \Fpbar}\) is coarser than the EO stratification, and that
each nonbasic Newton stratum is a union of EO strata, all of the same dimension
as that Newton stratum. Let \(d\) be the minimal dimension of a nonbasic Newton stratum. Then a direct
counting, using the descriptions of EO strata in Theorem~\ref{Thm:EOAbel} and of
Newton strata in \S\ref{Rmk:ParOrderNewton}, shows that
\[
\#\{\text{EO strata of dimension }\ge d\}
=
\#\{\text{nonbasic Newton strata}\}.
\]
Hence each nonbasic Newton stratum is a single EO stratum. The final statement
on \(p\)-ranks then follows from Theorem~\ref{Thm:Prank}.
\end{proof}

		\section{Examples in small rank}\label{S:Examples}

In this section, we present examples of Theorem~\ref{Thm:MainThmOrthog} in low-rank cases, in which the \(\GSpin\) Shimura varieties admit moduli interpretations.

\subsection{The case \(n=1\)} This is analogous to the 0-dimensional Kuga-Satake embedding in \S~\ref{S:ZerodimKS}.
In this case, \(\Sh_{\GSpin(\Lambda), \Fpbar}\) is a Shimura curve, while \(\Sh_{\GSpin(\Lambda'), \Fpbar}\) is a zero-dimensional toric Shimura variety defined over \(\Fp\) or \(\Fpsq\). The image of the embedding \(\Sh_{\GSpin(\Lambda'), \Fpbar}\to \Sh_{\GSpin(\Lambda), \Fpbar}\) consists of finitely many CM points. They lie in the ordinary locus (resp.~the superspecial locus) if \(\SO(\Lambda')_{\Fp}\) is split (resp.~nonsplit) over \(\Fp\). 

\subsection{The case \(n=2\)} \label{S:ExampHilbert}
Let \(F\) be the center of \(C^+(V)\). Then \(C^+(V)=C^+(V')\otimes_\Q F\), and \(F\) is either \(\Q\oplus\Q\) or a real quadratic field. The subalgebra \(C^+(V')\subseteq C^+(V)\) is a semisimple \(\Q\)-algebra of dimension \(4\), hence \(C^+(V')\cong B\) for some quaternion algebra \(B\) over \(\Q\). Moreover, \(\GSpin(V')=C^+(V')^\times\).

Now assume that \(B=M_2(\Q)\). Then we have
\[
\GSpin(V')=\GL_{2,\Q}\quad\text{and}\quad
\GSpin(V)=\bigl\{ g\in \mathrm{Res}_{F/\Q}\GL_2 \mid \det(g)\in \mathbb G_m \bigr\}.
\]
Thus \(\Sh_{\GSpin(V')}\) is the classical modular curve, while \(\Sh_{\GSpin(V)}\) is the Hilbert modular surface parametrizing abelian surfaces with \(O_F\)-action. Moreover, the embedding \(\Sh_{\GSpin(V')}\to \Sh_{\GSpin(V)}\) is given on the moduli side by sending an elliptic curve \(E\) to the Serre tensor product \(E\otimes_\Z O_F\), which comes tautologically with an \(O_F\)-action. Note that \(E\otimes_\Z O_F\cong E\times E\) noncanonically.

Now we assume that there exists a self-dual lattice \(\Lambda\subset V\) such that \(\Lambda':=V'\cap \Lambda\) is again self-dual for $V'$. At a prime \(v\mid p\), the corresponding Shimura data admit smooth integral models with hyperspecial level at \(p\), and the inclusion of Shimura varieties extends to a morphism of integral models over \(F_{(v)}\). Passing to geometric special fibers yields a natural embedding
\[
\Sh_{\GSpin(\Lambda'),\overline{\F}_p}\longrightarrow
\Sh_{\GSpin(\Lambda),\overline{\F}_p}.
\]
On the level of moduli, it is again given by \(E\mapsto E\otimes_\Z O_F\cong E\times E\). It follows that the ordinary (resp.~supersingular) locus of \(\Sh_{\GSpin(\Lambda'), \Fpbar}\) is mapped into the ordinary (resp.~superspecial) locus of the Hilbert modular surface \(\Sh_{\GSpin(\Lambda), \Fpbar}\); see below for a discussion of the EO strata of \(\Sh_{\GSpin(\Lambda), \Fpbar}\), which also prepares for the next subsection.

\subsection*{EO strata of \(\Sh_{\GSpin(\Lambda), \Fpbar}\)}
We need to distinguish the cases where \(p\) is split or inert in \(F\). In either case, there are four EO strata on \(\Sh_{\GSpin(\Lambda), \Fpbar}\): the \(0\)-dimensional superspecial locus, two distinct \(1\)-dimensional strata, and the open dense $\mu$-ordinary locus. Clearly, the superspecial locus has discrete invariants \((f,a)=(0,2)\), while the ordinary locus has \((f,a)=(2,0)\), where \(f\) and \(a\) denote the \(p\)-rank and \(a\)-number, respectively. For the two middle strata, if \(p\) is inert in \(F\), then \((f,a)=(0,1)\), while if \(p\) is split in \(F\), then \((f,a)=(1,1)\).

For the inert case, we refer to \cite{GO00}. The split case is elementary. Let \(A\) be an abelian surface over \(\Fpbar\) with \(O_F\)-action. Since \(O_F\otimes_\Z \Z_p\cong \Z_p\times \Z_p\), the induced action on the \(p\)-divisible group $A[p^\infty]$ yields a decomposition \(A[p^\infty]\cong G_1\times G_2\), where each \(G_i\) has dimension \(1\) and height \(2\). Thus, up to isomorphism, there are only four possibilities,
\[
E^{\mathrm{ord}}\times E^{\mathrm{ord}}, \quad
E^{\mathrm{ord}}\times E^{\mathrm{ss}}, \quad
E^{\mathrm{ss}}\times E^{\mathrm{ord}}, \quad
E^{\mathrm{ss}}\times E^{\mathrm{ss}},
\]
where the middle two are not identified because of the \(\Z_p\times \Z_p\)-action. This proves the above description of the EO strata.

\subsection{The case \(n=3\)}
Following \cite{KR00}, let \(B\) be an indefinite quaternion algebra over $\Q$, and put \(C:=M_2(B)\). Let \(x\mapsto x'\) be the involution on \(C\) induced by the main involution on \(B\), and define
\[
V:=\{\,y\in C: y'=y,\ \mathrm{Trd}(y)=0\,\}.
\]
The rule \(y^2=q(y)\cdot I_2\) defines a quadratic form \(q\) on \(V\), and \((V,q)\) has signature \((3,2)\). Moreover, the inclusion \(V\hookrightarrow C\) satisfies the Clifford relations, hence induces an isomorphism
\(
C^+(V)\xrightarrow{\sim} C=M_2(B).
\)
Accordingly, \(\GSpin(V)\) is a twisted form of \(\GSp_4\), and the associated Shimura variety \(\Sh_{\GSpin(V)}\) is a twisted Siegel threefold; when \(B=M_2(\Q)\), it is the usual Siegel modular threefold.

 Since \(\dim V'=4\) is even, the even Clifford algebra \(C^+(V')\) has quadratic \'etale center \(F:=Z(C^+(V'))\), which is necessarily either \(\Q\oplus\Q\) or a real quadratic field, since \(x\) is positive. The algebra \(F\) embeds into \(C^+(V)\cong M_2(B)\) as a noncentral quadratic subalgebra, and one has
\[
C^+(V')=\mathrm{Cent}_{C^+(V)}(F)=\mathrm{Cent}_{M_2(B)}(F)\cong B\otimes_\Q F.
\]
Thus \(\Sh_{\GSpin(V')}\) is a twisted Hilbert--Blumenthal surface. If \(F=\Q\oplus\Q\), it becomes a product of two modular curves, while if \(F\) is a real quadratic field and \(B=M_2(\Q)\), it is the usual Hilbert modular surface attached to \(F\).

From now on, assume that \(B=M_2(\Q)\) and that \(F\) is a real quadratic field. Then the embedding \(\Sh_{\GSpin(V')}\hookrightarrow \Sh_{\GSpin(V)}\) is the classical embedding of a Hilbert modular surface into the Siegel threefold, given on the PEL moduli interpretation by forgetting the \(O_F\)-action. Under the geometric mod \(p\) embedding
\(
\Sh_{\GSpin(\Lambda'), \Fpbar}\to \Sh_{\GSpin(\Lambda), \Fpbar},
\)
the superspecial (resp.~$\mu$-ordinary) locus clearly maps to the superspecial (resp.~$\mu$-ordinary) locus, while the two \(1\)-dimensional EO strata both map, in the split (resp.~inert) case, to the \(2\)-dimensional (resp.~\(1\)-dimensional) EO stratum of the Siegel threefold \(\Sh_{\GSpin(\Lambda), \Fpbar}\), which has \((f,a)=(1,1)\) (resp.~\((f,a)=(0,1)\)). In particular, in the split case the image is not contained in the supersingular locus, whereas in the inert case it is contained in the supersingular locus.
		
		\section{RSZ Unitary Shimura varieties}\label{S:RSZVar}
		In this section, we consider Rapoport-Smithling-Zhang (RSZ) unitary Shimura varieties over $\mathbb{Q}$ of signature $(n,1)$ and natural embeddings between them, mainly following \cite[\S2]{RSZDiagCycle} and \cite{RSZUnitary}.

		\subsection{Unitary  groups}\label{S:ManyUnitGp}
        Let $\kk$ be an imaginary quadratic field with the unique nontrivial automorphism $a \mapsto \overline{a}$. Let $W$ be a finite-dimensional $\kk$-vector space equipped with a Hermitian form $\Psi: W \times W \to \kk$ of signature $(n,1)$ (i.e., $\Psi$ is positive definite on an $n$-dimensional subspace and negative definite on a $1$-dimensional subspace). We assume that $p$ is unramified in $\kk$. Fix a square root $\delta = \sqrt{-D}$ of the discriminant $-D$ of $\kk$. Moreover, we set $ G = \mathrm{U}(W)$, $G^{\Q}= \GU(W)$, and $\nu: G^\mathbb{Q} \to \mathbb{G}_{m, \mathbb{Q}}$ the similitude character of \( G^\mathbb{Q} \), which has kernel \( G \subseteq G^\mathbb{Q} \). Moreover, we set 
		\( 
		Z^\mathbb{Q} = \text{Res}_{\kk/\mathbb{Q}} \mathbb{G}_{m, \kk},
		\)
        and also use $\nu: Z^\mathbb{\Q}\to \mathbb{G}_{m, \mathbb{Q}}$ to denote its similitude character; on \( \mathbb{Q} \)-valued points it is given by 
		\(
		a \mapsto a\bar{a}, a \in \kk^\times.
		\)
		Define a subgroup \( \tilde{G} \subseteq Z^\mathbb{Q} \times G^\mathbb{Q} \) consisting of elements \( (z, g) \) such that \( \nu(z) = \nu(g) \). Clearly, the canonical projection $\tilde{G}\to G^{\mathbb{Q}}$ is a central isogeny with kernel $Z=\mathbb{G}_{m, 
        \Q}$. Moreover, we have a decomposition:
		\begin{align}\label{tilGDecomp}
			\tilde{G} = Z^{\mathbb{Q}} \times G, \quad (z, g) \mapsto (z, z^{-1}g).
		\end{align}

		\subsection{Unitary Shimura varieties}
        Let
		\(
		h_{G^{\mathbb{Q}}}: \mathbb{S} \to G_{\mathbb{R}}^{\mathbb{Q}}\),  \( h_{Z^{\mathbb{Q}}}: \mathbb{S} \to Z_{\mathbb{R}}^{\mathbb{Q}} \), and $h_{\tilde{G}}= (h_{G^{\Q}}, h_{Z^{\Q}}): \mathbb{S}\to \tilde{G}_{\R}$ be as in \cite[\S3.1]{RSZDiagCycle}, and 
        $X(h_{G^{\mathbb{Q}}}) $ and $X(h_{Z^{\mathbb{Q}}})$, $X(h_{\tilde{G}})$ their associated Hermitian symmetric domain, so that we obtain the following Shimura data,
        \( (G^{\mathbb{Q}}, X(h_{G^{\Q}}))\),  \((Z^{\Q}, X(h_{Z^{\Q}}))\), and \((\tilde{G}, X(h_{\tilde{G}}))\).
		Set $h_{G}$ to be the composition of $h_{\tilde{G}}$ with the canonical projection $\tilde{G}\to G$ defined in \eqref{tilGDecomp}. Then we obtain another Shimura datum $(G, X(h_{G}))$. Moreover, the natural homomorphisms $\tilde{G}\to G^{\Q}$, $\tilde{G}\to G$  induce the morphisms between Shimura data,
        \[
         (G^{\Q}, X(h_{G^{\Q}})) \to (\tilde{G}, X(h_{\tilde{G}})),  \, (\tilde{G}, X(h_{\tilde{G}}))\to (G, X(h_G)), \, (\tilde{G}, X(h_{\tilde{G}}))\to (G\times Z^{\Q}, X(h_{G})\times X(h_{Z^{\Q}})). 
        \]
	     Note that the \( \tilde{G} \)-Shimura data above all have reflex field $\kk$,
		while the \( G^{\mathbb{Q}} \)-Shimura data have reflex field $\kk$ for all \( n \)
		except when \( n = 1 \). In the latter case, the \( G^{\mathbb{Q}} \)-Shimura data
		have reflex field \( \mathbb{Q} \).

        To form Shimura varieties from the Shimura data above, we need one more ingredient: level subgroups. For this we let $\mathsf{K}_{Z^{\Q}}$ be the maximal open compact subgroup of $Z^{\Q}(\Af)$, fix a lattice $\Lambda$ for \( W \), and define
        \(
        \mathsf{K}_{G}= \GU(\Lambda)(\hat{\Z}). 
        \)
        Then $\mathsf{K}_{\tilde{G}}:= \mathsf{K}_G\times \mathsf{K}_{Z^{\Q}}$ forms an open compact subgroup of $\tilde{G}(\Af)=G(\Af)\times Z^{\Q}(\Af)$. Now we can form the Shimura variety
		\[
		\mathrm{Sh}_{G^{\mathbb{Q}}} = \mathrm{Sh}_{\mathsf{K}_{G^{\mathbb{Q}}}}=\mathrm{Sh}_{\mathsf{K}_{G^{\mathbb{Q}}}}\bigl(G^{\mathbb{Q}}, X(h_{G^{\Q}})\bigr),
		\]
		such that its \( \mathbb{C} \)-points are identified with the orbifold quotient  
		\( G^{\mathbb{Q}}(\mathbb{Q}) \backslash (\tilde{G}(\mathbb{A}_f) \times  X(h_{G^{\Q}}) / \mathsf{K}_{G^{\mathbb{Q}}}.
		\)
        Similarly, we have $\Sh_{G}=\Sh_{G, \mathsf{K}_{G}}$ and $\Sh_{\tilde{G}}=\Sh_{\tilde{G}, \mathsf{K}_{\tilde{G}}}$. They are quasi-projective smooth varieties defined over their reflex field $\kk$ (or $\Q$ for $\Sh_{G^{\Q}}$ when $n=1$), all of dimension $n$. There are the following natural morphisms between these varieties,
        \[
        \Sh_{\tilde{G}} \to \Sh_{G^{\Q}}, \quad \Sh_{\tilde{G}} \to \Sh_{G}, \quad \mathrm{Sh}_{\tilde{G}}\hookrightarrow \mathrm{Sh}_{Z^{\mathbb{Q}}}\times \mathrm{Sh}_{G^{\mathbb{Q}}}. 
        \]
        
		\subsection{Embedding of unitary Shimura varieties}
      Let \(u\in W\) be a fixed positive vector, and set \(W_0=\kk\cdot u\) and
\(W^\flat=W_0^\perp\). Then \((W^\flat,\Psi|_{W^\flat})\) is a Hermitian space
of signature \((n-1,1)\). Choose a self-dual lattice $\Lambda^\flat$ for $W^\flat$.
Define $H=\mathrm{U}(\Lambda^\flat)$, $H^\Q=\GU(\Lambda^\flat)$, and $\tilde H$ as for $\tilde G$. Choose the level subgroups $\mathsf{K}_{H^{\Q}}, \mathsf{K}_{\tilde{H}}$ and $\mathsf{K}_{H}$ as for $\mathsf{K}_{G^{\Q}}, \mathsf{K}_{\tilde{G}}$ and $\mathsf{K}_{G}$.
The inclusion \(W^\flat\subseteq W\) induces natural embeddings
\(H\hookrightarrow G\) and \(\tilde H\hookrightarrow \tilde G\) of
\(\Q\)-algebraic groups, given by
\(h\mapsto \diag(h,1)\) and \((z,h)\mapsto (z,\diag(h,z))\), respectively.
These embeddings are compatible with the canonical projections
\(\tilde G\to G\) and \(\tilde H\to H\) defined in \eqref{tilGDecomp}\footnote{Note, however, that the embedding \(H\hookrightarrow G\) does not extend to an
embedding \(H^{\Q}\hookrightarrow G^{\Q}\).}. Thus if we fix a \( O_\kk \)-lattice \( \Lambda_0 \) of \( W_0 \) and set $\Lambda=\Lambda_0 \oplus \Lambda^{\flat} $, we obtain natural morphisms of Shimura varieties,
        \(
        \Sh_{H^{\Q}} \to \Sh_{G^{\Q}}\), \(\Sh_{\tilde{H}}\to \Sh_{\tilde{G}}\),  and \(\Sh_{H}\to \Sh_{G}\),
        fitting into the (commutative) diagram between morphisms of varieties over $\kk$,
        \begin{equation}\label{Eq:DiagShiUnit}
			\xymatrix{
				\mathrm{Sh}_{\tilde{H}} \ar[r] \ar[d] & \mathrm{Sh}_{\tilde{G}} \ar[d] \\ 
				\mathrm{Sh}_{H^{\mathbb{Q}}} & \mathrm{Sh}_{G^{\mathbb{Q}}},
			}
			\xymatrix{
				\mathrm{Sh}_{\tilde{H}} \ar[r] \ar[d] & \mathrm{Sh}_{\tilde{G}} \ar[d] \\ 
				\mathrm{Sh}_{H} \ar[r] & \mathrm{Sh}_{G}.
			}
		\end{equation}

		\begin{remark}
		    The Shimura varieties \(\Sh_{\tilde{G}}\) and \(\Sh_{\tilde{H}}\) were introduced by Rapoport--Smithling--Zhang in \cite{RSZDiagCycle} and are commonly referred to as the RSZ Shimura varieties. Like \(\Sh_{G^{\Q}}\) and \(\Sh_{H^{\Q}}\), they are of PEL type and hence admit moduli interpretations (see below), whereas \(\Sh_G\) and \(\Sh_H\) are in general only of abelian type. Moreover, the embedding \(\Sh_{\tilde{H}}\to \Sh_{\tilde{G}}\), covering \(\Sh_H\to \Sh_G\), admits a transparent moduli interpretation and can therefore sometimes be used to reduce questions downstairs to questions upstairs.
		\end{remark}
		
		\subsection{The Moduli Stack \( \mathcal{M}(n, r) \)}
		
		Let \( n, r \geq 0 \) be an integer. The moduli stack \( \mathcal{M}(n, r) \) over $\kk$ associates to each locally noetherian $\kk$-scheme \( S \) the triples \( (A, \iota, \lambda) \), where \( A \) is an abelian scheme over \( S \) of relative dimension \( n + r \), \( \iota: O_\kk \to \mathrm{End}_S(A) \) is an action of \( O_\kk \),
			\( \lambda: A \to A^t \) is a principal polarization such that for all \( a \in O_\kk \), \( \iota(a)^\dagger = \iota(\bar{a}) \), where \( (\cdot)^\dagger \) denotes the Rosati involution with respect to \( \lambda \). Additionally, the triple \( (A, \lambda, \iota) \) must satisfy Kottwitz's signature condition:
		    \[
		      \mathrm{char}(T, \iota(a)|_{\mathrm{Lie}(A)}) = (T - a)^n (T - \bar{a})^r, \quad a \in O_\kk.
		      \]
        
			The moduli stack \( \mathcal{M}(n, r) \) is a Deligne-Mumford stack over \( O_\kk \), and it is smooth of dimension \( nr \) over \( O_\kk[D^{-1}] \). In the case \( n = 1 \) and \( r = 0 \), the stack \( \mathcal{M}(1, 0) \) is \'etale over \( O_\kk \); see, for example, \cite[Prop.~2.1]{KudlaRapoport2014}. Moreover, $\Sh_{G^{\Q}}$ is represented by an open and closed subscheme of $\mathcal{M}(n, 1)\otimes k$. It represents the functor that associates each \(\kk\)-scheme \( S \) with the quadruple \( (A, \iota, \lambda, \bar{\eta}) \), where \( (A, \iota, \lambda) \in \mathcal{M}(n, 1)(S) \), and \( \bar{\eta} \) is a \(\mathsf{K}_{G^{\Q}}\)-level structure, i.e., a \(\mathsf{K}_{G^{\mathbb{Q}}}\)-orbit of $\mathbb{A}_{\kk,f}$-isomorphism \( \hat{V}(A) \cong W \otimes \mathbb{A}_f \), with \( \hat{V}(A) \) the rational Tate module of \( A \), compatible with the symplectic structure on both sides; see \cite[Prop.~4.4]{KudlaRapoport2014} for details. In particular, we know that the Shimura varieties $\Sh_{G^{\Q}}$ and $\Sh_{Z^{\Q}}$ admit integral models over $O_{\kk}$, which we denote by $\mathcal{S}h_{G^{\Q}}$ and $\mathcal{S}h_{Z^{\Q}}$ respectively. 
	
		\subsection{$p$-integral models}
		Fix a prime \( v \) of $\kk$ lying above \( p \), and denote by $O_{\kk, (v)}$ the localization at $v$ of $O_\kk$.
		\begin{theorem}[{\cite[Thm.~4.1, Thm.~5.2]{RSZDiagCycle}}]
			The Shimura variety $\mathrm{Sh}_{\tilde{G}}$ admits an integral model over $O_{\kk}$, denoted by $\mathcal{S}h_{\tilde{G}}$, which is an open and closed substack of $\mathcal{S}h_{Z^\mathbb{Q}}\times_{O_\kk}\mathcal{S}h_{G^\mathbb{Q}}$, and thus an open and closed substack of $\mathcal{M}(1,0)\times_{O_\kk}\mathcal{M}(n, 1)$. The base change $\mathcal{S}h_{\tilde{G}, O_{\kk, (v)}}$ represents the functor that associates to each locally noetherian $S$ the tuple $(A_0, \iota_0, \lambda_0, A, \iota, \lambda, \bar{\eta})$, where $(A_0, \iota_0, \lambda_0)\in \mathcal{M}(1, 0)(S)$, $(A, \iota, \lambda)\in \mathcal{M}(n, 1)(S)$, and $\bar{\eta}$ is a $\mathsf{K}_{\tilde{G}}$-orbit of $\mathbb{A}_{\kk,f}$-linear isometries:
			\[
			\eta^p: \hat{V}^p(A_0, A) \cong -W \otimes_\kk \mathbb{A}_{\kk, f}^p, \text{ with } \hat{V}(A_0, A) := \mathrm{Hom}_\kk(\hat{V}^p(A_0), \hat{V}^p(A)).
			\]
		\end{theorem} 
		
		The finite morphism $\mathrm{Sh}_{\tilde{H}} \to \mathrm{Sh}_{\tilde{G}}$ extends to a finite morphism of $O_\kk$-stacks, \(\mathcal{S}h_{\tilde{H}}  \to \mathcal{S}h_{\tilde{G}}\), given by, 
			\begin{equation}
			    (A_0, \iota_0, \lambda_0, A^{\flat}, \iota^{\flat}, \lambda, \bar{\eta}^{\flat}) 
		\mapsto \bigl(A_0, \iota_0, \lambda_0, A^{\flat}\times A_0, \iota^{\flat}\times \iota_0, \lambda^{\flat}\times \lambda_0(u), \bar{\eta}\bigr). \label{ModiofEmbedH2G}
			\end{equation}
	For unspecified notations such as $\lambda_0(u)$ and $\bar{\eta}$, we refer to \cite[P.~1762]{RSZDiagCycle}. These will not be used explicitly in this paper.
		
		For the Shimura varieties $\mathrm{Sh}_H$ and $\mathrm{Sh}_G$, which are of abelian type rather than PEL type, we lack knowledge of their $O_\kk$-integral models. However, Kisin's construction \cite{CIMK} ensures that they admit $p$-integral models over $O_{\kk, (v)}$, denoted $\mathcal{S}h_{H, O_{\kk,(v)}}$ and $\mathcal{S}h_{G, O_{\kk,(v)}}$, respectively. Moreover, the finite morphism $\mathrm{Sh}_H \to \mathrm{Sh}_G$ extends to a finite morphism  
		\[
		\mathcal{S}h_{H, O_{\kk, (v)}} \to \mathcal{S}h_{G, O_{\kk, (v)}}.
		\]
		We obtain the following commutative diagrams of morphisms between integral models  of Shimura varieties (over $O_\kk$ and $O_{\kk, (v)}$ respectively), extending \eqref{Eq:DiagShiUnit}:
		\begin{equation}\label{Eq:DiagShiUnitInt}
			\xymatrix{
				\mathcal{S}h_{\tilde{H}} \ar[r] \ar[d] & \mathcal{S}h_{\tilde{G}} \ar[d] \\ 
				\mathcal{S}h_{H^{\mathbb{Q}}} & \mathcal{S}h_{G^{\mathbb{Q}}},
			}\quad 
			\xymatrix{
				\mathcal{S}h_{\tilde{H}, O_{\kk, (v)}} \ar[r] \ar[d] & \mathcal{S}h_{\tilde{G}, O_{\kk, (v)}} \ar[d] \\ 
				\mathcal{S}h_{H, O_{\kk, (v)}} \ar[r] & \mathcal{S}h_{G, O_{\kk, (v)}}.
			}  
		\end{equation}

		\section{The EO stratifications for unitary Shimura varieties}\label{S:EOUnitShv}
		Write  \( \kk_v \) for the completion of \( \kk \) at \( v \). Then it is an unramified extension of \( \mathbb{Q}_p \), due to our assumption that $p$ is unramified in $\kk$.  Denote by \( \kappa = \kappa(v) \) the residue field of $\kk_v$, so \( \kappa = \mathbb{F}_p \) if \( p \) splits in $\kk$, or \( \kappa = \mathbb{F}_{p^2} \) if \( p \) is inert in $\kk$. Write $\mathcal{M}(n, r)_{\kappa(v)}=\mathcal{M}(n, r)\otimes_{O_{\kk}}\kappa(v)$ and $\mathcal{M}(n, r)_{\Fpbar}$ for the geometric fiber of \(\mathcal{M}(n,r)_{k(v)}\), obtained by base change along the embedding \(\kappa(v) \hookrightarrow \Fpbar\).
		\subsection{The EO stratification of $\mathcal{M}(n, 1)_{\Fpbar}$}
		We first discuss the EO stratification of the geometric fiber \( \mathcal{M}(n, 1)_{\overline{\mathbb{F}}_p} \) of \( \mathcal{M}(n, 1)_{\kappa} \), which has already been studied in \cite{MoonenWeyl}. It is defined by declaring that two $\Fpbar$ points $\underline{A}=(A, \iota, \lambda)$ and $\underline{B}=(A',\iota', \lambda')$ of $\mathcal{M}(n, 1)_{\Fpbar}$ lie in the same stratum if their $p$-kernels $A[p]$ and $B[p]$ are isomorphic, compatible with additional structures, i.e., 
		if \(\underline{A}[p]\cong \underline{A'}[p]\), where $\underline{A}[p]$ denotes the $1$-truncated triple \[(A[p], \iota: \tilde{\kappa} \to \End_{\Fpbar}(A[p]), \lambda:A[p]\cong A^{\vee}[p]=A[p]^{\vee}),\]
		equivalently, if $(A[p], \iota)\cong (A'[p], \iota')$ \cite[Thm.~6.7]{MoonenWeyl}. Here we set \(\tilde{\kappa}=O_\kk/pO_\kk\). It is naturally identified with 
\(\BF_p\times \BF_p\) if \(p\) splits in \(\kk\), and with
\(\BF_{p^2}\) if \(p\) is inert in \(\kk\). It comes equipped with the involution $*: \tilde{\kappa} \to \tilde{\kappa}$ given by the permutation $(x_1, x_2)\mapsto (x_2, x_1)$ in the first case and by the unique nontrivial automorphism in the second case. Here we view $\tilde{\kappa}$ as a simple $\mathbb{F}_p$-algebra with an involution $*: \tilde{\kappa}\to \tilde{\kappa}$ as in \cite[\S6]{MoonenWeyl}.
        
		  Hence the problem of defining the EO stratification of $\mathcal{M}(n, 1)_{\Fpbar}$ reduces to the problem of classifying the pair $(A, \iota)$, consisting of a BT-1 $A$ over $\Fpbar$ with an action $\iota: \tilde{\kappa}\to \End_{\Fpbar}(A)$. It follows essentially from \cite{MoonenWeyl} that the EO stratification $\mathcal{M}(n, 1)_{\Fpbar}$ is given by
        \[\mathcal{M}(n, 1)_{\Fpbar} = \bigsqcup_{a \in W_{\mathrm{EO}}^{(n, 1)}} \mathcal{M}(n, 1)_{\Fpbar}^a, \quad \text{ with } W_{\mathrm{EO}}^{(n, 1)}:= \{ \, 0, 1, \ldots, n\, \},\] where the individual stratum $\mathcal{M}(n, 1)_{\Fpbar}^a$ is described as follows. Given a $\Bar{\BF}_p$-point $\underline{A}=(A, \iota, \lambda)$ of $\mathcal{M}(n, 1)_{\Fpbar}$, let $M=(M, \rF, \rV)$ be the standard (contravariant) Dieudonn\'e module attached to $A[p]$. The action of $\tilde{\kappa}$ induces a decomposition, $M=M_1\oplus M_2$. We distinguish \( M_1 \) and \( M_2 \) by specifying that \( \dim M_1[F] = 1 \), which implies that \( \dim M_2[F] = n \).
		Consider the canonical filtration of $M$ in the sense of \cite[\S2.5]{MoonenWeyl}, 
		\[		M_{\emptyset, \bullet} 
		= \bigl(0 = M_{\emptyset, 0} \subsetneq M_{\emptyset, 1} \subsetneq \cdots \subsetneq M_{\emptyset, 2(n+1)} = M\bigr). 
		\]
		Intersecting $M_{\emptyset, \bullet}$ with $M_1$ gives the filtration 
		\begin{align} \label{CanfilInduced}
			M_{1, \bullet} = (0 = M_{1, 0} \subsetneq M_{1, 1} \subsetneq \cdots \subsetneq M_{1, n+1} = M_1). 
		\end{align}
		This filtration has the property that there exists a unique $b \in \{1, \ldots, n+1\}$ such that the function $\eta_1$ on the set $\{1, \ldots,\rg\}$ given by
		\(
		\eta_1(j) := \dim M_{1, j}[F]
		\), jumps at $j = b$, i.e., $\eta_1(b-1) = 0$ and $\eta_1(b) = 1$. Then $\underline{A}\in \mathcal{M}(n, 1)^a_{\Fpbar}$ for $a := b-1$.

        The above stratification enjoys the following properties: for each
\(a\in W_{\EO}^{(n,1)}\), the stratum \(\mathcal{M}(n,1)^a_{\Fpbar}\) is smooth
of dimension \(n\), and the Zariski closure of \(\mathcal{M}(n,1)^a_{\Fpbar}\)
is equal to the union of all \(\mathcal{M}(n,1)^b_{\Fpbar}\) with \(b\le a\).

\subsection{Standard basis for unitary Dieudonn\'e modules}\label{S:standardbasis}
Given a $\Bar{\BF}_p$-point $\underline{A}=(A, \iota, \lambda)$ of $\mathcal{M}(n, 1)_{\Fpbar}$, let $M=(M, \rF, \rV)$ be the standard (contravariant) Dieudonn\'e module attached to $A[p]$, and $M=M_1\oplus M_2$ the decomposition induced by the action of $\tilde{\kappa}$.  
Thanks to \cite[4.9]{MoonenWeyl}, there exists a basis $\{v_{ij} \mid 1 \leq i \leq 2, 1 \leq j \leq n+1\}$ of $M$ such that 
			\[
			\rF(v_{1,j}) =
			\begin{cases}
				v_{\gamma(1),j}, & \text{if } 1 \leq j \leq a, \\
				0, & \text{if } j = a+1, \\
				v_{\gamma(1),j-1}, & \text{if } j > a+1,
			\end{cases}
			\quad 
			\rV(v_{1,j}) =
			\begin{cases}
				0, & \text{if } j = 1, \\
				v_{\gamma(1),j-1}, & \text{if } 1 < j \leq \rg-a, \\
				v_{\gamma(1),j}, & \text{if } j > \rg-a.
			\end{cases}
			\]
            \[
			\rF(v_{2,j}) =
			\begin{cases}
				0, & \text{if } j \neq \rg-a,\\ 
				v_{\gamma(2),1}, & \text{if } j =\rg-a,
			\end{cases}
			\quad 
			\rV(v_{2,j}) = 
			\begin{cases}
				0, & \text{if } j \neq \rg, \\
				v_{\gamma(2), a+1}, & \text{if } j =\rg.
			\end{cases}
			\]            
Here \(\gamma(i)=i\) for \(i=1,2\) if \(p\) splits in \(\kk\), whereas
\(\gamma(1)=2\) and \(\gamma(2)=1\) if \(p\) is inert in \(\kk\).
We recall this structure on \((M,\rF,\rV)\) here, since it plays an implicit role in the EO stratification of \(\mathcal{M}(n,1)_{\Fpbar}\) and will also be used later.
		
		\subsection{EO Stratification of $\Sh_{G^{\mathbb{Q}}, \Fpbar}$}
		 Intersecting the EO stratification $\mathcal{M}(n, 1)_{\Fpbar}$ with $\Sh_{G^{\mathbb{Q}}, \Fpbar}$ induces:
		\(
		\Sh_{G^{\mathbb{Q}}, \Fpbar} = \bigsqcup_{a \in W_{\mathrm{EO}}^{(n, 1)}} \Sh_{G^{\mathbb{Q}}, \Fpbar}^a.
		\)
		This stratification coincides with the zip stratification in \cite{VW13}, implying that $W_{\mathrm{EO}}^{(n, 1)}$ corresponds to the minimal coset representatives ${}^{\mu}W$ for $(\mathcal{G}_{\kappa}^{\mathbb{Q}}, \mu)$. Here, $\mathcal{G}^{\mathbb{Q}} \coloneqq \mathrm{GU}(\Lambda_{\Zp})$ is the reductive model over $\mathbb{Z}_p$ for $G^{\mathbb{Q}}\otimes_{\Q}\Qp$, defined by the self-dual lattice $\Lambda$ of $W$, and $\mathcal{G}_{\kappa}^{\mathbb{Q}}$ is its base change to $\kappa$. Moreover, $\mu: \mathbb{G}_{m,\kappa} \to \mathcal{G}_{\kappa}^{\mathbb{Q}}$ denotes a representative in the conjugacy class $[\mu]_{\kappa}$.

        Below we make explicit the bijection between the indexing sets
\(W_{\mathrm{EO}}^{(n,1)}\) and \({}^{J_{\mu}}W\). Our discussion partially
follows \cite[\S3]{WoodingEO}. To ease notation, for an algebraic group \(G\) over \(\Fpbar\), we write \(G\)
for \(G(\Fpbar)\), and set \(\rg=n+1\). Let
\(\Lambda_\kappa=\Lambda_\kappa^1\oplus \Lambda_\kappa^0\) be the weight
decomposition of \(\Lambda_\kappa\) induced by the cocharacter \(\mu\). For
\(i\in\{1,2\}\), choose a \(\kappa\)-basis \(e_{i,1},\dots,e_{i,\rg}\) of
\(\Lambda_\kappa^i\) such that \(\Psi(e_{1,j},e_{2,j'})=0\) unless
\(j+j'=\rg+1\). These choices induce isomorphisms
\(\GL_\kappa(\Lambda_\kappa^i)\cong \GL_{\rg,\kappa}\), which we fix
henceforth. Then we have
\[
\mathcal{G}^{\mathbb{Q}}_{\kappa}(\Fpbar)
=\GU(\Lambda_{\kappa},\Psi)(\Fpbar)
\cong \{(X,aX^{\vee})\in \GL_{\rg}(\Fpbar)\times \GL_{\rg}(\Fpbar)\mid a\in \Fpbar^\times\},
\]
where \(X^\vee=(X^t)^{-1}\). Under these identifications, we may choose a
representative of the cocharacter \(\mu\) whose image in
\(\GL_{\rg,\kappa}\times \GL_{\rg,\kappa}\) is
\(t\mapsto (\diag(t,\dots,t,1),\diag(1,t,\dots,t))\). Let \((B_\rg,T_\rg)\) be the standard Borel pair of
\(\GL_{\rg,\Fpbar}\), where \(B_\rg\) is the upper triangular Borel subgroup
and \(T_\rg\) is the diagonal torus. We identify the Weyl group of \(\GL_\rg\)
with \(W_\rg=\mathfrak{S}_\rg\), whose simple reflections are
\(s_i=(i,i+1)\) for \(1\le i\le \rg-1\). Let \(B\subseteq
\mathcal{G}^{\mathbb{Q}}_{\Fpbar}\) be the subgroup consisting of elements of
the form \((X,aX^\vee)\) with \(a\in \Fpbar^\times\) and \(X\in B_\rg\), and
set \(T=(T_\rg\times T_\rg)\cap \mathcal{G}^{\mathbb{Q}}_{\Fpbar}\). Then
\((B,T)\) is a Borel pair of \(\mathcal{G}^{\mathbb{Q}}_{\Fpbar}\). Under this
choice, the Weyl group \(W\) of \(\mathcal{G}^{\mathbb{Q}}_{\Fpbar}\) is
identified with the subgroup of \(W_\rg\times W_\rg\) consisting of pairs
\((w_1,w_2)\) satisfying \(w_1=w_0w_2w_0\), where \(w_0\in W_\rg\) is the
longest element, and the corresponding simple reflections are the pairs
\((s_i,s_{\rg-i})\) for \(1\le i\le \rg-1\).

Next, let \(\mu_i\) be the cocharacter of \(\GL_\rg\) obtained by composing
\(\mu\) with the projection \(\mathrm{pr}_i:G_\kk\to \GL_\rg\) for \(i=1,2\). Put
\(J_n=S_\rg\setminus\{s_n\}\) and \(J_1=S_\rg\setminus\{s_1\}\). The types of
\(\mu_1\) and \(\mu_2\) are \(J_n\) and \(J_1\), respectively. We only
consider \({}^{J_n}W_\rg\), which consists of those \(w\in W_\rg\) such that
\(w^{-1}(1)<\cdots<w^{-1}(n)\); see \cite[Ex.~3.1.2]{WoodingEO}. Via
projection to the first factor, \({}^JW\) is identified with
\({}^{J_n}W_\rg\). We therefore obtain the following bijection:
\[
{}^JW \xrightarrow{1\text{-}1} {}^{J_n}W_\rg \xrightarrow{1\text{-}1} \{1,\ldots,\rg\} \xrightarrow{1\text{-}1} W_{\mathrm{EO}}^{(n,1)},
\]
where the first map is \((w_1,w_2)\mapsto w_1\), the second is
\(w\mapsto w^{-1}(1)\), and the last is \(b\mapsto b-1\). Hence an element
\(a\in W_{\mathrm{EO}}^{(n,1)}\) corresponds to the cycle
\(\delta_a=(1\,2\,\cdots\,a+1)\in {}^{J_n}W_\rg\).

\begin{remark}
Note that our \({}^JW\) here corresponds to the \({}^JW\) in
\cite{WoodingEO} after applying the permutation
\((w_1,w_2)\mapsto (w_2,w_1)\).
\end{remark}

		\subsection{EO Stratification of $\Sh_{\tilde{G}, \Fpbar}$}

       For a \(\kappa(v)=\Fp\)-algebra \(S\),
the groupoid \(\mathcal{M}(1,0)_{\kappa(v)}(S)\) parametrizes pairs
\((A_0,\iota)\), where \(A_0\) is an elliptic curve over \(S\) and
\(\iota:O_{\kk}\hookrightarrow \End_S(A_0)\) is an
\(O_{\kk}\)-action such that the induced action of
\(O_{\kk}/p\) on \(\mathrm{Lie}(A_0/S)\) is via the composition
\(
O_{\kk}/p \xrightarrow{\mathrm{pr}} \kappa(v)\to O_S. 
\)
For \(S=\Spec \Fpbar\), it follows from Deuring's theorem that the elliptic curves underlying the \(\Fpbar\)-points of \(\mathcal{M}(1,0)_{\kappa(v)}\) are all \emph{ordinary} when \(p\) splits in \(\kk\), and all \emph{supersingular} when \(p\) is inert in \(\kk\). Hence, the EO stratification for \(\mathcal{M}(1,0)_{\Fpbar}\) is trivial. 
 Thus, the EO stratification of $\mathcal{M}(n, 1)_{\Fpbar} \times \mathcal{M}(1, 0)_{\Fpbar} $ is:
		\[
		 \mathcal{M}(n, 1)_{\Fpbar}\times\mathcal{M}(1, 0)_{\Fpbar}  = \bigsqcup_{a \in W_{\mathrm{EO}}^{(n, 1)}}  \mathcal{M}(n, 1)_{\Fpbar}^a \times \mathcal{M}(1, 0)_{\Fpbar}.
		\]
		Intersecting with $\Sh_{\tilde{G}, \Fpbar}$ yields the EO stratification \(
		\Sh_{\tilde{G}, \Fpbar} = \bigsqcup_{a \in W_{\mathrm{EO}}^{(n, 1)}} \Sh_{\tilde{G}, \Fpbar}^a
		\) of $\Sh_{\tilde{G}, \Fpbar}$. Explicitly, an $\Fpbar$-point $(A_0, \iota_0, \lambda_0, A, \iota, \lambda, \bar{\eta})$ of $\Sh_{\tilde{G}, \Fpbar}$ lies in $\Sh_{\tilde{G}, \Fpbar}^a$ if and only if $(A, \iota, \lambda)$ lies in $\mathcal{M}(n, 1)_{\Fpbar}^a$.

		\begin{remark}\label{Rmk:EOCenIsogUnit}
Let \(\mathcal{G}\) and \(\tilde{\mathcal{G}}\) be the \(\mathbb{Z}_p\)-reductive
models of \(G\otimes \Qp\) and \(\tilde{G}\otimes \Qp\), respectively,
constructed from the lattices \(\Lambda^\flat\) and \(\Lambda_0\).
\begin{enumerate}
    \item By construction, the EO stratification on
    \(\Sh_{\tilde{G},\Fpbar}\) is the pullback of the EO stratification on
    \(\Sh_{G^{\mathbb{Q}},\Fpbar}\) along the morphism
    \(\Sh_{\tilde{G},\Fpbar}\to \Sh_{G^{\mathbb{Q}},\Fpbar}\). This also
    follows from Lemma~\ref{Lem:ZipIdentity}, since the projection
    \(\tilde{\mathcal{G}}_{\kappa}\to \mathcal{G}_{\kappa}^{\mathbb{Q}}\)
    induces an isomorphism on adjoint groups.

    \item Similarly, the natural projection
    \(\tilde{\mathcal{G}}_{\kappa}\to \mathcal{G}_{\kappa}\) induces an
    isomorphism on adjoint groups, so the corresponding sets \({}^J W\) are
    identified. Hence the EO stratification on \(\Sh_{\tilde{G},\Fpbar}\) is
    the pullback of the stratification on \(\Sh_{G,\Fpbar}\) along the
    morphism \(\Sh_{\tilde{G},\Fpbar}\to \Sh_{G,\Fpbar}\). Note that
    \(\Sh_{G,\Fpbar}\) is of abelian type, and its EO stratification is
    established in \cite{SZ22}.
\end{enumerate}
\end{remark}

        \subsection{EO strata under the embedding {$
		\Sh_{\widetilde{\GU}(n, 1), \Fpbar} \to \Sh_{\widetilde{\GU}(n+1,1), \Fpbar}$}}
		Specifically, given an EO stratum of $\Sh_{\tilde{H}, \Fpbar}$, we wish to determine the corresponding EO stratum of $\Sh_{\tilde{G}, \Fpbar}$ to which it maps. As discussed in Remark \ref{Rmk:EOCenIsogUnit}, this problem is equivalent to the same question for the morphism  
		\(
		\Sh_{H, \Fpbar} \to \Sh_{G, \Fpbar}.
		\)
		For notational convenience, in this section we increase the dimension of \(W\) by one and suppose that
		\(
		\dim_{\kk}W = n+2,
		\)
		so that \((W, \Psi)\) and \((W^{\flat}, \Psi|_{W^{\flat}})\) have signature \((n+1, 1)\) and \((n, 1)\), respectively. To make the dimensions transparent, below we write $\Sh_{\widetilde{\GU}(n, 1), \Fpbar}$ for $\Sh_{\tilde{H}, \Fpbar}$ and similarly we write $\Sh_{\GU(n, 1), \Fpbar}$ and $\Sh_{\mathrm{U}(n, 1), \Fpbar}$ for $\Sh_{H^{\Q}, \Fpbar}$ and $\Sh_{H, \Fpbar}$ respectively.	
		
		% \subsection{The proof of Theorem {\ref{Thm:MainThmUnit} ($p$ splits in $\kk$)}}
		\begin{theorem}[Split case]\label{KeyRelationUnitaryEOSplit} 
			Suppose that $p$ splits in $\kk$. Let $a\in W_{\mathrm{EO}}^{(n, 1)}$. Let $a'\in W_{\mathrm{EO}}^{(n+1, 1)}$ be such that the EO stratum $\Sh_{\widetilde{\GU}(n, 1), \Fpbar}^a$ is mapped into the stratum $\Sh_{\widetilde{\GU}(n+1, 1), \Fpbar}^{a'}$ under the morphism \(
			\Sh_{\widetilde{\GU}(n, 1), \Fpbar} \to \Sh_{\widetilde{\GU}(n+1, 1), \Fpbar}
			\).  Then we have the relation $ a' =a+1$.
		\end{theorem}
		\begin{proof}
		    In this split case, the unitary Shimura varieties
\(\Sh_{\widetilde{\GU}(n, 1),\Fpbar}\) and \(\Sh_{\widetilde{\GU}(n+1, 1),\Fpbar}\) are of \emph{fake}
unitary type: their EO, Newton, and \(p\)-rank stratifications coincide. More
precisely, for each \(0\le i\le n\), there is a unique stratum of dimension
\(i\). For \(i=n\), this is the stratum consisting of abelian varieties over
\(\Fpbar\) of \(p\)-rank \(n+1\); for \(0\le i\le n-1\), it is the stratum
consisting of abelian varieties of \(p\)-rank \(i\). For an \(\Fpbar\)-point
\((\underline{A_0}, \underline{A},\bar{\eta})\) of \(\Sh_{\widetilde{\GU}(n,1)}(\Fpbar)\), the
underlying elliptic curve \(A_0\) is \emph{ordinary}. Hence
\(f(A_0\times A)=f(A)+1\), where \(f\) denotes the usual \(p\)-rank of an
abelian variety over \(\Fpbar\).
		\end{proof}
		%\subsection{The proof of Theorem \ref{Thm:MainThmUnit} ($p$ inert in $\kk$)}
		\begin{theorem}[Inert case]\label{KeyRelationUnitaryEOInert} 
			Suppose that $p$ is inert in $\kk$. Let $a\in W_{\mathrm{EO}}^{(n, 1)}$. Let $a'\in W_{\mathrm{EO}}^{(n+1, 1)}$ be such that the EO stratum $\Sh_{\widetilde{\GU}(n, 1), \Fpbar}^a$ is mapped into the stratum $\Sh_{\widetilde{\GU}(n+1, 1), \Fpbar}^{a'}$ under the morphism \(
			\Sh_{\widetilde{\GU}(n, 1), \Fpbar} \to \Sh_{\widetilde{\GU}(n+1, 1), \Fpbar}
			\),  Then we have the relation 
			\begin{equation}\label{EqKeyRelationUnitaryEOInert}
				a' =
				\begin{cases}
					a,   & \text{if } a \leq n/2,\\[2pt]
					a+1, & \text{if } a > n/2.\\[2pt]
				\end{cases}
		\end{equation}	\end{theorem}
		
		  Below we give two slightly different proofs of this theorem, both based on the moduli interpretation of the embedding
         \(
       \Sh_{\widetilde{\GU}(n, 1),\Fpbar} \to \Sh_{\widetilde{\GU}(n+1, 1),\Fpbar}
       \)
       in \eqref{ModiofEmbedH2G}, and  ultimately rely on the classification of unitary BT-1's in \cite{MoonenWeyl}.
       Let $(A_0, \iota_0, \lambda_0, A, \iota, \lambda, \bar{\eta})$ be an $\Fpbar$-point of  \(
		\Sh_{\widetilde{\GU}(n, 1)}\), and let $(M, \rF,\rV)$ and $(N, \rF, \rV)$ be the Dieudonn\'e module of $A$ and $A_0$ respectively.
       Denote by \(\rF_i:M_i\to M_{\gamma(i)}\) and \(\rV_i:M_i\to M_{\gamma(i)}\) the restrictions of \(\rF\) and \(\rV\) to \(M_i\), respectively, where, by convention, \(\gamma(1)=2\) and \(\gamma(2)=1\).
       We consider the operator $\mathsf{T}=\rV^{-1}\rF:M\to M$ on $M$, and a similar operator \(\mathsf{T}_i: M_i \to M_i\) on $M_i$ obtained by replacing $\rV^{-1}\rF$ by  $\rV_i^{-1}\rF_i$. Let $M=M_1\oplus M_2$ be the decomposition induced by the action of  $\tilde{\kappa}$, and let $M_{i, \bullet}$ be the filtration as in \eqref{CanfilInduced}. 
		One then checks easily that for each of $i=1, 2$, we have 
		\(
		\mathsf{T}(M_{1,j}\oplus M_{2,j})\cap M_i= \mathsf{T}_i(M_{ij}).
		\)
		Set \[r=\min\{a+1, n+1-a\}, \text{\ and\ } s=\max\{a, n+1-a\}.\]
        \begin{lemma}[{cf.~\cite[\S~4.2.2]{ABFGGNEOGU}}] \label{Lem:OperatorT} The following holds: \begin{enumerate}
				\item  For \(1\le j\le \rg\), \(\mathsf T_1(M_{1,j})=M_{1,j'}\), where
                \(j'=j+1\) if \(j<r\), \(j'=j\) if \(j\in\{r,s\}\), and \(j'=j-1\) if \(j>s\).
				\item For $c\in \mathbb{Z}_{\geq 0}$, we have $\mathsf{T}_1^c(0)=M_{1, \min\{c, r\}}$ and $\mathsf{T}_1^c(M_1)=M_{1, \max\{c, s\}}$.
			\end{enumerate}
		\end{lemma}
		\begin{proof} Assume the filtration $M_{i,\bullet}$ of $M_i$ is generated by $\{v_{ij} \mid 1 \leq j \leq j'\}$. The lemma follows by direct verification, using the standard basis  $(M, \rF, \rV)$ recalled in \ref{S:standardbasis}. 
		\end{proof}
		
		\begin{proof}[{Proof of Theorem \ref{KeyRelationUnitaryEOInert}}] 
			Let $N=N_1\oplus N_2$ be the decomposition induced by the action of $\tilde{\kappa}$ and $N_{i, \bullet}$ the counterpart of $M_{i, \bullet}$ for $N$. Then, for dimension reasons, we have $N_{1,1}=N_1$ and $N_1[F]=0$. Set $L=M\oplus N$, the Dieudonn\'e module of $A\times A_0$, and let $L_{i, \bullet}$ be the counterparts of $M_{i, \bullet}$ for $L$. We claim that for each $c\gg 1$, we have
			\begin{align}\label{Goingdown}
				\mathsf{T}_1^{c}(L_{1,0})=M_{1, r}\oplus N_{1, 0}=M_{1,r},\quad  \mathsf{T}_1^{c}(L_1)= M_{1, s}\oplus N_{1,1}=M_{1,s}\oplus N. 
			\end{align}
			Indeed, by Lemma~\ref{Lem:OperatorT}, one has 
			\(
				\mathsf{T}^c(L_{1,0})
				= \mathsf{T}_1^c(M_{1,0})\oplus \mathsf{T}_1^c(N_{1,0})
				= (\mathsf{T}_1|_{M_1})^{c}(M_{1,0})\oplus N_{1,0}
				= M_{1, r}.
			\)
			The equality in \eqref{Goingdown} follows from a similar argument, noting that $\mathsf{T}_1(N_{1,1})=N_{1,1}$. The claim is proved. 
			
			Let $r'$ and $s'$ be the counterparts of $r$ and $s$ for $L=M\oplus N$. Then we see that $r'=\dim L_{1, r'}= r$ and $s'=\dim L_{1, s'}= s+1$. Unwinding the definitions of $r'$ and $s'$, one finds that
			\begin{align*}
				\min\{ a+1, n+1-a\} = \min \{ a'+1, n+2-a'\}, \quad 
				\min\{ a, n+1-a \}+1= \max \{ a', n+2-a'\}.
			\end{align*}
			It follows directly from these two equations that the desired relation in \eqref{EqKeyRelationUnitaryEOInert} holds.
		\end{proof} 
	
		\subsection*{Another proof of Theorem~\ref{KeyRelationUnitaryEOInert}}
		To ease notation, from now on we omit the subscript $\Fpbar$ from the notation $\Sh_{\widetilde{\GU}(n,1),\Fpbar}$, $\Sh_{\widetilde{\GU}(n+1,1),\Fpbar}$, and their EO strata.
		Following \cite[\S5.4]{BW06}, the EO stratification of $\Sh_{\widetilde{\GU}(n,1)}$ can also be described as
		\[
		\Sh_{\widetilde{\GU}(n,1)}=\bigcup_{1\leq \rho \leq n+1}\Sh_{\widetilde{\GU}(n,1),\rho},
		\]
		where $\Sh_{\widetilde{\GU}(n,1),\rho}$ is the locus of points $(\underline{A_0},\underline{A}, \bar{\eta})$ over $\Fpbar$ such that the associated unitary BT-1 corresponds to the unitary Dieudonn\'e space
		\(
		\underline{\bar{B}}(\rho)\oplus \underline{\bar{S}}^{\,n-\rho}.
		\)
		We refer to loc.\ cit.\ for the notation $\underline{\bar{B}}(\rho)$ and $\underline{\bar{S}}^{\,n-\rho}$. Note, however, that $\underline{\bar{S}}$ is isomorphic to the Dieudonn\'e space attached to the supersingular elliptic curve $A_0$. The EO stratum $\Sh_{\widetilde{\GU}(n,1),\rho}$ has codimension
		\begin{equation}\label{Eq:Codim}
			\mathrm{cod}(\rho):=\mathrm{codim}(\Sh_{\widetilde{\GU}(n,1),\rho},\Sh_{\widetilde{\GU}(n,1)})
			=
			\begin{cases}
				\rho/2-1,& \text{if } \rho \text{ is even},\\
				n+1-(\rho+1)/2,& \text{if } \rho \text{ is odd}.
			\end{cases}
		\end{equation}
		Since for each $0\leq i\leq n$ there is a unique EO stratum of $\Sh_{\widetilde{\GU}(n,1)}$ of dimension $i$, we have
		\(
		\Sh_{\widetilde{\GU}(n,1),\rho}=\Sh_{\widetilde{\GU}(n,1)}^{\,n-\mathrm{cod}(\rho)}.
		\)
		
		Under the embedding map $\Sh_{\widetilde{\GU}(n,1)}\to \Sh_{\widetilde{\GU}(n+1,1)}$, the tuple $(\underline{A_0},\underline{A}, \bar{\eta})$ is sent to $(\underline{A_0},\underline{A'},\bar{\eta}')$, where $A'=A\times A_0$. Since the BT-1 $A'[p]$ corresponds to the unitary Dieudonn\'e space
		\[
		\underline{\bar{B}}(\rho)\oplus \underline{\bar{S}}^{\,n+1-\rho},
		\]
		the EO stratum $\Sh_{\widetilde{\GU}(n,1),\rho}$ is mapped to the EO stratum $\Sh_{\widetilde{\GU}(n+1,1),\rho}$. Hence, by \eqref{Eq:Codim}, we obtain
		\begin{equation}\label{Eq:CodimReltion}
			\mathrm{codim}(\Sh_{\widetilde{\GU}(n+1,1),\rho},\Sh_{\widetilde{\GU}(n+1,1)})
			=
			\begin{cases}
				\mathrm{codim}(\Sh_{\widetilde{\GU}(n,1),\rho},\Sh_{\widetilde{\GU}(n,1)}),& \text{if } \rho \text{ is even},\\
				\mathrm{codim}(\Sh_{\widetilde{\GU}(n,1),\rho},\Sh_{\widetilde{\GU}(n,1)})+1,& \text{if } \rho \text{ is odd}.
			\end{cases}
		\end{equation}
		Note that the set $\left\{\Sh_{\widetilde{\GU}(n,1),\rho}\mid \rho \text{ is even}\right\}$ is equal to
		\begin{equation}\label{EvenOddEOlist}
			\begin{cases}
				\{\Sh_{\widetilde{\GU}(n,1),2},\Sh_{\widetilde{\GU}(n,1),4},\ldots,\Sh_{\widetilde{\GU}(n,1),n}\}
				=
				\{\Sh_{\widetilde{\GU}(n,1)}^{\,n},\Sh_{\widetilde{\GU}(n,1)}^{\,n-1},\ldots,\Sh_{\widetilde{\GU}(n,1)}^{\,n/2+1}\},
				& \text{if } n \text{ is even},\\[0.5em]
				\{\Sh_{\widetilde{\GU}(n,1),2},\Sh_{\widetilde{\GU}(n,1),4},\ldots,\Sh_{\widetilde{\GU}(n,1),n+1}\}
				=
				\{\Sh_{\widetilde{\GU}(n,1)}^{\,n},\Sh_{\widetilde{\GU}(n,1)}^{\,n-1},\ldots,\Sh_{\widetilde{\GU}(n,1)}^{\,(n+1)/2}\},
				& \text{if } n \text{ is odd}.
			\end{cases}
		\end{equation}
		For both odd and even $n$, this set consists exactly of those EO strata of $\Sh_{\widetilde{\GU}(n,1)}$ of dimension strictly greater than $n/2$. Therefore Theorem~\ref{KeyRelationUnitaryEOInert} follows from \eqref{Eq:CodimReltion} and \eqref{EvenOddEOlist}. \hfill\qedsymbol

        \subsection{Supersingular locus and $p$-ranks for {$\Sh_{\widetilde{\GU}(n+1,1), \Fpbar}$}}\label{S:PrankUnitShv}
        We continue to assume that $p$ is inert in $\kk$. The supersingular locus of the PEL-type Shimura variety $\Sh_{\widetilde{\GU}(n+1,1), \Fpbar}$ is the union, 
        \[
        \bigcup_{\rho\, \mathrm{odd}}\Sh_{\widetilde{\GU}(n+1,1), \rho}= \bigcup_{i\leq n/2}\Sh_{\widetilde{\GU}(n+1,1)}^i;
        \]
		and hence is of dimension $ \lfloor n/2 \rfloor $; see, for example, \cite[Lem.~ 4.2, \S5.4]{BW06}. 
        
      Since $\Sh_{\widetilde{\GU}(n+1,1),\Fpbar}$ is of PEL type, it is natural to consider its $p$-rank stratification. Here we adopt the convention that for each point $(\underline{A}_0, \underline{A}, \bar{\eta})\in \Sh_{\widetilde{\GU}(n+1,1),\rho}(\Fpbar)$, its $p$-rank means the $p$-rank of $A$. Clearly, for $n=0$, the whole space $\Sh_{\widetilde{\GU}(n+1,1),\Fpbar}$ is supersingular, and hence has $p$-rank $0$. Assume now that $n\ge 1$. 
      \begin{lemma}\label{Lem:PrankUnit}Assume that $p$ is inert in $\kk$.
          The $\mu$-ordinary locus $\Sh_{\widetilde{\GU}(n+1,1),\Fpbar}^{n}=\Sh_{\widetilde{\GU}(n+1,1),\Fpbar,2}$ has $p$-rank~$2$, whereas every other EO stratum has $p$-rank $0$.
      \end{lemma} 
      \begin{proof}
          Since the supersingular locus has $p$-rank $0$, it is enough to consider a stratum $\Sh_{\widetilde{\GU}(n+1,1),\rho}$ with $\rho=2m$ even.
Let $(\underline{A}_0, \underline{A},\bar{\eta})\in \Sh_{\widetilde{\GU}(n+1,1),\rho}(\Fpbar)$. By \cite[Lem.~3.3]{BW06}, the isocrystal attached to the Dieudonn\'e space $\underline{B}(\rho)\oplus \underline{S}^{\,n+1-\rho}$ has slopes
\[
\left\{
\left(\frac12-\frac{1}{2m}\right)^{2m},
\left(\frac12\right)^{2n+2-4m},
\left(\frac12+\frac{1}{2m}\right)^{2m}
\right\}.
\]
Therefore the slope $0$ occurs if and only if $m=1$, equivalently $\rho=2$, in which case it occurs with multiplicity $2$. Hence $\Sh_{\widetilde{\GU}(n+1,1),\rho}$ has $p$-rank~$2$ for $\rho=2$, and $p$-rank $0$ for all even $\rho>2$.
      \end{proof} 

\subsection{$a$-numbers of EO strata of $\Sh_{\widetilde{\GU}(n+1,1), \Fpbar}$}\label{S:AnumberUnitShiv}

Note that, if we adopt the convention that for an $\Fpbar$-point $(\underline{A}_0, \underline{A}, \bar{\eta})$ of $\Sh_{\widetilde{\GU}(n+1,1),\Fpbar}$, its $a$-number is defined to be $a(A)$, then the $a$-number of an EO stratum of $\Sh_{\widetilde{\GU}(n+1,1),\Fpbar}$ is well-defined.

     \begin{lemma}\label{Lem:AnumberUnit}
The superspecial EO stratum $\Sh_{\widetilde{\GU}(n+1,1),\Fpbar}^{0}$ has $a$-number $n+1$,
while every other EO stratum has $a$-number $n-1$.
\end{lemma}

\begin{proof}
The case of the superspecial EO stratum $\Sh_{\widetilde{\GU}(n+1,1),\Fpbar}^{0}$ is clear by
definition: for each $\Fpbar$-point of $\Sh_{\widetilde{\GU}(n+1,1),\Fpbar}^{0}$, the underlying
abelian variety is superspecial, and hence has $a$-number $n+1$. Now let $\Sh_{\widetilde{\GU}(n+1,1),\Fpbar}^{r}$ be an EO stratum with $r\neq 0$, and let
$(M,\rF,\rV)$ be its standard contravariant Dieudonn\'e module, with standard basis
as described in the proof of Lemma~\ref{Lem:OperatorT}. Let
$(M_0,\rF_0,\rV_0)$ be the contravariant Dieudonn\'e module of the finite group
scheme $\alpha_p$ over $\Fpbar$. Then
\(
M_0\cong \Fpbar\) and \(\rF_0=\rV_0=0.
\)
By contravariant Dieudonn\'e theory, we have
\[a(A)
=
\dim_{\Fpbar}\Hom_{\Fpbar}(\alpha_p,A)
=
\dim_{\Fpbar}\Hom_{\mathrm{DM}}(M,M_0)=\dim_{\Fpbar} M/(\rF M+\rV M).
\]
It then follows directly from the explicit description of $\rF$ and $\rV$ on the
standard basis given in the proof of Lemma~\ref{Lem:OperatorT} that
\[
\dim_{\Fpbar} M/(\rF M+\rV M)=n-1.
\]
Therefore every EO stratum $\Sh_{\widetilde{\GU}(n+1,1),\Fpbar}^{r}$ with $r\neq 0$ has
$a$-number $n-1$.
\end{proof}

       \subsection{The Illustrative Diagram (unitary)} \label{S:UnitDiagram}
		The diagram below schematically illustrates the inert case of Theorem~\ref{KeyRelationUnitaryEOInert}, together with the other strata. It is read in the same way as the diagram for the orthogonal case in \S~\ref{S:DiagOrthog}. The middle vertical line is the \(p\)-rank-zero line: strata on or to its right are supersingular, while those to its left have \(p\)-rank \(>0\).
		\begin{align*}
			\scalebox{0.8}{ % Adjust the scale factor to fit the diagram on the page
				\xymatrix@R=1.1cm@C=1.1cm{ 
					&&&&\boxed{0}\ar@{--}[dd]\ar[dl] &&&&& \mathrm{U}(0,1) \\
					&&& \boxed{1} \ar[dl] \ar@{-->}[rr]& & \boxed{0} \ar[dr]& & && \mathrm{U}(1,1) \\
					&&\boxed{2} \ar[dl] \ar@{-->}[rr]& & \boxed{1} \ar@{-->}[rr]\ar@{-->}[dd]\ar[dl] & & \boxed{0}\ar[dr] & & &\mathrm{U}(2,1) \\
					&\boxed{3} \ar[dl] \ar@{-->}[rr]& & \boxed{2}\ar[rr] \ar[dl] & & \boxed{1}\ar@{-->}[rr] \ar[dr] & & \boxed{0} \ar[dr] &&	\mathrm{U}(3,1) \\
					\boxed{4}\ar@{-->}[rr]  & & \boxed{3}\ar@{-->}[rr]  & & \boxed{2}\ar@{-->}[rr] & & \boxed{1}\ar@{-->}[rr]  & & \boxed{0} &\mathrm{U}(4,1) }}
		\end{align*}
		%\begin{lemma}
		%	The map $w \mapsto  w_{\mu}w w_0$ defines an order-reversing involution of ${}^\mu W \to {}^\mu W$. Consequently we have an identification of the set of $E_{\mu}$-orbits $\{ E_{\mu} \cdot w\mid w\in {}^JW\}$ and $\{E_{\mu}\cdot (w_{0, \mu}ww_0)\mid w\in {}^JW\}$. 
		%\end{lemma}

	\end{document}